\documentclass[nospthms, 10pt, final]{svjour3_arXiv}%
\usepackage{amsmath}
\usepackage{amsfonts}
\usepackage{amssymb}
\usepackage{graphicx}
\usepackage[a4paper, headheight = 2.0cm, tmargin=1.5cm, bmargin = 1.5cm,
hmargin={1.8cm,1.8cm}]{geometry}%

\providecommand{\U}[1]{\protect\rule{.1in}{.1in}}
\newenvironment{proof}[1][Proof]{\noindent\textbf{#1.} }{\ \rule{0.5em}{0.5em}}
\def\arraystretch{1.3}
\allowdisplaybreaks

\usepackage[unicode]{hyperref}%

\usepackage[hyperref]{xcolor}

\hypersetup{colorlinks = true}

\usepackage{enumitem}

\usepackage{float}

\usepackage{wallpaper}
\journalname{}

\usepackage[compress]{natbib}

\newtheorem{example}{Example}[section]
\newtheorem{remark}{Remark}[section]

\newtheorem{proposition}{Proposition}[section]

\begin{document}
	
	\LRCornerWallPaper{0.2}{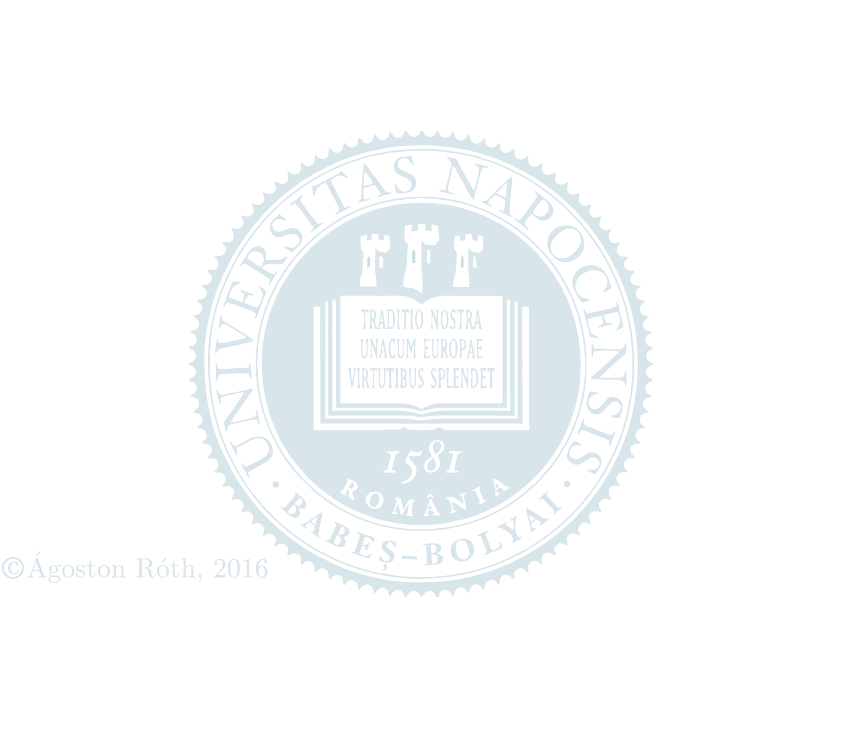}

\hypersetup{
	colorlinks = true, 
	citecolor = blue, 
	linkcolor = red, 
	urlcolor = blue}

\renewcommand*{\thefootnote}{\fnsymbol{footnote}}

\title{Nielson-type transfinite triangular interpolants\\by means of quadratic energy functional optimizations\thanks{In memoriam: Gerald Farin (March 20, 1953 -- January 14, 2016).}}

\author{\'{A}goston R\'{o}th}

\institute{
	\'A. R\'oth \at Department of Mathematics and Computer Science, Babe\c{s}--Bolyai University, RO--400084 Cluj-Napoca, Romania \\
	Tel.: +40-264-405300\\
	Fax:  +40-264-591906\\
	\email{agoston\_roth@yahoo.com}
}

%\date{}
\date{Submitted to arXiv on April 6, 2016}

\maketitle
%\vspace{-2.85cm}
%\begin{small}
\begin{abstract}
	We generalize the transfinite triangular interpolant of \citep{Nielson1987} in order to generate visually smooth (not necessarily polynomial) local interpolating quasi-optimal triangular spline surfaces. Given as input a triangular mesh stored in a half-edge data structure, at first we produce a local interpolating network of curves by optimizing quadratic energy functionals described along the arcs as weighted combinations of squared length variations of first and higher order derivatives, then by optimizing weighted combinations of first and higher order quadratic thin-plate-spline-like energies we locally interpolate each curvilinear face of the previous curve network with triangular patches that are usually only $C^0$ continuous along their common boundaries. In a following step, these local interpolating optimal triangular surface patches are used to construct quasi-optimal continuous vector fields of averaged unit normals along the joints, and finally we extend the $G^1$ continuous transfinite triangular interpolation scheme of \citep{Nielson1987} by imposing further optimality constraints concerning the isoparametric lines of those groups of three side-vertex interpolants that have to be convexly blended in order to generate the final visually smooth local interpolating quasi-optimal triangular spline surface. While we describe the problem in a general context, we present examples in special polynomial, trigonometric, hyperbolic and algebraic-trigonometric vector spaces of functions that may be useful both in computer-aided geometric design and in computer graphics.
\end{abstract}
%\end{small}


%\normalsize
\renewcommand*{\thefootnote}{\arabic{footnote}}

\section{Introduction}

Spline-like surfaces that consist of geometrically continuous or visually smooth triangular interpolants (i.e., interpolating triangular surface patches that have continuous tangent plane along their common boundary curves) are important in approximation theory, in computer-aided geometric design and in computer graphics as well. Without providing an exhaustive survey, we cite some related publications:
\begin{itemize}[noitemsep]
	\item
	in \citep{Nielson1987} each surface segment is defined over a triangle such that it matches a group of transfinite data formed by three boundary curves and associated vector fields of normals;
	
	\item
	in \citep{Loop1994} the author proposes a $G^1$ triangular spline surface of arbitrary topological type that consists of sextic triangular B\'ezier patches;
	
	\item
	in \citep{WaltonMeek1996} the authors fit a given triangular network of cubic B\'ezier curves with a composite $G^1$ continuous surface that consists of rational polynomial triangular patches, by constructing at first tangent ribbons along each boundary curve then by defining surface patches with cross-boundary directional derivatives that lie in common planes along the shared joints;
	
	\item
	in \citep{VlachosPetersBoydMitchell2001} the authors use the point-normal interpolation method of \citep{Pieper1987} in the context of so-called triangular PN patches that are used to improve the visual quality of existing triangle-based art in real-time entertainment (such as computer games), by replacing flat triangles with cubic B\'ezier patches that ensure a quadratic normal vector variation for Gouraud shading;
	
	\item
	in \citep{TongKim2009} the authors present a polynomial method for the approximation of implicit surfaces by $G^1$ triangular spline surfaces that is capable of interpolating positions, normals and normal curvatures at the vertices of a triangular mesh which is a piecewise linear homeomorphic approximation of the given implicit surface (the $G^1$ continuity constraints are ensured by solving equality-constrained least squares fitting problems);
	
	\item
	in \citep{FarinHansford2012} the authors also rely on \citep{Pieper1987} in order to build patch boundaries as B\'ezier curves, then they propose special $G^1$ smoothness conditions for rectangular and triangular Gregory patches \citep{Gregory1974, ChiyokuraKimura1984} that can be incorporated  into a surface fitting algorithm, by estimating tangent ribbons along the boundary curves and enforcing $G^1$ continuity across interior boundary curves by means of underdetermined linear systems.
\end{itemize}	
	
Our approach originates in the transfinite triangular interpolant of \citep{Nielson1987} and relies on quadratic energy functional optimizations in order to generate visually smooth quasi-optimal triangular spline surfaces (the similarities and main differences will be detailed in Section \ref{sec:generalized_G1_Nielson_patches}). 

Consider the triangular mesh $\mathcal{M}=\left(  \mathcal{V},\mathcal{F}%
,\mathcal{E}\right)  $ that consists of the unique vertices $\mathcal{V}%
=\left\{  \mathbf{p}_{i}\right\}  _{i=1}^{n_{v}}$ and counterclockwise
oriented triangular faces (i.e., triplets of vertex nodes)
\[
\mathcal{F}=\left\{f_r\right\}_{r=1}^{n_f}=\left\{  \left(  i,j,k\right)  :\left(  i,j\right)  ,\left(
j,k\right)  ,\left(  k,i\right)  \in\mathcal{E}\right\}  ,
\]
where $\mathcal{E}$ denotes the set of oriented (half-)edges. As an example, Fig.\
\ref{fig:mesh_normals_tangents}(\emph{a}) illustrates such a triangular mesh.%

\begin{figure}
[H]
\begin{center}
\includegraphics{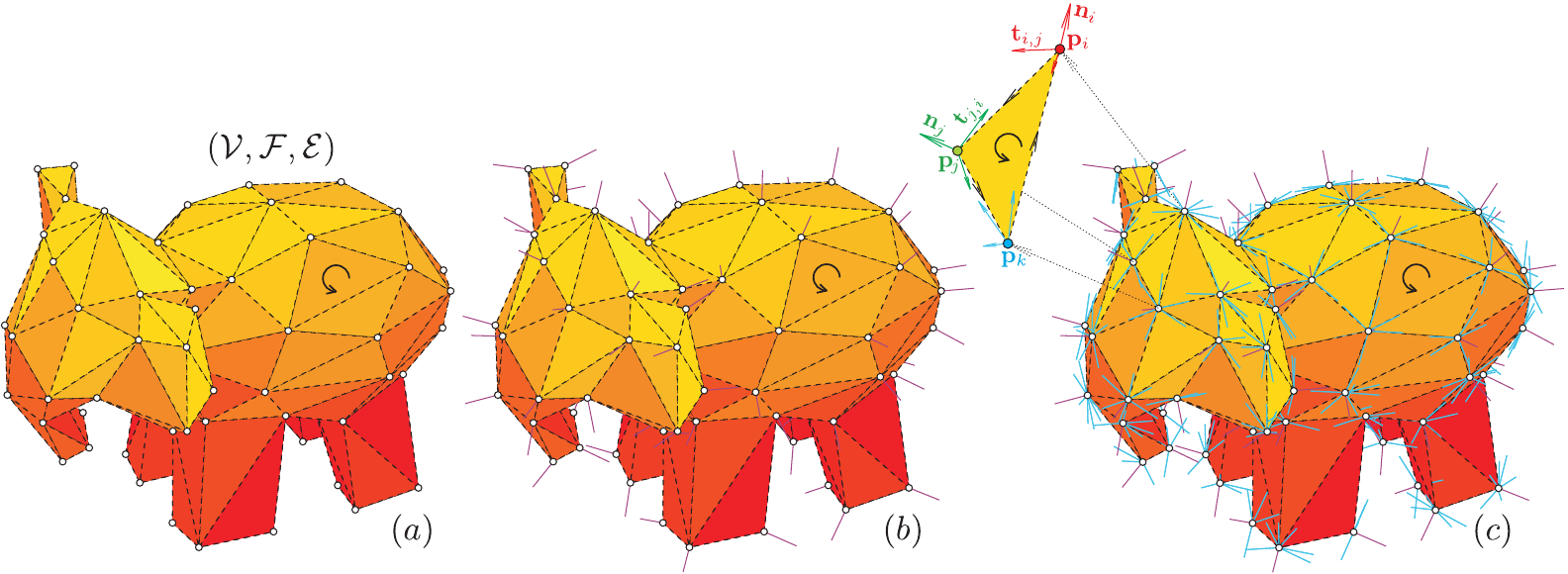}%
\caption{(\emph{a}) A triangular mesh $\left(  \mathcal{V},\mathcal{F}%
,\mathcal{E}\right)  $ consisting of unique vertices and counterclockwise
oriented faces. (\emph{b})--(\emph{c}) Angle-weighted averaged unit normal vectors and approximated
unit tangent vectors translated to the corresponding mesh vertices. ($n_v = 89$, $n_f=174$)}%
\label{fig:mesh_normals_tangents}%
\end{center}
\end{figure}

Based on the connectivity information stored in $\mathcal{F}$ one is able to
determine the angle-weighted \citep{ThurrnerWuthrich1998} averaged unit normal vectors $\mathcal{N}=\left\{
\mathbf{n}_{i}\right\}  _{i=1}^{n_{v}}$ that can be associated with the
corresponding vertices of $\mathcal{V}$ as it is shown in Fig.\
\ref{fig:mesh_normals_tangents}(\emph{b}). Now, consider an arbitrarily selected directed edge $\left(  i,j\right)  \in\mathcal{E}$ and based on
vectors%
\[%
\begin{array}
[c]{cclcccl}%
\mathbf{f}_{i} & = & -\mathbf{n}_{i}, &  & \mathbf{f}_{j} & = & -\mathbf{n}%
_{j},\\
&  &  &  &  &  & \\
\mathbf{b}_{i,j} & = & \dfrac{\left(  \mathbf{p}_{j}-\mathbf{p}_{i}\right)
	\times\mathbf{f}_{i}}{\left\Vert \left(  \mathbf{p}_{j}-\mathbf{p}_{i}\right)
	\times\mathbf{f}_{i}\right\Vert }, &  & \mathbf{b}_{j,i} & = & \dfrac{\left(
	\mathbf{p}_{i}-\mathbf{p}_{j}\right)  \times\mathbf{f}_{j}}{\left\Vert \left(
	\mathbf{p}_{i}-\mathbf{p}_{j}\right)  \times\mathbf{f}_{j}\right\Vert }%
\end{array}
\]
determine the unit tangents $$\mathbf{t}_{i,j}=\mathbf{f}_{i}%
\times\mathbf{b}_{i,j}~~~\text{  and  }~~~ \mathbf{t}_{j,i}=\mathbf{f}_{j}\times \mathbf{b}_{j,i}$$ that can be associated with vertices $\mathbf{p}_{i}$ and
$\mathbf{p}_{j}$, respectively. Fig.\ \ref{fig:mesh_normals_tangents}(\emph{c})
shows all these approximated unit tangent vectors.
Observe that the normal vector $\mathbf{n}_{i}$ is orthogonal to all tangent
vectors emanating from the vertex $\mathbf{p}_{i}$, i.e., all tangent vectors
emanating from a given vertex lie in the same approximated tangent plane. The rest of the manuscript is organized as follows.

	By means of the unique non-negative normalized B-bases of different types of reflection invariant extended Chebyshev vector spaces that also comprise the constants, in Section \ref{sec:proper_univariate_basis_functions} we select proper univariate basis functions in order to ensure more optimal shape preserving properties.
	
	In Section \ref{sec:optimal_arcs} we build a network of optimal curves that locally interpolate the vertices and associated tangent vectors along each edge of the given mesh by means of (not necessarily polynomial) arcs that minimize the weighted combinations of squared length variations of their first and higher oder derivatives (i.e., by means of explicit closed formulas, our objective is to define and locally minimize certain quadratic energy functionals that always have unique global optimum points that indirectly also influence the length, the curvature variation or other higher order intrinsic properties of the arcs of the constructed network).
	
	Section \ref{sec:C0_optimal_patches} gives explicit closed formulas to locally interpolate all curvilinear triangular faces formed by the arcs of the previously constructed curve network by using optimal (not necessarily polynomial) triangular surface patches that are only $C^0$ continuous along their common boundaries and which minimize the weighted sum of first and higher order quadratic thin-plate-spline-like energies (the considered energy functionals will always have unique optimum points that indirectly also influence the surface area, the curvature variation or any higher order intrinsic properties of the constructed patches). These triangular interpolants should be described by means of barycentric coordinate-dependent constrained trivariate basis functions that along the boundaries are compatible with (i.e., degenerate to) the univariate basis functions used for the construction of the arcs of the optimal curve network.
	
	In Section \ref{sec:generalized_G1_Nielson_patches}, at first, we use the connectivity information stored in $\mathcal{F}$ in order to define continuous vector fields of averaged unit normals along the common boundary curves of the previously generated local interpolating optimal $C^0$ triangular patches, then we generalize the convexly blended side-vertex method of \citep{Nielson1987} in order to produce $G^1$ continuous (not necessarily polynomial) quasi-optimal triangular patches that interpolate the previously calculated local interpolating optimal arcs and also match the   vector fields of averaged unit normals of the neighboring $C^0$ patches.
	
	Our final remarks are included in Section \ref{sec:final_remarks}. In order to ease and speed up possible reimplementations we close the manuscript with Appendices \ref{app:univariate_integrals} and \ref{app:double_integrals} that list closed formulas for certain univariate and double integrals, respectively, on which the proposed method relies. The steps of our triangular spline surface modeling tool are briefly outlined in Fig.\ \ref{fig:process}.

\begin{figure}[H]
	\centering
	\includegraphics{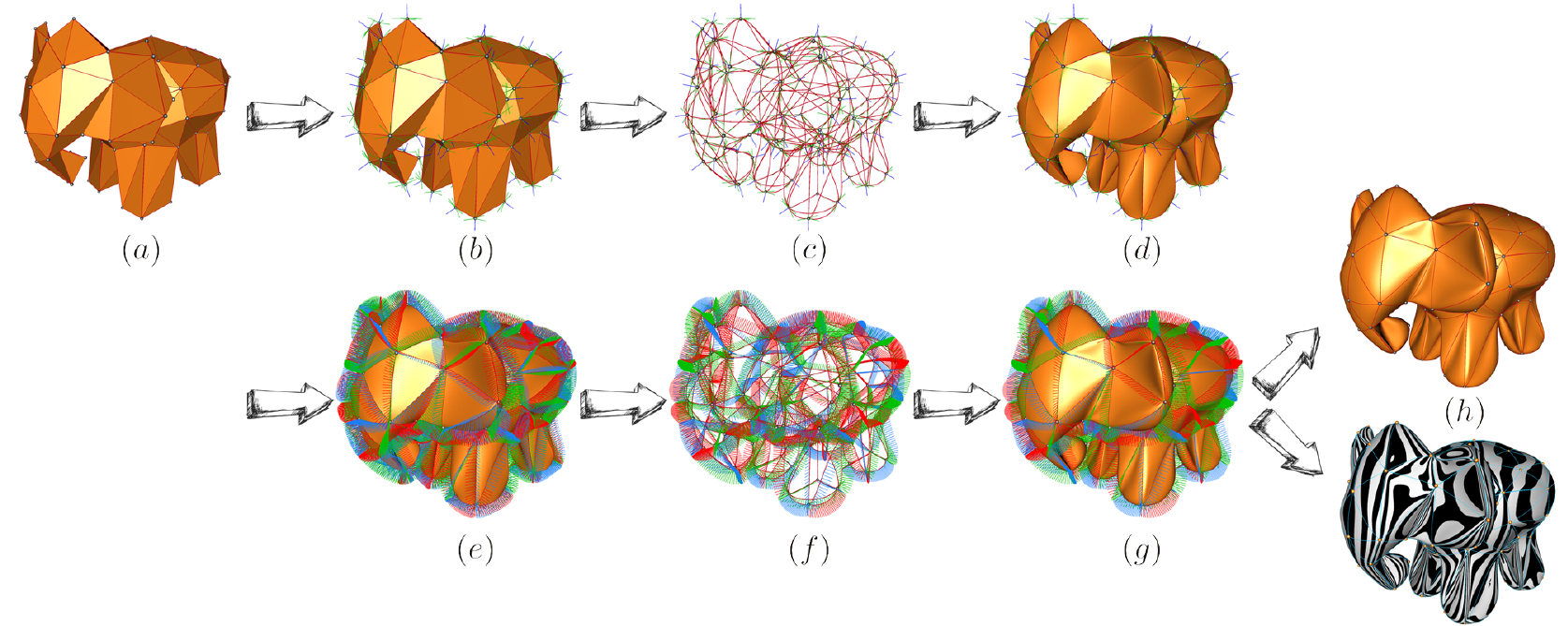}
	\caption{(\textit{a}) A triangular mesh stored in a half-edge data structure. (\textit{b}) Angle-weighted averaged unit normals and approximated unit tangents associated with the vertices. (\textit{c}) Local interpolating piecewise optimal curve network. (\textit{d}) Local interpolating piecewise optimal triangular surface patches that are only $C^0$ continuous along their common boundaries. (\textit{e}) Different vector fields of unit normals along the boundaries. (\textit{f}) Continuous quasi-optimal vector fields of averaged unit normals along the joints. (\textit{g})--(\textit{h}) Material and reflection line based rendering of generalized $G^1$ continuous triangular Nielson-type interpolants with quasi-optimal isoparametric lines.}
	\label{fig:process}
\end{figure}

\vspace{-0.5cm}
\section{Selecting proper basis functions}\label{sec:proper_univariate_basis_functions}
Consider the endpoints of the directed edge $\left(i,j\right)\in\mathcal{E}$ and their associated unit tangent vectors. In order to describe the points and higher order derivatives of a smooth curve as the varying linear combination of the control points $\mathbf{p}_i$, $\mathbf{p}_i+\lambda_{i,j}\mathbf{t}_{i,j}$, $\mathbf{p}_j+\lambda_{j,i}\mathbf{t}_{j,i}$, $\mathbf{p}_j$ (where scaling factors $\lambda_{i,j}$ and $\lambda_{j,i}$ are unknown at the moment), one has to define smooth basis functions that also ensure several shape preserving properties. Let%
\begin{equation}
\mathcal{B}=\left\{  B_{3,0}\left(  x\right)  ,B_{3,1}\left(  x\right)
,B_{3,2}\left(  x\right)  ,B_{3,3}\left(  x\right)  :x\in\left[
0,\beta\right]  \right\},~\beta > 0
\label{eq:system}
\end{equation}
be a system of sufficiently smooth non-negative normalized symmetric basis functions that also ensure endpoint interpolation, i.e.,%
\begin{align}
B_{3,k}&\in C^{\rho}\left(\left[0,\beta\right]\right),&~\rho \geq 1,~k=0,1,2,3,\label{eq:smoothness}\\
B_{3,k}\left(  x\right)   &  \geq0,&~\forall x \in \left[0,\beta\right],~k=0,1,2,3,\label{eq:nonnegativity}
\\
\sum_{k=0}^{3}B_{3,k}\left(  x\right)   &  \equiv1,&~\forall x\in\left[
0,\beta\right] ,\label{eq:partition_of_unity}
\\
B_{3,0}\left(0\right) &= 1,& ~B_{3,1}\left(0\right) = B_{3,2}\left(0\right) = B_{3,3}\left(0\right) = 0,
\\
B_{3,3}\left(\beta\right) &= 1,& ~B_{3,0}\left(\beta\right) = B_{3,1}\left(\beta\right) = B_{3,2}\left(\beta\right) = 0,
%\\
%B_{3,0}\left(\beta\right) &= B_{3,1}\left(\beta\right) = B_{3,2}\left(\beta\right) = 0,~B_{3,3}\left(\beta\right) = 1, 
\\
B_{3,k}\left(  x\right)   &  =B_{3,3-k}\left(  \beta-x\right)  ,&~\forall
x\in\left[  0,\beta\right]  ,~k=0,1.\label{eq:symmetry}
\end{align}

As we will see in the forthcoming sections, conditions (\ref{eq:smoothness})--(\ref{eq:symmetry}) form a part of the minimal requirements of the proposed curve and surface modeling tools. Naturally, there are infinitely many bases that fulfill conditions (\ref{eq:smoothness})--(\ref{eq:symmetry}), however, in order to ensure additional ideal shape preserving properties, we advise to define the system (\ref{eq:system}) by means of unique normalized B-basis functions of classes of reflection invariant (mixed) extended Chebyshev (EC) spaces as it is illustrated by several examples at the end of the current section. The following parts recall the notion of EC spaces and motivate the application of non-negative normalized B-basis functions.

Let $\eta\geq 3$ be a fixed integer and consider the EC system%
\begin{equation}
\mathcal{S}_{\eta}^{\beta}=\left\{  S
_{\eta,r}\left(  x\right)  :x\in\left[  0,\beta\right]  \right\}  _{r=0}%
^{\eta},~S_{\eta,0}	\equiv 1,\label{eq:ordinary_basis}%
\end{equation}
of basis functions in $C^{\eta}\left(  \left[  0,\beta\right]  \right)  $,
i.e., by definition \citep{KarlinStudden1966}, for any integer $0\leq r\leq \eta$, any strictly increasing
sequence of knot values $0\leq x_{0}<x_{1}<\ldots<x_{r}\leq\beta$, any
positive integers (called multiplicities) $\left\{  m_{k}\right\}  _{k=0}^{r}$
such that $\sum_{k=0}^{r}m_{k}=\eta+1$, and any real numbers $\left\{
y_{k,\ell}\right\}  _{k=0,~\ell=0}^{r,~m_{k}-1}$ there always exists a
unique function
\begin{equation}
S:=\sum_{r=0}^{\eta}\sigma_{\eta,r}S_{\eta,r}\in\mathbb{S}_{\eta}^{\beta}:=
%\left\langle\mathcal{S}_{\eta}^{\beta}\right\rangle:=
\operatorname{span}\mathcal{S}_{\eta}^{\beta},~\sigma_{\eta,r}\in%
%TCIMACRO{\U{211d} }%
%BeginExpansion
\mathbb{R}
%EndExpansion
,~r=0,1,\ldots,\eta \label{eq:unique_solution}%
\end{equation}
that satisfies the conditions of the Hermite interpolation problem%
\begin{equation}
\left.\frac{\mathrm{d^{\ell}}}{\mathrm{d}x^{\ell}}S\left(x\right)\right|_{x=x_k}=:S^{\left(  \ell\right)  }\left(  x_{k}\right)  =y_{k,\ell},~\ell
=0,1,\ldots,m_{k}-1,~k=0,1,\ldots,r. \label{eq:Hermite_interpolation_problem}%
\end{equation}
In what follows, we assume that the sign-regular determinant of the coefficient
matrix of the linear system (\ref{eq:Hermite_interpolation_problem}) of
equations is strictly positive for any permissible parameter settings introduced above.
Under these circumstances, the vector space $\mathbb{S}_{\eta}^{\beta}$ of functions is
called an EC space of dimension $\eta+1$. In terms of zeros, this
definition means that any non-zero element of $\mathbb{S}_{\eta}^{\beta}$
vanishes at most $\eta$ times in the interval $\left[0,\beta\right]  $. Such spaces and their corresponding spline counterparts have been widely studied, consider e.g.\ articles \citep{PottmannWagner1994,Mazure1999, Mazure2001,MainarPenaSanchez2001,LuWangYang2002,CarnicerMainarPena2004,MainarPena2004,CarnicerMainarPena2007,MainarPena2010,Roth2015} and many other references therein.

%Hereafter we will also refer to $\mathcal{S}_{\eta}^{\beta}$ as the ordinary basis of $\mathbb{S}_{\eta}^{\beta}$. 
Using \citep[Theorem
5.1]{CarnicerPena1995} and \citep{CarnicerMainarPena2004}, it follows that the vector space $\mathbb{S}%
_{\eta}^{\beta}$ also has a strictly totally positive basis for appropriately fixed values of the parameter $\beta$, i.e., a basis
such that all minors of all its collocation matrices are strictly positive. Since the constant function $1\equiv S_{\eta,0}\in\mathbb{S}_{\eta}%
^{\beta}$, the aforementioned strictly positive basis is normalizable,
therefore the vector space $\mathbb{S}_{\eta}^{\beta}$ also has a unique non-negative normalized B-basis%
\begin{equation}
\widetilde {\mathcal{B}}_{\eta}^{\beta}=\left\{  \widetilde{B}_{\eta,r}\left(  x\right)  :x\in\left[
0,\beta\right]  \right\}  _{r=0}^{\eta} \label{eq:B-basis}%
\end{equation}
that besides the identity%
\begin{equation}
\sum_{r=0}^{\eta}\widetilde{B}_{\eta,r}\left(  x\right)  \equiv 1,~\forall x\in\left[
0,\beta\right]  \label{eq:partition_of_unity_B_basis}%
\end{equation}
also fulfills the properties%
\begin{align}
\widetilde{B}_{\eta,0}\left( 0 \right)   &  =\widetilde{B}_{\eta,\eta}\left(  \beta\right)
=1,\label{eq:endpoint_interpolation}\\
\widetilde{B}_{\eta,r}^{\left(  j\right)  }\left(  0\right)   &  =0,~j=0,\ldots
,r-1,~\widetilde{B}_{\eta,r}^{\left(  r\right)  }\left(  0\right)
>0,\label{eq:Hermite_conditions_0}\\
\widetilde{B}_{\eta,r}^{\left(  j\right)  }\left(  \beta\right)   &  =0,~j=0,1,\ldots
,\eta-1-r,~\left(  -1\right)  ^{\eta-r}\widetilde{B}_{\eta,r}^{\left(  \eta-r\right)  }\left(
\beta\right)  >0 \label{eq:Hermite_conditions_alpha}%
\end{align}
conform \citep[Theorem 5.1]{CarnicerPena1995} and \citep[Equation (3.6)]%
{Mazure1999}. Among such EC spaces there are ones (see e.g.\ \citep{MainarPena2010}) that can be formed by all solutions of those linear homogeneous differential equations of order $\eta + 1$ the coefficients of which are constants and the characteristic polynomial of which is an either even or odd function that also admits $0$ as one of its (presumably higher order) zeros. Under these
conditions, $1\in\mathbb{S}_{\eta}^{\beta}$, moreover the space $\mathbb{S}%
_{\eta}^{\beta}$ is also invariant under reflections and consequently under
translations as well, i.e., for any function $S\in\mathbb{S}_{\eta}^{\beta}$ and
fixed scalar $\tau\in%
%TCIMACRO{\U{211d} }%
%BeginExpansion
\mathbb{R}
%EndExpansion
$ the functions $P_{\tau}\left(  x\right)  :=S\left(  \tau-x\right)  $ and
$Q_{\tau}\left(  x\right)  :=S\left(  x-\tau\right)  $ also belong to
$\mathbb{S}_{\eta}^{\beta}$. Therefore, the unique normalized B-basis of the members of this class of EC spaces also fulfills the symmetry
\begin{equation}
\label{eq:B-symmetry}
\widetilde{B}_{\eta,r}\left(  x\right)  =\widetilde{B}_{\eta,\eta-r}\left(  \beta-x\right),~\forall x \in \left[0, \beta\right]  ,~r=0,1,\ldots
,\left\lfloor \frac{\eta}{2}\right\rfloor.
\end{equation}

Compared with the traditionally used cubic or higher degree Hermite basis functions that would easily fulfill the required endpoint interpolation conditions, the unique normalized B-bases of such vector spaces ensure more optimal shape preserving properties (like closure for the affine transformations of the control points, convex hull, variation diminishing, monotonicity preserving, hodograph and length diminishing, symmetry with respect to reversing the order of control points), important evaluation or subdivision algorithms and useful shape (or tension) parameters.

In cases $\eta = 3$ and $\eta > 3$, the required basis functions of the system (\ref{eq:system}) can be obtained from the normalized B-basis (\ref{eq:B-basis}) either by direct index correspondence or by creating four linearly independent linear combinations of normalized B-basis functions such that the obtained expressions do not violate conditions (\ref{eq:nonnegativity})--(\ref{eq:symmetry}), respectively.

\begin{example}[Cubic and higher degree Bernstein polynomials]\label{exmp:Bernstein}
An immediate choice of the required basis functions would be the cubic case of the Bernstein polynomials
\[
\widetilde{\mathcal{B}}^1_{\eta} = 
\left\{
\widetilde{B}_{\eta,r}\left(x\right)
=
\binom{\eta}{r}x^r \left(1-x\right)^{\eta-r}:x\in\left[0,1\right]
\right\}_{r=0}^{\eta},~\eta \geq 3
\]
that form the non-negative normalized B-basis \citep{Carnicer1993} of the EC space of polynomials of degree at most $\eta$, i.e., $B_{3,k}\left(x\right):=\widetilde{B}_{3,k}\left(x\right),~k=0,1,2,3$. However, in order to ensure higher order non-vanishing derivatives, one could define the required basis functions for example as the linear combinations
\[
B_{3,0}\left(x\right) := \widetilde{B}_{4,0}\left(x\right),~
B_{3,1}\left(x\right) := \widetilde{B}_{4,1}\left(x\right) + \frac{1}{2}\widetilde{B}_{4,2}\left(x\right),~
B_{3,2}\left(x\right) :=  \frac{1}{2}\widetilde{B}_{4,2}\left(x\right)+\widetilde{B}_{4,3}\left(x\right),~
B_{3,3}\left(x\right) := \widetilde{B}_{4,4}\left(x\right)
\]
that also fulfill the conditions (\ref{eq:smoothness})--(\ref{eq:symmetry}).
\end{example}

\begin{example}[Second and higher order trigonometric normalized B-basis functions]\label{exmp:Sanchez}
Let $\beta\in\left(  0,\pi\right)  $ be a fixed shape parameter and let $\eta= 2\mu \geq 4$. 
The non-negative normalized B-basis of the EC space
\begin{equation}
\mathbb{S}_{2\mu}^{\beta}=\left\{S_{2\mu,0}\left(  u\right)
\equiv 1,~\left\{  S_{2\mu,2r-1}\left(
u\right)  =\sin\left(  r x\right)  ,~S_{2\mu,2r}\left(  x\right)
=\cos\left(  r x\right)  \right\}  _{r=1}^{\mu}:x\in\left[  0,\beta\right] \right\}
\label{eq:trigonometric_polynomials_n}%
\end{equation}
of trigonometric polynomials of order at most $\mu$ (degree $\eta = 2\mu$) provided by \citep{Sanchez1998} can linearly be reparametrized into the form%
\begin{equation}
\widetilde{\mathcal{B}}_{2\mu}^{\beta}=\left\{  \widetilde{B}_{2\mu,r}\left(  x\right)  =c_{2\mu,r}%
^{\beta}\sin^{2\mu-r}\left(  \frac{\beta-x}{2}\right)  \sin^{r}\left(  \frac
{x}{2}\right)  :x\in\left[  0,\beta\right]  \right\}  _{r=0}^{2\mu},~\mu \geq 2
\label{eq:trigonometric_B-basis-n}%
\end{equation}
where%
\begin{equation}
c_{2\mu,r}^{\beta}=c_{2\mu,2\mu-r}^{\beta}=\frac{1}{\sin^{2\mu}\left(  \frac{\beta}%
	{2}\right)  }\sum_{\ell=0}^{\left\lfloor \frac{r}{2}\right\rfloor }\binom{\mu}%
{r-\ell}\binom{r-\ell}{\ell}\left(  2\cos\left(  \frac{\beta}{2}\right)  \right)
^{r-2\ell},~r=0,1,\ldots,\mu \label{eq:trigonometric_normalizing_constants_n}%
\end{equation}
are symmetric normalizing coefficients. Similarly to Example \ref{exmp:Bernstein}, if $\mu = 2$ (i.e., $\eta = 4$), then the functions of the system (\ref{eq:system}) can be defined for example as follows
\[
B_{3,0}\left(x\right) := \widetilde{B}_{4,0}\left(x\right),~
B_{3,1}\left(x\right) := \widetilde{B}_{4,1}\left(x\right) + \frac{1}{2}\widetilde{B}_{4,2}\left(x\right),~
B_{3,2}\left(x\right) :=  \frac{1}{2}\widetilde{B}_{4,2}\left(x\right)+\widetilde{B}_{4,3}\left(x\right),~
B_{3,3}\left(x\right) := \widetilde{B}_{4,4}\left(x\right).
\]
\end{example}

\begin{example}[Second and higher oder hyperbolic normalized B-basis functions]\label{exmp:hyperbolic}
	Now, let $\beta>0$ and $\eta = 2\mu \geq 4$ be fixed parameters. Using hyperbolic sine and cosine
	functions in expressions (\ref{eq:trigonometric_polynomials_n}%
	)--(\ref{eq:trigonometric_normalizing_constants_n}) instead of the
	trigonometric ones, we obtain the vector space of hyperbolic polynomials of
	order at most $\mu$ (or degree $\eta = 2\mu$) the unique non-negative normalized B-basis of which was
	introduced in \citep{ShenWang2005}. If $\mu = 2$ (i.e., $\eta = 4$), then the basis functions (\ref{eq:system}) can be defined as
	in Example \ref{exmp:Sanchez}.
\end{example}

\begin{example}[First order algebraic-trigonometric normalized B-basis functions]
	\label{exmp:algebraic_trigonometric}The non-negative normalized B-basis%
	\begin{align*}
	\widetilde{\mathcal{B}}_{3}^{\beta}=  &  ~\left\{  \widetilde{\mathcal{B}}_{3,0}\left(  x\right)
	=\widetilde{\mathcal{B}}_{3,3}\left(  \beta-x\right)  ,~\widetilde{\mathcal{B}}_{3,1}\left(  x\right)  =\widetilde{\mathcal{B}}_{3,2}\left(
	\beta-x\right)  ,\right. \\
	&  ~~\left.  \widetilde{\mathcal{B}}_{3,2}\left(  x\right)  =\frac{\left(  \beta-x+\sin\left(  \beta
		-x\right)  +\sin\left(  x\right)  -\sin\left(  \beta\right)  +x\cos\left(
		\beta\right)  -\beta\cos\left(  x\right)  \right)  \cdot \sin\left(  \beta\right)
	}{\left(  \beta-\sin\left(  \beta\right)  \right)  \left(  2\sin\left(
	\beta\right)  -\beta-\beta\cos\left(  \beta\right)  \right)  }\right., \\
	&  ~~\left.  \widetilde{\mathcal{B}}_{3,3}\left(  x\right)  =\frac{x-\sin\left(  x\right)  }{\beta
		-\sin\left(  \beta\right)  } :x\in\left[  0,\beta\right]  \right\}  ,~\beta
	\in\left(  0,2\pi\right)
	\end{align*}
	of the mixed EC space %
	$
	\mathbb{S}_{3}^{\beta}    =\operatorname{span} \{  S_{3,0}\left(  x\right)  =1,~S_{3,1}\left(  x\right)  =x,~S_{3,2}\left(  x\right)  =\sin\left(x\right)%
	,S_{3,3}\left(  x\right)  =\cos\left(  x\right)   :x\in\left[  0,\beta\right] \}
	\rangle
	$
	of algebraic-trigonometric functions can be constructed by using either
	the differential equation based iterative integral representation published in
	\citep{MainarPena2010} and references therein or the determinant based formulas
	of \citep[Theorem 3.4]{Mazure1999}. The critical length %
	$2\pi$ was determined in \citep[Section 5]{CarnicerMainarPena2004}. As in the cubic case of the polynomial Example \ref{exmp:Bernstein}, the required basis functions (\ref{eq:system}) can be defined by direct index correspondence, i.e., $B_{3,k}\left(x\right) := \widetilde{B}_k\left(x\right),~k=0,1,2,3$. 
\end{example}

\begin{remark}
	Observe that the polynomial, trigonometric, hyperbolic and algebraic-trigonometric reflection invariant EC spaces detailed in Examples \ref{exmp:Bernstein}--\ref{exmp:algebraic_trigonometric} above correspond to the spaces of solutions of those constant-coefficient homogeneous linear differential equations of order $\eta + 1$ that are determined by the either even or odd characteristic polynomials
	$
	p_{\eta + 1}\left(z\right) = z^\eta, ~
	p_{\eta+1}\left(z\right) := p_{2\mu+1} \left(z\right) =  z \prod_{r = 1}^\mu \left(z^2 + r^2\right), ~
	p_{\eta+1}\left(z\right) := p_{2\mu+1} \left(z\right) =  z \prod_{r = 1}^\mu \left(z^2 - r^2\right)
	$
	and
	$
	p_{\eta+1}\left(z\right) := p_4\left(z\right) = z^2(z^2+1),
	$
	respectively, where $z \in \mathbb{C}$. Constructing the normalized B-basis functions of other reflection invariant (mixed) EC spaces \citep{MainarPena2010}, one can also define other useful systems in a similar fashion.
\end{remark}

\section{Construction of networks of local interpolating optimal curves}\label{sec:optimal_arcs}

Let $(i,j)\in \mathcal{E}$ be an arbitrarily selected directed edge of the given triangular mesh and consider the unit tangent vectors $\mathbf{t}_{i,j}$ and $\mathbf{t}_{j,i}$ associated with vertices $\mathbf{p}_i$ and $\mathbf{p}_j$, respectively. Our first objective is to determine the scaling factors $\lambda_{i,j}$ and $\lambda
_{j,i}$ of tangent vectors $\mathbf{t}_{i,j}$ and $\mathbf{t}_{j,i}$, respectively, such that the curve%
\[
\mathbf{c}_{i,j}\left(  x;\lambda_{i,j},\lambda_{j,i}\right)  =\mathbf{p}%
_{i}\cdot B_{3,0}\left(  x\right)  +\left(  \mathbf{p}_{i}+\lambda
_{i,j}\mathbf{t}_{i,j}\right)  \cdot B_{3,1}\left(  x\right)  +\left(
\mathbf{p}_{j}+\lambda_{j,i}\mathbf{t}_{j,i}\right)  \cdot B_{3,2}\left(
x\right)  +\mathbf{p}_{j}\cdot B_{3,3}\left(  x\right)  ,~x\in\left[
0,\beta\right]
\]
becomes the solution of the optimization problem \citep{Schweikert1966,Barsky84,LasserHagen1992}%
\begin{equation}
\label{eq:strain_energy}
E_{i,j}^{\rho}\left(  \lambda_{i,j},\lambda_{j,i}\right)  :=%
{\displaystyle\sum\limits_{r=1}^{\rho}}
\theta_{r}%
{\displaystyle\int\limits_{0}^{\beta}}
\left\Vert \mathbf{c}_{i,j}^{\left(  r\right)  }\left(  x;\lambda
_{i,j},\lambda_{j,i}\right)  \right\Vert ^{2}\text{d}x\rightarrow\min,
\end{equation}
where $\rho\geq1$ denotes the maximal order of involved derivatives (such that $B_{3,1}^{\left(r\right)} \not\equiv 0 \not\equiv B_{3,2}^{\left(r\right)}, ~\forall r=1,\ldots,\rho$), while
parameters $\left\{  \theta_{r}\right\}  _{r=1}^{\rho}$ are user defined
non-negative weights of rank $1$ (i.e., $\theta_{r}\geq0,~r=1,\ldots
,\rho,~\sum_{r=1}^{\rho}\theta_{r}\neq0$).

Let us denote by $\left\langle\cdot,\cdot\right\rangle$ the inner product of two vectors.
Since this operator is linear in both of its components and $\left\|\cdot\right\|^2:=\left\langle\cdot,\cdot\right\rangle$, by using the notations%
\begin{align*}
\varphi_{k,\ell}^{r}  &  :=\int_{0}^{\beta}B_{3,k}^{\left(
r\right)  }\left(  x\right)  B_{3,\ell}^{\left(  r\right)  }\left(  x\right)
\text{d}x=:\varphi_{\ell,k}^{r},~k,\ell\in\left\{  0,1,2,3\right\}  ,\\
\phi_{k,\ell}^{\rho}  &  :=\sum_{r=1}^{\rho}\theta_r \cdot \varphi_{k,\ell}^{r},~k,\ell
\in\left\{  0,1,2,3\right\}  ,
\end{align*}
one obtains that%
\begin{align*}
&  ~E_{i,j}^{\rho}\left(  \lambda_{i,j},\lambda_{j,i}\right) \\
=  &  ~\left\langle \mathbf{p}_{i},\mathbf{p}_{i}\right\rangle \cdot\phi
_{0,0}^{\rho}+2\cdot\left\langle \mathbf{p}_{i}+\lambda_{i,j}\mathbf{t}%
_{i,j},\mathbf{p}_{i}\right\rangle \cdot\phi_{0,1}^{\rho}+2\cdot\left\langle
\mathbf{p}_{i},\mathbf{p}_{j}+\lambda_{j,i}\mathbf{t}_{j,i}\right\rangle
\cdot\phi_{0,2}^{\rho}+2\cdot\left\langle \mathbf{p}_{i},\mathbf{p}%
_{j}\right\rangle \cdot\phi_{0,3}^{\rho}\\
&  ~+\left\langle \mathbf{p}_{i}+\lambda_{i,j}\mathbf{t}_{i,j},\mathbf{p}%
_{i}+\lambda_{i,j}\mathbf{t}_{i,j}\right\rangle \cdot\phi_{1,1}^{\rho}%
+2\cdot\left\langle \mathbf{p}_{i}+\lambda_{i,j}\mathbf{t}_{i,j}%
,\mathbf{p}_{j}+\lambda_{j,i}\mathbf{t}_{j,i}\right\rangle \cdot\phi
_{1,2}^{\rho}+2\cdot\left\langle \mathbf{p}_{i}+\lambda_{i,j}\mathbf{t}%
_{i,j},\mathbf{p}_{j}\right\rangle \cdot\phi_{1,3}^{\rho}\\
&  ~+\left\langle \mathbf{p}_{j}+\lambda_{j,i}\mathbf{t}_{j,i},\mathbf{p}%
_{j}+\lambda_{j,i}\mathbf{t}_{j,i}\right\rangle \cdot\phi_{2,2}^{\rho}%
+2\cdot\left\langle \mathbf{p}_{j}+\lambda_{j,i}\mathbf{t}_{j,i}%
,\mathbf{p}_{j}\right\rangle \cdot\phi_{2,3}^{\rho}\\
&  ~+\left\langle \mathbf{p}_{j},\mathbf{p}_{j}\right\rangle \cdot\phi
_{3,3}^{\rho},
\end{align*}
by means of which the unknown parameters $\lambda_{i,j},\lambda_{j,i}$ can be
determined as the solution of the linear system%
\[\def\arraystretch{1.8}
\left\{
\begin{array}
[c]{ccc}%
\dfrac{\partial}{\partial\lambda_{i,j}}E_{i,j}^{\rho}\left(  \lambda
_{i,j},\lambda_{j,i}\right)  & = & 0,\\
\dfrac{\partial}{\partial\lambda_{j,i}}E_{i,j}^{\rho}\left(  \lambda
_{i,j},\lambda_{j,i}\right)  & = & 0,
\end{array}
\right.\def\arraystretch{1.3}
\]
the matrix form of which is%
\begin{equation}
\left[
\begin{array}
[c]{cc}%
\phi_{1,1}^{\rho} & \left\langle \mathbf{t}_{i,j},\mathbf{t}_{j,i}%
\right\rangle \cdot\phi_{1,2}^{\rho}\\
\left\langle \mathbf{t}_{i,j},\mathbf{t}_{j,i}\right\rangle \cdot\phi
_{1,2}^{\rho} & \phi_{2,2}^{\rho}%
\end{array}
\right]  \left[
\begin{array}
[c]{c}%
\lambda_{i,j}\\
\lambda_{j,i}%
\end{array}
\right]  =-\left[
\begin{array}
[c]{c}%
\left\langle \mathbf{p}_{i}\cdot\left(  \phi_{0,1}^{\rho}+\phi_{1,1}^{\rho
}\right)  +\mathbf{p}_{j}\cdot\left(  \phi_{1,2}^{\rho}+\phi_{1,3}^{\rho
}\right)  ,\mathbf{t}_{i,j}\right\rangle \\
\left\langle \mathbf{p}_{i}\cdot\left(  \phi_{0,2}^{\rho}+\phi_{1,2}^{\rho
}\right)  +\mathbf{p}_{j}\cdot\left(  \phi_{2,2}^{\rho}+\phi_{2,3}^{\rho
}\right)  ,\mathbf{t}_{j,i}\right\rangle
\end{array}
\right]  . \label{eq:matrix_equation}%
\end{equation}
\begin{proposition}[Uniqueness of the solution]
	The linear system (\ref{eq:matrix_equation}) always admits a unique solution, i.e.,
	\[
	\Delta_{i,j}^{\rho}:=\det\left[
	\begin{array}
	[c]{cc}%
	\phi_{1,1}^{\rho} & \left\langle \mathbf{t}_{i,j},\mathbf{t}_{j,i}%
	\right\rangle \cdot\phi_{1,2}^{\rho}\\
	\left\langle \mathbf{t}_{i,j},\mathbf{t}_{j,i}\right\rangle \cdot\phi
	_{1,2}^{\rho} & \phi_{2,2}^{\rho}%
	\end{array}
	\right]  =\phi_{1,1}^{\rho}\cdot\phi_{2,2}^{\rho}-\left\langle \mathbf{t}%
	_{i,j},\mathbf{t}_{j,i}\right\rangle ^{2}\cdot\left(  \phi_{1,2}^{\rho
	}\right)  ^{2}\neq0.
	\]
\end{proposition}
\begin{proof}
	Observe that by means of the well-known Cauchy--Schwarz inequality, one can
	obtain that%
	\begin{align*}
	\left(  \phi_{1,2}^{\rho}\right)  ^{2}  &  =\left(  \sum_{r=1}^{\rho}%
	\theta_{r}\int_{0}^{\beta}B_{3,1}^{\left(  r\right)  }\left(  x\right)
	B_{3,2}^{\left(  r\right)  }\left(  x\right)  \text{d}x\right)  ^{2}\\
	&  \leq\left(  \sum_{r=1}^{\rho}\theta_{r}\int_{0}^{\beta}\left(
	B_{3,1}^{\left(  r\right)  }\left(  x\right)  \right)  ^{2}\text{d}x\right)
	\cdot\left(  \sum_{r=1}^{\rho}\theta_{r}\int_{0}^{\beta}\left(  B_{3,2}%
	^{\left(  r\right)  }\left(  x\right)  \right)  ^{2}\text{d}x\right) \\
	&  =\phi_{1,1}^{\rho}\cdot\phi_{2,2}^{\rho},
	\end{align*}
	where the equality holds if and only if $B_{3,1}^{\left(r\right)}\left(x\right) = \alpha \cdot B_{3,2}^{\left(r\right)}\left(x\right),~\forall x \in \left[0,\beta\right],~\forall r=1,\ldots,\rho$ for some fixed constant $\alpha\in\mathbb{R}$, which is impossible by definition, since $B_{3,1}^{\left(r\right)} \not\equiv 0 \not\equiv B_{3,2}^{\left(r\right)}$ and $B_{3,1}^{\left(r\right)}\left(x\right) = \left(-1\right)^r B_{3,2}^{\left(r\right)}\left(\beta - x\right),~\forall x \in \left[0,\beta\right],~\forall r=1,\ldots,\rho$ due to the symmetry condition (\ref{eq:symmetry}), therefore%
	\begin{equation}
	\left(  \phi_{1,2}^{\rho}\right)  ^{2}<\phi_{1,1}^{\rho}\cdot\phi_{2,2}^{\rho
	}. \label{eq:CBS_1}%
	\end{equation}

	Using once again the Cauchy--Schwarz inequality, we also obtain that%
	\begin{equation}
	0\leq\left\langle \mathbf{t}_{i,j},\mathbf{t}_{j,i}\right\rangle ^{2}%
	\leq\left\Vert \mathbf{t}_{i,j}\right\Vert ^{2}\cdot\left\Vert \mathbf{t}%
	_{j,i}\right\Vert ^{2}=1. \label{eq:CBS_2}%
	\end{equation}
	Combining inequalities (\ref{eq:CBS_1}) and (\ref{eq:CBS_2}), finally, we have
	that %
	$
	\Delta_{i,j}^{\rho}>0,
	$
	i.e., the system (\ref{eq:matrix_equation}) always admits the unique solution%
	\begin{small}
		\def\arraystretch{1.6}
		\begin{equation}
		\label{eq:solution}
		\left[
		\begin{array}
		[c]{c}%
		\lambda_{i,j}\\
		\lambda_{j,i}%
		\end{array}
		\right] 
		=-\dfrac{1}{\Delta_{i,j}^{\rho}}\cdot\left[
		\begin{array}
		[c]{cc}%
		\phi_{2,2}^{\rho} & -\left\langle \mathbf{t}_{i,j},\mathbf{t}_{j,i}%
		\right\rangle \cdot\phi_{1,2}^{\rho}\\
		-\left\langle \mathbf{t}_{i,j},\mathbf{t}_{j,i}\right\rangle \cdot\phi
		_{1,2}^{\rho} & \phi_{1,1}^{\rho}%
		\end{array}
		\right]  \cdot\left[
		\begin{array}
		[c]{c}%
		\left\langle \mathbf{p}_{i}\cdot\left(  \phi_{0,1}^{\rho}+\phi_{1,1}^{\rho
		}\right)  +\mathbf{p}_{j}\cdot\left(  \phi_{1,2}^{\rho}+\phi_{1,3}^{\rho
	}\right)  ,\mathbf{t}_{i,j}\right\rangle \\
	\left\langle \mathbf{p}_{i}\cdot\left(  \phi_{0,2}^{\rho}+\phi_{1,2}^{\rho
	}\right)  +\mathbf{p}_{j}\cdot\left(  \phi_{2,2}^{\rho}+\phi_{2,3}^{\rho
}\right)  ,\mathbf{t}_{j,i}\right\rangle
\end{array}
\right]
\end{equation}
\end{small}which completes the proof.\def\arraystretch{1.3}
\end{proof}

\begin{remark}[Invariance property of the solution]
	Observe that the unique solution (\ref{eq:solution}) of the linear system (\ref{eq:matrix_equation}) is invariant
	under parametrizations of the applied basis functions (\ref{eq:system}), moreover isometries and uniform scalings of the given mesh vertices would generate congruent or proportional local interpolating optimal arcs, respectively.
\end{remark}

The remaining parts of the current section provides explicit closed formulas for (\ref{eq:solution}) in case of basis functions that were introduced in Examples \ref{exmp:Bernstein}, \ref{exmp:Sanchez}, \ref{exmp:hyperbolic} and \ref{exmp:algebraic_trigonometric}.

\begin{example}
[Cubic Bernstein polynomials; $\rho=\theta_{1}=1$]%
\label{exmp:Bersntein_rho_1_theta1_1}Assume that the system $\mathcal{B}$
denotes the cubic Bernstein polynomials defined on $\left[
0,1\right]  $ and let $\rho=\theta_{1}=1$ be fixed parameters. In this
case, by means of values $\{\varphi_{0,1}^1,\varphi_{0,2}^1,\varphi_{1,1}^1,\varphi_{1,2}^1,\varphi_{1,3}^1,\allowbreak{}\varphi_{2,2}^1,\varphi_{2,3}^1\}$ listed in Appendix \ref{app:cubic_Bernstein_r_1}, one has that%
\begin{small}
\[
\left[
\begin{array}
[c]{c}%
\lambda_{i,j}\\
\lambda_{j,i}%
\end{array}
\right]  =\frac{1}{16-\left\langle \mathbf{t}_{i,j},\mathbf{t}_{j,i}%
\right\rangle ^{2}}\cdot\left[
\begin{array}
[c]{c}%
\left\langle \mathbf{p}_{j}-\mathbf{p}_{i},~\mathbf{t}_{j,i}\cdot\left\langle
\mathbf{t}_{i,j},\mathbf{t}_{j,i}\right\rangle +4\cdot\mathbf{t}%
_{i,j}\right\rangle \\
\left\langle \mathbf{p}_{i}-\mathbf{p}_{j},~\mathbf{t}_{i,j}\cdot\left\langle
\mathbf{t}_{i,j},\mathbf{t}_{j,i}\right\rangle +4\cdot\mathbf{t}%
_{j,i}\right\rangle
\end{array}
\right]  .
\]
\end{small}Using these parameter settings in case of a triangulated cube, Fig.\ \ref{fig:curve_energy_effect_of_order_and_shape_parameter}(\textit{a}) illustrates a network of local interpolating optimal cubic B\'ezier curves. Each arc of the network locally minimizes the squared length variation of its tangent vectors.
\end{example}

\begin{figure}[H]
	\begin{center}
		\includegraphics%
		{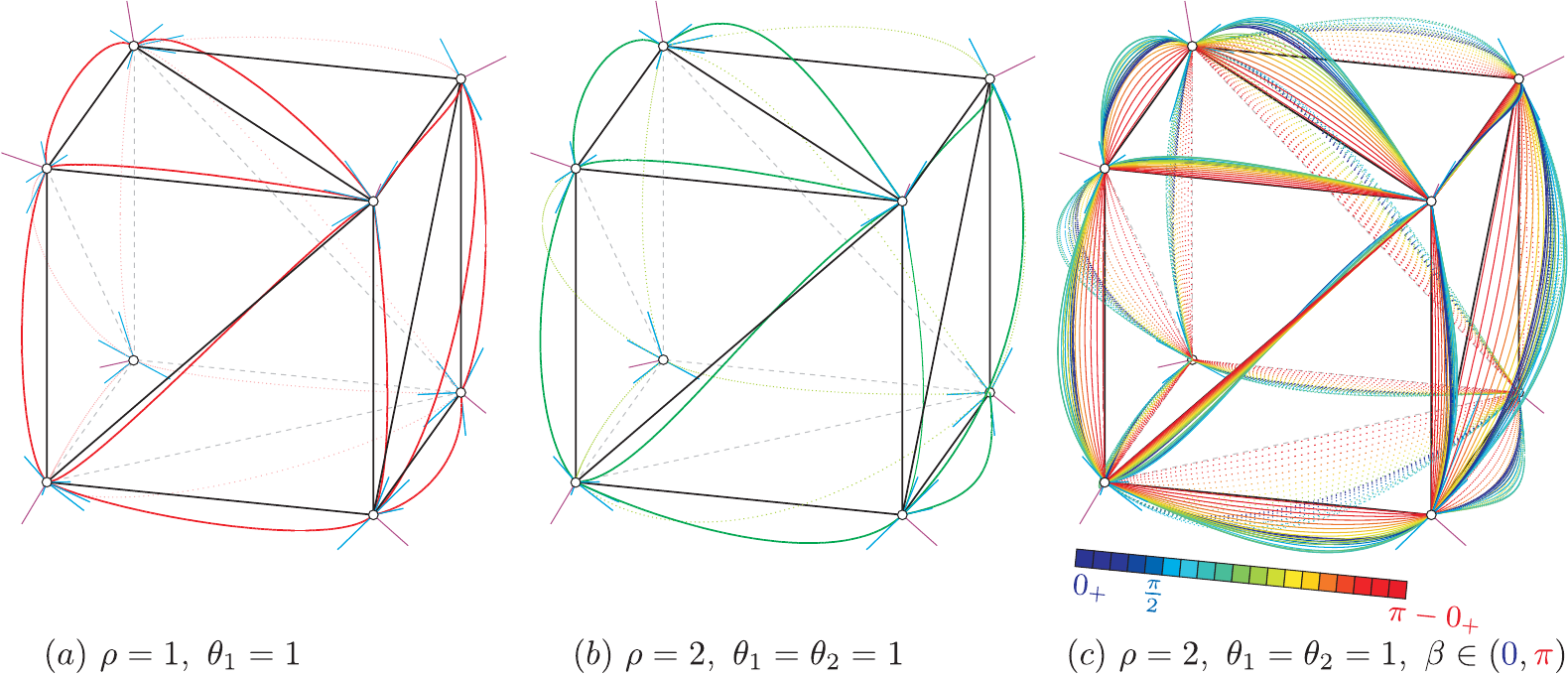}%
		\caption{Cases (\emph{a}), (\emph{b}) and (\emph{c}) show the output of the
			local interpolating piecewise optimal curve network construction method in case of parameter settings
			described in Examples \ref{exmp:Bersntein_rho_1_theta1_1},
			\ref{exmp:Bersntein_rho_2_theta1_1_theta2_1}\ and
			\ref{exmp:univariate_trigonometric_rho_1_2}, respectively.}%
		\label{fig:curve_energy_effect_of_order_and_shape_parameter}%
	\end{center}
\end{figure}

\begin{example}
[Cubic Bernstein polynomials; $\rho=2$, $\theta_{1}=\theta_{2}=1$%
]\label{exmp:Bersntein_rho_2_theta1_1_theta2_1}Let $\mathcal{B}$ once again be
the system of cubic Bernstein polynomials on $\left[
0,1\right]  $, and assume that $\rho=2$ and $\theta_{1}=\theta_{2}=1$ are
given parameters. Then, by using the values $\left\{\varphi_{0,1}^r,\varphi_{0,2}^r,\varphi_{1,1}^r,\varphi_{1,2}^r,\varphi_{1,3}^r,\varphi_{2,2}^r,\varphi_{2,3}^r\right\}_{r=1}^2$ listed in Appendices \ref{app:cubic_Bernstein_r_1}--\ref{app:cubic_Bernstein_r_2}, one obtains that%
\begin{small}
\[
\left[
\begin{array}
[c]{c}%
\lambda_{i,j}\\
\lambda_{j,i}%
\end{array}
\right]  =\frac{1}{15376-3481\cdot\left\langle \mathbf{t}_{i,j},\mathbf{t}%
_{j,i}\right\rangle ^{2}}\cdot\left[
\begin{array}
[c]{c}%
\left\langle \mathbf{p}_{j}-\mathbf{p}_{i},~7564\cdot\mathbf{t}_{i,j}%
-3599\cdot\mathbf{t}_{j,i}\cdot\left\langle \mathbf{t}_{i,j},\mathbf{t}%
_{j,i}\right\rangle \right\rangle \\
\left\langle \mathbf{p}_{i}-\mathbf{p}_{j},~7564\cdot\mathbf{t}_{j,i}%
-3599\cdot\mathbf{t}_{i,j}\cdot\left\langle \mathbf{t}_{i,j},\mathbf{t}%
_{j,i}\right\rangle \right\rangle
\end{array}
\right]  .
\]
\end{small}Using the same triangulated cube as in Example \ref{exmp:Bersntein_rho_1_theta1_1}, Fig.\ \ref{fig:curve_energy_effect_of_order_and_shape_parameter}(\textit{b}) shows the local interpolating optimal cubic B\'ezier curve network determined by the parameter settings above. The obtained network consists of arcs that locally minimize the combined squared length variation of their first and second order derivatives.
\end{example}

\begin{example}
	[Trigonometric basis functions; $\rho=2$%
	]\label{exmp:univariate_trigonometric_rho_1_2}
Let $\beta\in\left(  0,\pi\right)  $ and $\rho=2$ be
	 fixed parameters and consider the trigonometric basis functions introduced in Example \ref{exmp:Sanchez}. Evaluating the expressions $\{\varphi_{0,1}^r,\varphi_{0,2}^r,\varphi_{1,1}^r,\varphi_{1,2}^r,\allowbreak{}\varphi_{1,3}^r,\varphi_{2,2}^r,\allowbreak{}\varphi_{2,3}^r\}_{r=1}^2$ listed in Appendices \ref{app:trigonometric_r_1}--\ref{app:trigonometric_r_2} for $\beta = \frac{\pi}{2}$, in case of weight vectors $[\theta_1 = 1, ~\theta_2 = 0]$ and $[\theta_1 = 1, ~\theta_2 = 1]$ one can write the solution (\ref{eq:solution}) in the form
\begin{small}
\begin{align*}
\left[
\begin{array}
[c]{c}%
\lambda_{i,j}\\
\lambda_{j,i}%
\end{array}
\right]  &=\dfrac{3\pi-8}{2\cdot\left(  \left(  6\pi-16\right)  ^{2}-\left(
	10-3\pi\right)  ^{2}\cdot\left\langle \mathbf{t}_{i,j},\mathbf{t}%
	_{j,i}\right\rangle ^{2}\right)  }\cdot\left[
\begin{array}
[c]{c}%
\left\langle \mathbf{p}_{j}-\mathbf{p}_{i},\left(6\pi - 16\right)  \cdot\mathbf{t}_{i,j}+\left(  10-3\pi\right)
\cdot\left\langle \mathbf{t}_{i,j},\mathbf{t}_{j,i}\right\rangle
\cdot\mathbf{t}_{j,i}%
\right\rangle \\
\left\langle \mathbf{p}_{i}-\mathbf{p}_{j},\left(6\pi-16\right)  \cdot\mathbf{t}_{j,i}+\left(  10-3\pi\right)
\cdot\left\langle \mathbf{t}_{i,j},\mathbf{t}_{j,i}\right\rangle
\cdot\mathbf{t}_{i,j}%
\right\rangle
\end{array}
\right]
\end{align*}
\end{small}and
\begin{small}
\begin{align*}
\left[
\begin{array}
[c]{c}%
\lambda_{i,j}
\\
\lambda_{j,i}%
\end{array}
\right]   &= \frac{15\pi-16}{2\cdot\left(  \left(  21\pi-32\right)
	^{2}-\left(  15\pi-32\right)  ^{2} \cdot \left\langle \mathbf{t}_{i,j},\mathbf{t}_{j,i}\right\rangle ^{2}
	\right)  } %\\
~\cdot\left[
\begin{array}
[c]{c}%
\left\langle \mathbf{p}_{j}-\mathbf{p}_{i},\left(  21\pi-32\right)
\cdot\mathbf{t}_{i,j}-\left(  15\pi-32\right)  \cdot\left\langle
\mathbf{t}_{i,j},\mathbf{t}_{j,i}\right\rangle \cdot\mathbf{t}_{j,i}%
\right\rangle \\
\left\langle \mathbf{p}_{i}-\mathbf{p}_{j},\left(  21\pi-32\right)
\cdot\mathbf{t}_{j,i}-\left(  15\pi-32\right)  \cdot\left\langle
\mathbf{t}_{i,j},\mathbf{t}_{j,i}\right\rangle \cdot\mathbf{t}_{i,j}%
\right\rangle
\end{array}
\right],
\end{align*}
\end{small}respectively. Using the same triangulated cube as in Examples \ref{exmp:Bersntein_rho_1_theta1_1}--\ref{exmp:Bersntein_rho_2_theta1_1_theta2_1}, Fig.\ \ref{fig:curve_energy_effect_of_order_and_shape_parameter}(\textit{c}) uses the weights $\theta_1 = \theta_2 = 1$ and shows the effect of the design parameter $\beta\in\left(0,\pi\right)$ on the shape of the local interpolating optimal curve network that -- compared with the polynomial Example \ref{exmp:Bersntein_rho_2_theta1_1_theta2_1} -- now consists of second order (quartic) trigonometric arcs which also minimize the combined squared length variations of their velocity and acceleration vectors.
\end{example}

In case of hyperbolic and algebraic-trigonometric basis functions, Appendix \ref{app:univariate_integrals} lists further explicit formulas for the corresponding $\beta$-dependent values of the integrals $\varphi^{r}_{0,1} = \varphi^{r}_{2,3}$, $\varphi^{r}_{0,2} = \varphi^{r}_{1,3}$, $\varphi^{r}_{1,1}=\varphi^{r}_{2,2}$ and $\varphi^{r}_{1,2}$, where $r=1,2$. As in case of Examples \ref{exmp:Bersntein_rho_1_theta1_1}--\ref{exmp:univariate_trigonometric_rho_1_2}, for a given shape parameter $\beta$ and non-negative weight vector $\left[\theta_r\right]_{r=1}^2$ of rank 1, these values can easily be evaluated and substituted into the general formula (\ref{eq:solution}).

\section{Construction of $C^{0}$ continuous local interpolating optimal triangular spline surfaces}\label{sec:C0_optimal_patches}

Consider the counterclockwise oriented triangular face $\left(  i,j,k\right)  \in\mathcal{F}$, the
domain%
\[
\Omega_{\beta}=\left\{  \left(  u,v,w\right)  :u,v,w\in\left[  0,\beta\right]
,~u+v+w=\beta\right\}
\]
and the triangular surface%
\begin{equation}
\mathbf{s}_{i,j,k}\left(  u,v,w\right)  =\sum_{r=0}^{3}\sum_{s=0}%
^{3-r}\mathbf{p}_{r,s,3-r-s}^{i,j,k}T_{r,s,3-r-s}\left(  u,v,w\right)
,~\left(  u,v,w\right)  \in\Omega_{\beta}, \label{eq:local_triangular_surface}%
\end{equation}
where the sufficiently smooth non-negative constrained trivariate normalized function system%
\begin{equation}
\mathcal{T}_{\beta}:=\left\{  T_{r,s,3-r-s}\left(  u,v,w\right)  :\left(  u,v,w\right)  \in
\Omega_{\beta}\right\}  _{r=0,~s=0}^{3,~3-r}
\label{eq:trivariate_function_system}%
\end{equation}
is linearly independent and also fulfills the boundary conditions%
\begin{align*}
\left\{  T_{r,s,3-r-s}\left(  0,\beta-w,w\right)  :w\in\left[  0,\beta\right]
\right\}  _{r=0,~s=0}^{3,~3-r}  &  =\left\{  B_{3,0}\left(  w\right)
,B_{3,1}\left(  w\right)  ,B_{3,2}\left(  w\right)  ,B_{3,3}\left(  w\right)
:w\in\left[  0,\beta\right]  \right\}  ,\\
\left\{  T_{r,s,3-r-s}\left(  u,0,\beta-u\right)  :u\in\left[  0,\beta\right]
\right\}  _{r=0,~s=0}^{3,~3-r}  &  =\left\{  B_{3,0}\left(  u\right)
,B_{3,1}\left(  u\right)  ,B_{3,2}\left(  u\right)  ,B_{3,3}\left(  u\right)
:u\in\left[  0,\beta\right]  \right\}  ,\\
\left\{  T_{r,s,3-r-s}\left(  \beta-v,v,0\right)  :v\in\left[  0,\beta\right]
\right\}  _{r=0,~s=0}^{3,~3-r}  &  =\left\{  B_{3,0}\left(  v\right)
,B_{3,1}\left(  v\right)  ,B_{3,2}\left(  v\right)  ,B_{3,3}\left(  v\right)
:v\in\left[  0,\beta\right]  \right\}  .
\end{align*}

Using the notations of Fig.\ \ref{fig:local_surface}, we also assume that the
boundary curves are determined by the control polygons%
\[%
\begin{array}
[c]{llll}%
\lbrack~\mathbf{p}_{3,0,0}^{i,j,k}=\mathbf{p}_{i}, & \mathbf{p}_{2,0,1}%
^{i,j,k}=\mathbf{p}_{i}+\lambda_{i,j}\mathbf{t}_{i,j}, & \mathbf{p}%
_{1,0,2}^{i,j,k}=\mathbf{p}_{j}+\lambda_{j,i}\mathbf{t}_{j,i}, &
\mathbf{p}_{0,0,3}^{i,j,k}=\mathbf{p}_{j}~],\\
\lbrack~\mathbf{p}_{0,0,3}^{i,j,k}=\mathbf{p}_{j}, & \mathbf{p}_{0,1,2}%
^{i,j,k}=\mathbf{p}_{j}+\lambda_{j,k}\mathbf{t}_{j,k}, & \mathbf{p}%
_{0,2,1}^{i,j,k}=\mathbf{p}_{k}+\lambda_{k,j}\mathbf{t}_{k,j}, &
\mathbf{p}_{0,3,0}^{i,j,k}=\mathbf{p}_{k}~],\\
\lbrack~\mathbf{p}_{0,3,0}^{i,j,k}=\mathbf{p}_{k}, & \mathbf{p}_{1,2,0}%
^{i,j,k}=\mathbf{p}_{k}+\lambda_{k,i}\mathbf{t}_{k,i}, & \mathbf{p}%
_{2,1,0}^{i,j,k}=\mathbf{p}_{i}+\lambda_{i,k}\mathbf{t}_{i,k}, &
\mathbf{p}_{3,0,0}^{i,j,k}=\mathbf{p}_{i}~].
\end{array}
\]

\begin{figure}
[H]
\centering
\includegraphics%
{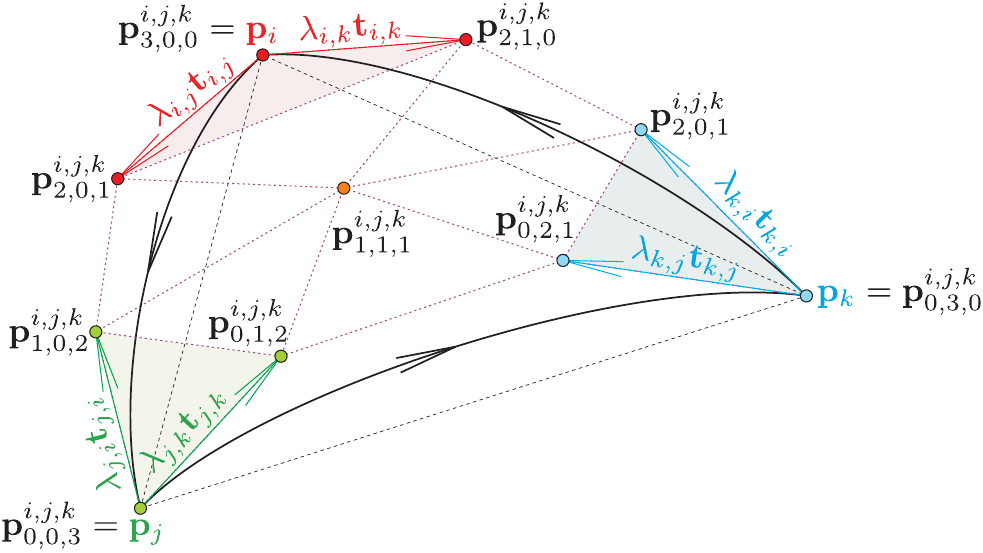}%
\caption{The control net of the local interpolating optimal triangular surface patch to be generated}%
\label{fig:local_surface}%
\end{figure} 

The unknown control point $\mathbf{p}_{1,1,1}^{i,j,k}$ of the local interpolating triangular
surface element $\mathbf{s}_{i,j,k}$ will be determined as the unique solution
of the optimization problem%
\[
\left\{
\begin{array}
[c]{l}%
E_{i,j,k}^{\gamma}\left( \mathbf{p}_{1,1,1}^{i,j,k} \right)  \rightarrow
\min,\\
\mathbf{p}_{1,1,1}^{i,j,k}\in%
\mathbb{R}^{3},
\end{array}
\right.
\]
where%
\begin{equation}
\label{eq:thin_plate_spline}
\begin{array}{rl}
&E_{i,j,k}^{\gamma}\left(  \mathbf{p}_{1,1,1}^{i,j,k}\right)
\\
=  &  \displaystyle\sum_{g=1}^{\gamma}\varepsilon_{g}\left(  \sum_{z=0}^{g}\binom{g}%
{z}\int_{0}^{\beta}\left(\int_{0}^{\beta-x}\left(  \sum_{r=0}^{3}\sum_{s=0}%
^{3-r}\mathbf{p}_{r,s,3-r-s}^{i,j,k}\frac{\partial^{g}}{\partial x^{z}\partial
	y^{g-z}}T_{r,s,3-r-s}\left(  x,y,\beta-x-y\right)  \right)  ^{2}%
\text{d}y\right)\text{d}x\right)
\end{array}
\end{equation}
denotes the generalized quadratic thin-plate spline surface energy \citep{Duchon1977} of order $\gamma
\geq1$, while the parameters $\left\{  \varepsilon_{g}\right\}  _{g=1}%
^{\gamma}$ are user defined non-negative weights of rank $1$.

Using the notations%
\begin{equation}
\def\arraystretch{1.7}
\label{eq:double_integrals}
\begin{array}{rl}
&\tau_{r,s,3-r-s}^{z,~g-z} \\  :=&\displaystyle\int_{0}^{\beta}\left(
\int_{0}^{\beta-x}\left(  \frac{\partial^{g}}{\partial x^{z}\partial y^{g-z}%
}T_{r,s,3-r-s}\left(  x,y,\beta-x-y\right)  \cdot\frac{\partial^{g}}{\partial
x^{z}\partial y^{g-z}}T_{1,1,1}\left(  x,y,\beta-x-y\right)  \right)
\text{d}y\right)  \text{d}x,
\\
&s=0,\ldots,3-r,~r=0,1,2,3,
\end{array}
\def\arraystretch{1.3}
\end{equation}
the $\mathbf{p}_{1,1,1}^{i,j,k}$-dependent part of the energy functional
(\ref{eq:thin_plate_spline}) is%
\begin{align*}
\widetilde{E}_{i,j,k}^{\gamma}\left(  \mathbf{p}_{1,1,1}^{i,j,k}\right)   
=&~2\sum_{g=1}^{\gamma}\varepsilon_{g}\sum_{z=0}^{g}\binom{g}{z}\left\langle \sum
_{r=0,~r\neq1}^{3}\sum_{s=0}^{3-r}\tau_{r,s,3-r-s}^{z,g-z}\cdot\mathbf{p}%
_{r,s,3-r-s}^{i,j,k}+\sum_{s=0,~s\neq1}^{2}\tau_{1,s,2-s}^{z,g-z}\cdot\mathbf{p}%
_{1,s,2-s}^{i,j,k},~\mathbf{p}_{1,1,1}^{i,j,k}\right\rangle \\
&  ~+\left(\sum_{g=1}^{\gamma}\varepsilon_{g}\sum_{z=0}^{g}\binom{g}{z}\tau_{1,1,1}^{z,g-z}\right)%
\cdot\left\langle \mathbf{p}_{1,1,1}^{i,j,k},~\mathbf{p}_{1,1,1}^{i,j,k}\right\rangle .
\end{align*}

\begin{proposition}[Uniqueness of the solution]
Provided that $\mathcal{T}_{\beta}$, $\gamma$ and $\{\varepsilon_g\}_{g=1}^{\gamma}$ are chosen such that the coefficient of $\langle \mathbf{p}_{1,1,1}^{i,j,k},~\mathbf{p}_{1,1,1}^{i,j,k}\rangle$ in $\widetilde{E}_{i,j,k}^{\gamma
}\left(\mathbf{p}_{1,1,1}^{i,j,k}\right)$ is not zero, the solution of%
\begin{align*}
\frac{\partial}{\partial\mathbf{p}_{1,1,1}^{i,j,k}}\widetilde{E}_{i,j,k}^{\gamma
}\left(  \mathbf{p}_{1,1,1}^{i,j,k}\right)  :=  &  ~2\sum_{g=1}^{\gamma}\varepsilon_{g}\sum_{z=0}%
^{g}\binom{g}{z}\left(  \sum_{r=0,~r\neq1}^{3}\sum_{s=0}^{3-r}\tau
_{r,s,3-r-s}^{z,g-z}\cdot\mathbf{p}_{r,s,3-r-s}^{i,j,k}+\sum_{s=0,~s\neq1}^{2}%
\tau_{1,s,2-s}^{z,g-z}\cdot\mathbf{p}_{1,s,2-s}^{i,j,k}\right) \\
&  ~+\left(2\sum_{g=1}^{\gamma}\varepsilon_{g}\sum_{z=0}^{g}\binom{g}{z}\tau_{1,1,1}^{z,g-z}\right)%
\cdot\mathbf{p}_{1,1,1}^{i,j,k}\\
  =&~\mathbf{0}%
\end{align*}
is the unique critical control point%
\begin{equation}
\label{eq:solution_surface}
\mathbf{p}_{1,1,1}^{i,j,k}=-\frac{\displaystyle\sum\limits_{g=1}^{\gamma}\varepsilon_{g}\displaystyle\sum\limits_{z=0}^{g}%
	\binom{g}{z}\left(  \displaystyle\sum\limits_{r=0,~r\neq1}^{3}\displaystyle\sum\limits_{s=0}^{3-r}\tau_{r,s,3-r-s}%
	^{z,g-z}\cdot\mathbf{p}_{r,s,3-r-s}^{i,j,k}+\displaystyle\sum\limits_{s=0,~s\neq1}^{2}\tau_{1,s,2-s}%
	^{z,g-z}\cdot\mathbf{p}_{1,s,2-s}^{i,j,k}\right)}{%
%TCIMACRO{\dsum \limits_{g=1}^{\gamma}}%
%BeginExpansion
{\displaystyle\sum\limits_{g=1}^{\gamma}}
%EndExpansion%
%TCIMACRO{\dsum \limits_{z=0}^{g}}%
%BeginExpansion
\varepsilon_{g}{\displaystyle\sum\limits_{z=0}^{g}}
%EndExpansion
\dbinom{g}{z}\tau_{1,1,1}^{z,g-z}}.
\end{equation}
\end{proposition}

\begin{remark}[Invariance property of the solution]
	Observe that unique solution (\ref{eq:solution_surface}) is also invariant under the parametrization of the applied constrained trivariate basis functions (\ref{eq:trivariate_function_system}), moreover isometries and uniform scalings of the given mesh vertices would generate congruent and proportional local interpolating optimal triangular patches, respectively.
\end{remark}

The following examples provide proper constrained trivariate linearly
independent function systems that can be used for the description of the triangular surface element (\ref{eq:local_triangular_surface}).

\begin{example}
[Constrained trivariate Bernstein polynomials; $\beta=1$]\label{exmp:trivariate_Bernstein}If the boundary
curves of (\ref{eq:local_triangular_surface}) are determined by means of
cubic Bernstein polynomials defined on the interval $\left[  0,1\right]  $,
then one should define the function system
(\ref{eq:trivariate_function_system}) as the cubic constrained trivariate
Bernstein polynomials%
\begin{align*}
T_{r,s,3-r-s}\left(  u,v,w\right)   = & ~B_{r,s,3-r-s}^3\left(  u,v,w\right)  \\ =&~ \frac{3!}{r!s!\left(  3-r-s\right)
!}u^{r}v^{s}w^{3-r-s},~\left(u,v,w\right)\in\Omega_1,
~s  =0,\ldots,3-r,~r=0,1,2,3
\end{align*}
that would lead to triangular cubic Bernstein--B\'ezier patches \citep{Farin1986}.
However, if one uses quartic Bernstein polynomials at the boundary $\partial\Omega_{1}$ as illustrated in Example $\ref{exmp:Bernstein}$, then one can define the functions of the required constrained trivariate basis as the linearly independent combinations
\begin{align*}
T_{0,0,3}\left(  u,v,w\right)  = &~ B_{0,0,4}^4\left(  u,v,w\right), \\
T_{0,1,2}\left(  u,v,w\right) =&~
B_{0,1,3}^4\left(  u,v,w\right)+
\frac{1}{2}B_{0,2,2}^4\left(  u,v,w\right),
\\
T_{0,2,1}\left(  u,v,w\right) =&~
\frac{1}{2}B_{0,2,2}^4\left(  u,v,w\right)+
B_{0,3,1}^4\left(  u,v,w\right),
\\
T_{0,3,0}\left(  u,v,w\right) =&~
B_{0,4,0}^4\left(  u,v,w\right),
\\
\\
T_{1,0,2}\left(  u,v,w\right) =&~
B_{1,0,3}^4\left(  u,v,w\right)
+
\frac{1}{2}B_{2,0,2}^4\left(  u,v,w\right)
,
\\
T_{1,1,1}\left(  u,v,w\right)= &~
B_{1,1,2}^4\left(  u,v,w\right)
+
B_{1,2,1}^4\left(  u,v,w\right)
+
B_{2,1,1}^4\left(  u,v,w\right)
,
\\
T_{1,2,0}\left(  u,v,w\right) =&~
B_{1,3,0}^4\left(  u,v,w\right)
+
\frac{1}{2}B_{2,2,0}^4\left(  u,v,w\right)
,
\\
\\
T_{2,0,1}\left(  u,v,w\right) =&~
\frac{1}{2}B_{2,0,2}^4\left(  u,v,w\right)
+
B_{3,0,1}^4\left(  u,v,w\right)
,
\\
T_{2,1,0}\left(  u,v,w\right) =&~
\frac{1}{2}B_{2,2,0}^4\left(  u,v,w\right)
+
B_{3,1,0}^4\left(  u,v,w\right)
,
\\
\\
T_{3,0,0}\left(  u,v,w\right) =&~
B_{4,0,0}^4\left(  u,v,w\right)
\end{align*}
of quartic constrained trivariate Bernstein polynomials. In this way the constructed functions will also fulfill six cyclic symmetry properties in their variables. The construction steps and the layout of the original and final systems of functions can also be seen in Fig.\ \ref{fig:quartic_polynomial_and_second_order_trigonometric_layout}(\textit{b})$\to$(\textit{a}). For the sake of convenience, in case of the latter quartic basis, the values of those double integrals (\ref{eq:double_integrals}) that are required for the minimization of the thin-plate spline energy (\ref{eq:thin_plate_spline}) of order at most $\gamma=2$, can be found in Appendix \ref{app:polynomial}.
\end{example}

\begin{figure}[H]
	\centering
	\includegraphics[width = \textwidth]{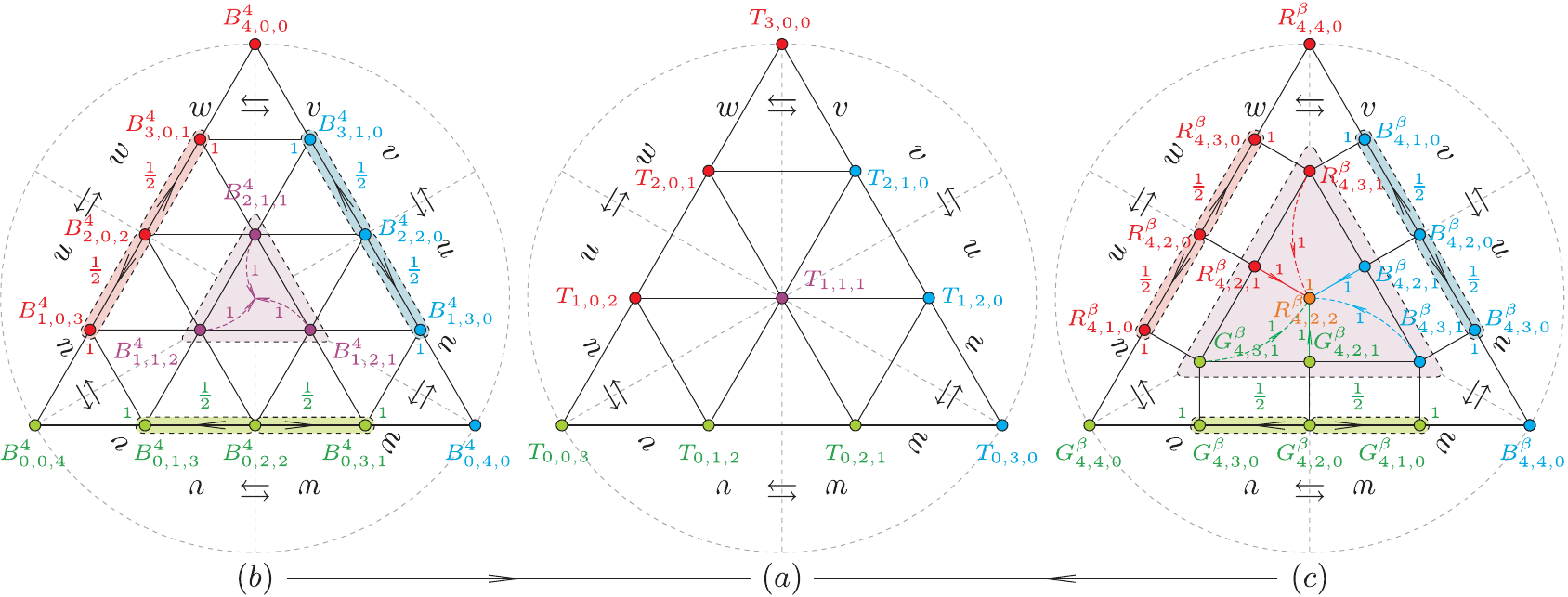}
	\caption{Cases (\textit{b}) and (\textit{c}) show the layout of the constrained trivariate quartic Bernstein polynomials and of the second order trigonometric basis that are detailed in Examples \ref{exmp:trivariate_Bernstein} and \ref{exmp:trivariate_trigonometric}, respectively. The highlighted nodes represent those functions that have to be linearly combined with the shown weights above the arrows or near the nodes in order to obtain the final non-negative normalized basis (\textit{a}) that preserves the six cyclic symmetry properties of the original systems.}
	\label{fig:quartic_polynomial_and_second_order_trigonometric_layout}
\end{figure}

\begin{example}
[Constrained trivariate trigonometric basis functions;~$\beta\in\left(
0,\pi\right)  $]\label{exmp:trivariate_trigonometric}Consider the domain $\Omega_{\beta}$, where $\beta\in\left(  0,\pi\right)  $ is
an arbitrarily fixed shape parameter. Using the second order non-negative constrained
trivariate normalized trigonometric basis functions%
\begin{align*}
R_{4,4,0}^{\beta}\left(  u,v,w\right)   &  =\frac{1}{\sin^{4}\left(
\frac{\beta}{2}\right)  }\sin^{4}\left(  \frac{u}{2}\right)  ,~R_{4,3,0}%
^{\beta}\left(  u,v,w\right)  =\frac{4\cos\left(  \frac{\beta}{2}\right)
}{\sin^{4}\left(  \frac{\beta}{2}\right)  }\sin^{3}\left(  \frac{u}{2}\right)
\sin\left(  \frac{w}{2}\right)  \cos\left(  \frac{v}{2}\right)  ,\\
R_{4,2,0}^{\beta}\left(  u,v,w\right)   &  =\frac{2+4\cos^{2}\left(
\frac{\beta}{2}\right)  }{\sin^{4}\left(  \frac{\beta}{2}\right)  }\sin
^{2}\left(  \frac{u}{2}\right)  \sin^{2}\left(  \frac{w}{2}\right)  \cos
^{2}\left(  \frac{v}{2}\right)  ,\\
R_{4,1,0}^{\beta}\left(  u,v,w\right)   &  =R_{4,3,0}^{\beta}\left(
w,v,u\right)  ,\\
R_{4,3,1}^{\beta}\left(  u,v,w\right)   &  =\frac{4+8\cos^{2}\left(
\frac{\beta}{2}\right)  }{\sin^{5}\left(  \frac{\beta}{2}\right)  }\sin
^{3}\left(  \frac{u}{2}\right)  \sin\left(  \frac{w}{2}\right)  \sin\left(
\frac{v}{2}\right)  ,\\
R_{4,2,1}^{\beta}\left(  u,v,w\right)  &=\frac
{16\cos\left(  \frac{\beta}{2}\right)  +8\cos^{3}\left(  \frac{\beta}%
{2}\right)  }{\sin^{5}\left(  \frac{\beta}{2}\right)  } \sin^2\left(\frac{u}{2}\right)\sin^2\left(\frac{w}{2}\right)\cos\left(\frac{v}{2}\right)\sin\left(\frac{v}{2}\right),\\
R_{4,2,2}^{\beta}\left(  u,v,w\right)   &  =\frac{10+20\cos^{2}\left(
\frac{\beta}{2}\right)  }{\sin^{4}\left(  \frac{\beta}{2}\right)  }\sin^2\left(
\frac{u}{2}\right)  \sin^2\left(  \frac{v}{2}\right)  \sin^2\left(  \frac{w}%
{2}\right)  ,\\
& \\
G_{4,4,0}^{\beta}\left(  u,v,w\right)   &  =R_{4,4,0}^{\beta}\left(
w,u,v\right)  ,~G_{4,3,0}^{\beta}\left(  u,v,w\right)  =R_{4,3,0}^{\beta
}\left(  w,u,v\right)  ,~G_{4,2,0}^{\beta}\left(  u,v,w\right)  =R_{4,2,0}%
^{\beta}\left(  w,u,v\right)  ,\\
G_{4,1,0}^{\beta}\left(  u,v,w\right)   &  =G_{4,3,0}^{\beta}\left(
u,w,v\right)  ,\\
G_{4,3,1}^{\beta}\left(  u,v,w\right)   &  =R_{4,3,1}^{\beta}\left(
w,u,v\right)  ,~G_{4,2,1}^{\beta}\left(  u,v,w\right)  =R_{4,2,1}^{\beta
}\left(  w,u,v\right) \\
& \\
B_{4,4,0}^{\beta}\left(  u,v,w\right)   &  =R_{4,4,0}^{\beta}\left(
v,w,u\right)  ,~B_{4,3,0}^{\beta}\left(  u,v,w\right)  =B_{4,3,0}^{\beta
}\left(  v,w,u\right)  ,~B_{4,2,0}^{\beta}\left(  v,w,u\right)  =R_{4,2,0}%
^{\beta}\left(  v,w,u\right)  ,\\
B_{4,1,0}^{\beta}\left(  u,v,w\right)   &  =B_{4,3,0}^{\beta}\left(
v,u,w\right)  ,\\
B_{4,3,1}^{\beta}\left(  u,v,w\right)   &  =R_{4,3,1}^{\beta}\left(
v,w,u\right)  ,~B_{4,2,1}^{\beta}\left(  u,v,w\right)  =R_{4,2,1}^{\beta
}\left(  v,w,u\right)  ,
\end{align*}
a linearly independent function system that is compatible along the boundary $\partial\Omega_{\beta}$ with the univariate trigonometric basis of Example
\ref{exmp:Sanchez} and also fulfills six cyclic symmetry properties can be constructed as follows:%
\begin{align*}
T_{0,0,3}\left(  u,v,w\right)  =  &  ~G_{4,4,0}^{\beta}\left(  u,v,w\right)
,\\
T_{0,1,2}\left(  u,v,w\right)  =  &  ~G_{4,3,0}^{\beta}\left(  u,v,w\right)
+\frac{1}{2}G_{4,2,0}^{\beta}\left(  u,v,w\right)  ,\\
T_{0,2,1}\left(  u,v,w\right)  =  &  ~\frac{1}{2}G_{4,2,0}^{\beta}\left(
u,v,w\right)  +G_{4,1,0}^{\beta}\left(  u,v,w\right)  ,\\
T_{0,3,0}\left(  u,v,w\right)  =  &  ~B_{4,4,0}^{\beta}\left(  u,v,w\right)
,\\
\\
T_{1,0,2}\left(  u,v,w\right)  =  &  ~\frac{1}{2}R_{4,2,0}^{\beta}\left(  u,v,w\right)
+R_{4,1,0}^{\beta}\left(  u,v,w\right)  ,\\
T_{1,1,1}\left(  u,v,w\right)  =  &  ~R_{4,3,1}^{\beta}\left(  u,v,w\right)
+R_{4,2,1}^{\beta}\left(  u,v,w\right)  +R_{4,2,2}^{\beta}\left(  u,v,w\right)
\\
&  ~+G_{4,3,1}^{\beta}\left(  u,v,w\right)  +G_{4,2,1}^{\beta}\left(
u,v,w\right) \\
&  ~+B_{4,3,1}^{\beta}\left(  u,v,w\right)  +B_{4,2,1}^{\beta}\left(
u,v,w\right)  ,\\
T_{1,2,0}\left(  u,v,w\right)  =  &  ~B_{4,3,0}^{\beta}\left(  u,v,w\right)
+\frac{1}{2}B_{4,2,0}^{\beta}\left(  u,v,w\right)  ,\\
\\
T_{2,0,1}\left(  u,v,w\right)  =  &  ~R_{4,3,0}^{\beta}\left(
u,v,w\right)  +\frac{1}{2}R_{4,2,0}^{\beta}\left(  u,v,w\right)  ,\\
T_{2,1,0}\left(  u,v,w\right)  =  &  ~\frac{1}{2}B_{4,2,0}^{\beta}\left(
u,v,w\right)  +B_{4,1,0}^{\beta}\left(  u,v,w\right)  ,\\
\\
T_{3,0,0}\left(  u,v,w\right)  =  &  ~R_{4,4,0}^{\beta}\left(  u,v,w\right).
\end{align*}
The construction steps and the layout of the original and final constrained trivariate functions can also be seen in Fig.\ \ref{fig:quartic_polynomial_and_second_order_trigonometric_layout}(\textit{c})$\to$(\textit{a}). Basis functions $\left\{R_{4,4-i,j}\right\}_{j=0,i=j}^{1,3-j}$, $R_{4,2,2}$, $\left\{G_{4,4-i,j}\right\}_{j=0,i=j}^{1,3-j}$ and $\left\{B_{4,4-i,j}\right\}_{j=0,i=j}^{1,3-j}$ were first introduced and later generalized to higher order in \citep{ShenWang2010} and \citep{Roth2013}, respectively. For the sake of convenience, the $\beta$-dependent values of those double integrals (\ref{eq:double_integrals}) that are required for the minimization of the thin-plate spline energy (\ref{eq:thin_plate_spline}) of order at most $\gamma=2$, can be found in Appendix \ref{app:trigonometric}.
\end{example}

\begin{example}[Constrained trivariate hyperbolic basis functions; $\beta > 0$]Let $\beta > 0$ be a fixed shape parameter. Using hyperbolic sine and cosine functions instead of the trigonometric ones presented in Example \ref{exmp:trivariate_trigonometric}, one obtains a non-negative constrained trivariate normalized hyperbolic basis that is compatible along $\partial\Omega_{\beta}$ with the univariate hyperbolic basis described in Example \ref{exmp:hyperbolic}.
\end{example}

\begin{example}[Constrained trivariate algebraic-trigonometric basis functions; $\beta\in\left(0,2\pi\right)$]\label{exmp:trivariate_algebraic_trigonometric}
	In this case, one can use the non-negative constrained trivariate normalized basis functions
%	\begin{small}
\begin{align*}
T_{0,0,3}\left(  u,v,w\right)  =&~T_{3,0,0}\left(  w,v,u\right)
%= &
%~\frac{w-\sin\left(  w\right)  }{\beta-\sin\left(  \beta\right)  }
,
~T_{0,1,2}\left(  u,v,w\right)  =T_{0,2,1}\left(  u,w,v\right)
%  = &
%~\frac{\left(  v+\sin\left(  v\right)  +\sin\left(  w\right)  -\sin\left(
%	\beta-u\right)  +w\cos\left(  \beta-u\right)  -\left(  \beta-u\right)
%	\cos\left(  w\right)  \right)  \sin\left(  \beta\right)  }{\left(
%	2\cos\left(  \beta\right)  +\beta\sin\left(  \beta\right)  -2\right)  \left(
%	\beta-\sin\left(  \beta\right)  \right)  }
,
~
T_{0,2,1}\left(  u,v,w\right)  =T_{1,2,0}\left(  w,v,u\right)  
%= &
%~\frac{\left(  w+\sin\left(  w\right)  +\sin\left(  v\right)  -\sin\left(
%	\beta-u\right)  +v\cos\left(  \beta-u\right)  -\left(  \beta-u\right)
%	\cos\left(  v\right)  \right)  \sin\left(  \beta\right)  }{\left(
%	2\cos\left(  \beta\right)  +\beta\sin\left(  \beta\right)  -2\right)  \left(
%	\beta-\sin\left(  \beta\right)  \right)  }
,\\
T_{0,3,0}\left(  u,v,w\right)  =&~T_{3,0,0}\left(  v,u,w\right)  
%= &
%~\frac{v-\sin\left(  v\right)  }{\beta-\sin\left(  \beta\right)  }
,\\
& \\
T_{1,0,2}\left(  u,v,w\right)  =&~T_{2,0,1}\left(  w,v,u\right) 
% = &
%~\frac{\left(  u+\sin\left(  u\right)  +\sin\left(  w\right)  -\sin\left(
%	\beta-v\right)  +w\cos\left(  \beta-v\right)  -\left(  \beta-v\right)
%	\cos\left(  w\right)  \right)  \sin\left(  \beta\right)  }{\left(
%	2\cos\left(  \beta\right)  +\beta\sin\left(  \beta\right)  -2\right)  \left(
%	\beta-\sin\left(  \beta\right)  \right)  }
,\\
T_{1,1,1}\left(  u,v,w\right)  = &  ~\frac{4\left(  3\beta+4\sin\left(
	\beta\right)  -\beta\cos\left(  \beta\right)  \right)  \cos\left(  \frac
	{\beta}{2}\right)    }{\left(  2\sin\left(  \beta\right) - \beta -\beta
	\cos\left(  \beta\right) \right)  \left(  \beta-\sin\left(  \beta\right)
	\right) }\sin\left(  \frac{u}{2}\right)  \sin\left(  \frac{v}%
{2}\right)  \sin\left(  \frac{w}{2}\right)
\\
&  ~-\frac{4\sin\left(  \beta\right)  \cos\left(  \frac{\beta}{2}\right)
	}{\left(  2\sin\left(  \beta\right) - \beta -\beta
	\cos\left(  \beta\right) \right)  \left(  \beta-\sin\left(  \beta\right)
	\right)  } \left(u\cos\left(  \frac{u}{2}\right)  \sin\left(  \frac{v}{2}\right)
\sin\left(  \frac{w}{2}\right)\right.
\\
&  ~\left.+ v\sin\left(  \frac{u}{2}\right)  \cos\left(  \frac{v}{2}\right)
\sin\left(  \frac{w}{2}\right)
+ w\sin\left(  \frac{u}{2}\right)  \sin\left(  \frac{v}{2}\right)
\cos\left(  \frac{w}{2}\right)\right),\\
T_{1,2,0}\left(  u,v,w\right)  =&~T_{2,1,0}\left(  v,u,w\right)  
,\\
& \\
T_{2,0,1}\left(  u,v,w\right)  =&~T_{2,1,0}\left(  u,w,v\right)  
,\\
T_{2,1,0}\left(  u,v,w\right)  =& ~ ~\frac{\left(  v+\sin\left(  v\right)
	+\sin\left(  u\right)  -\sin\left(  \beta-w\right)  +u\cos\left(
	\beta-w\right)  -\left(  \beta-w\right)  \cos\left(  u\right)  \right)
	\sin\left(  \beta\right)  }{\left(  2\sin\left(  \beta\right) - \beta -\beta
	\cos\left(  \beta\right) \right)  \left(  \beta-\sin\left(  \beta\right)
	\right)  },\\
& \\
T_{3,0,0}\left(  u,v,w\right) =  &~  ~\frac{u-\sin\left(  u\right)  }%
{\beta-\sin\left(  \beta\right)  }%
\end{align*}
%\end{small}
that were introduced in \citep{WeiShenWang2011} and which degenerate at $\partial\Omega_{\beta}$ to the univariate algebraic-trigonometric normalized B-basis presented in Example \ref{exmp:algebraic_trigonometric}. Once again, for the sake of convenience, the $\beta$-dependent values of those double integrals (\ref{eq:double_integrals}) that are required for the minimization of the thin-plate spline energy (\ref{eq:thin_plate_spline}) of order at most $\gamma=2$, can be found in Appendix \ref{app:algebraic_trigonometric}.
\end{example}

\section{Construction of visually smooth quasi-optimal Nielson-type transfinite triangular interpolants}\label{sec:generalized_G1_Nielson_patches}

In order to generate $G^1$ continuous triangular interpolants that fit the local interpolating piecewise optimal curve network constructed in Section \ref{sec:optimal_arcs}, this part of the manuscript follows the side-vertex transfinite interpolation scheme presented in \citep{Nielson1987}.  However, compared to this technique:
\begin{itemize}[noitemsep]
	\item 
	instead of cubic Hermite basis functions we use function systems obtained from non-negative normalized B-bases as described in Examples \ref{exmp:Bernstein}--\ref{exmp:algebraic_trigonometric}, thus ensuring more optimal shape preserving properties;
	
	\item
	by means of the local interpolating optimal $C^0$ triangular patches of Section \ref{sec:C0_optimal_patches} we propose continuous quasi-optimal vector fields of averaged unit normals along the common boundary curves;
	
	\item
	we also impose further optimality constraints concerning the isoparametric lines of those groups of three side-vertex interpolants that have to be convexly blended in order to generate the final visually smooth local interpolating quasi-optimal triangular surface patches.
\end{itemize}

In general, the local interpolating piecewise optimal triangular surface patches%
\[
\mathbf{s}_{i,j,k}\left(  u,v,w\right)  ,~\left(  u,v,w\right)  \in
\Omega_{\beta},~\left(  i,j,k\right)  \in\mathcal{F}%
\]
constructed in Section \ref{sec:C0_optimal_patches} are only $C^{0}$ continuous along their common boundary curves as it can also
be seen in cases (\textit{a}), (\textit{b}) and (\textit{c}) of Fig.\ \ref{fig:C0_patches_normal_vector_fields}. However, they can be used to define continuous vector fields
of averaged unit normals along the corresponding joints as it is illustrated in Fig.\ \ref{fig:C0_patches_normal_vector_fields}(\textit{d}) and mathematically detailed below. 

\begin{figure}[H]
	\centering
	\includegraphics{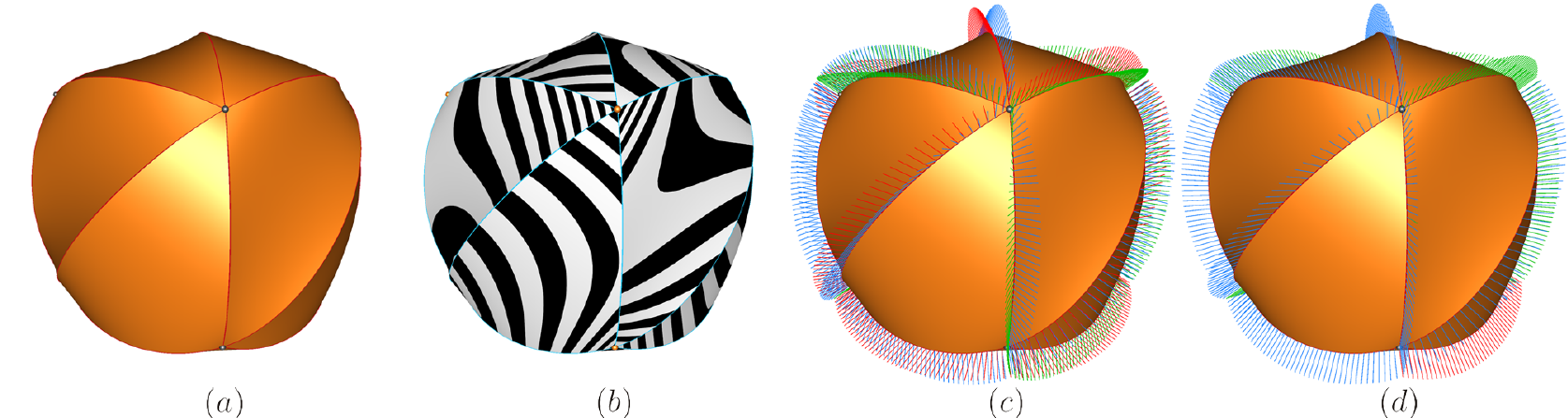}
	\caption{(\textit{a}) Optimal $C^0$ continuous cubic triangular Bernstein--B\'ezier interpolants obtained by parameter settings $\rho = \gamma = 2$ and $\theta_1=\theta_2=\varepsilon_1=\varepsilon_2 = 1$. (\textit{b}) Discontinuous reflection lines across the common local interpolating optimal boundary curves. (\textit{c}) Different vector fields of unit normals along the shared boundary curves. (\textit{d}) Continuous vector fields of averaged unit normals along the joints. ($n_v = 8$, $n_f=12$)}
	\label{fig:C0_patches_normal_vector_fields}
\end{figure}

Consider for example
the neighboring counterclockwise oriented faces $\left(  i,j,k\right)
,\left(  \ell,k,j\right)  \in\mathcal{F}$ and the local interpolating optimal triangular surface
patches $\mathbf{s}_{i,j,k}\left(  u,v,w\right)  $ and $\mathbf{s}_{\ell
,k,j}\left(  u,v,w\right)  $ implied by them, i.e., the selected patches share
the optimal boundary curve
\[
\mathbf{c}_{j,k}\left(  u\right)  =\mathbf{s}_{i,j,k}\left(  u,0,\beta
-u\right)  =\mathbf{s}_{\ell,k,j}\left(  \beta
-u,0,u\right)  ,~u\in\left[
0,\beta\right]  .
\]
If $\mathbf{n}_{i,j,k}\left(  u,v,w\right)  $ and $\mathbf{n}_{\ell
,k,j}\left(  u,v,w\right)  $ denote the unit normal vectors\ that belong to the
surface points $\mathbf{s}_{i,j,k}\left(  u,v,w\right)  $ and $\mathbf{s}%
_{\ell,k,j}\left(  u,v,w\right)  $, respectively, then one can define the
continuous vector field%
\begin{equation}
\label{eq:vector_field_of_averaged_normals}
\widetilde{\mathbf{n}}_{j,k}\left(  u\right)  =\frac{\mathbf{n}_{i,j,k}\left(
u,0,\beta
-u\right)  +\mathbf{n}_{\ell,k,j}\left(  \beta
-u,0,u\right)
}{\left\Vert \mathbf{n}_{i,j,k}\left(  u,0,\beta
-u\right)  +\mathbf{n}%
_{\ell,k,j}\left(  \beta-u,0,u\right)  \right\Vert },~u\in\left[
0,\beta\right]
\end{equation}
of averaged unit normals along the shared boundary curve $\mathbf{c}%
_{j,k}\left(  u\right)  $. If the selected face $\left(  i,j,k\right)  $ does
not have a neighbor across the edge $\left(  j,k\right)  $, then instead of
averaging one can use the inherited unit normal vector field
\begin{equation}
\label{eq:inherited_vector_fieldnormals}
\widetilde{\mathbf{n}}_{j,k}\left(  u\right)  =\frac{\mathbf{n}_{i,j,k}\left(
u,0,\beta-u\right)  }{\left\Vert \mathbf{n}_{i,j,k}\left(  u,0,\beta-u\right)
\right\Vert },~u\in\left[  0,\beta\right].
\end{equation}

Based on the connectivity information stored in $\mathcal{F}$ one can evaluate
all vector fields of (averaged) unit normals along the curve network
constructed in Section \ref{sec:optimal_arcs}. Therefore, in case of an arbitrarily
selected face $\left(  i,j,k\right)  \in\mathcal{F}$, we also assume the
existence of the unit normal vector field operators $\widetilde{\mathbf{n}%
}_{k,i}\left(  v\right) $ and $\widetilde
{\mathbf{n}}_{i,j}\left(  w\right)$ along the
boundaries $\mathbf{c}_{k,i}\left(  v\right) $
and $\mathbf{c}_{i,j}\left(  w\right)$, respectively, where $v,w\in\left[0,\beta\right]$. Due to the definition of these continuous normal vector fields, one has that
\[
\mathbf{n}_i = \widetilde{\mathbf{n}}_{i,j}\left(0\right) = \widetilde{\mathbf{n}}_{k,i}\left(\beta\right),~~
\mathbf{n}_j = \widetilde{\mathbf{n}}_{i,j}\left(\beta\right) = \widetilde{\mathbf{n}}_{j,k}\left(0\right)~~\text{and}~~
\mathbf{n}_k = \widetilde{\mathbf{n}}_{j,k}\left(\beta\right) = \widetilde{\mathbf{n}}_{k,i}\left(0\right)
\]
for all faces $\left( i,j,k\right)  \in\mathcal{F}$.

In case of the selected face $\left(  i,j,k\right)  \in\mathcal{F}$, consider
the boundary curve $\mathbf{c}_{j,k}\left(  u;\lambda_{j,k},\lambda
_{j,k}\right)$, its opposite vertex
$\mathbf{p}_{i}$ and its associated vector field $\widetilde{\mathbf{n}}%
_{j,k}\left(  u\right)$ of averaged unit normals. In what follows, the triangular surface patch%
\begin{equation}
\widetilde{\mathbf{s}}_{i,\left(  j,k\right)  }\left(  b_{0},b_{1}%
,b_{2}\right)  ,~\left(  b_{0},b_{1},b_{2}\right)  \in\Omega_{1-0_+}
\label{eq:patch_i_jk}%
\end{equation}
is constructed by joining the points of the selected
boundary with the opposite vertex $\mathbf{p}_{i}$ by optimal curves as it is illustrated in Fig.\
\ref{fig:side_vertex}, where $0_+$ denotes an arbitrarily small positive number.%

\begin{figure}
[H]
\centering
\includegraphics{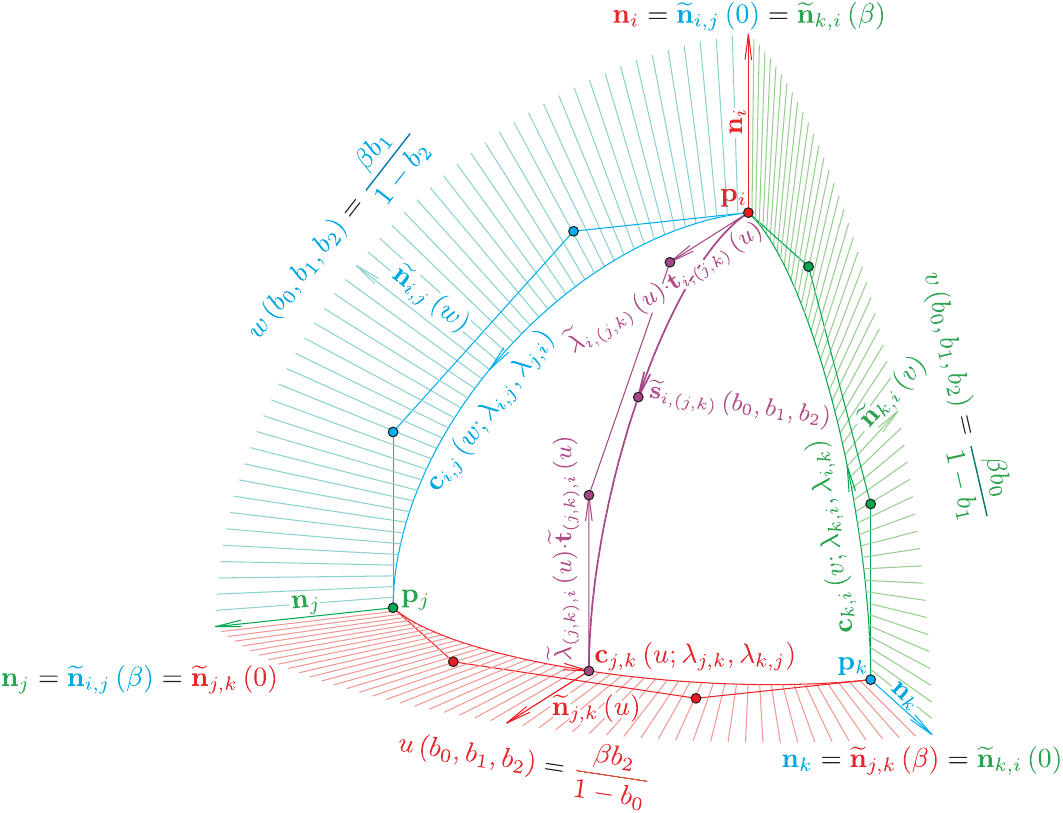}%
\caption{Construction of an optimal triangular side-vertex interpolant}%
\label{fig:side_vertex}%
\end{figure}

The parametric equation of the patch (\ref{eq:patch_i_jk}) can be derived as follows:

\begin{itemize}[noitemsep]
	\item for an arbitrarily fixed parameter value $u\in\left[  0,\beta\right]  $
	calculate the unit normal, binormal and tangent vectors%
	\[%
	\begin{array}
	[c]{cclccl}%
	\mathbf{f}_{i} & = & -\mathbf{n}_{i}, & \widetilde{\mathbf{f}}_{j,k}\left(
	u\right)  & = & -\widetilde{\mathbf{n}}_{j,k}\left(  u\right)  ,\\
	&  &  &  &  & \\
	\mathbf{b}_{i,\left(  j,k\right)  }\left(  u\right)  & = & \dfrac{\left(
		\mathbf{c}_{j,k}\left(  u;\lambda_{j,k},\lambda_{k,j}\right)  -\mathbf{p}%
		_{i}\right)  \times\mathbf{f}_{i}}{\left\Vert \left(  \mathbf{c}_{j,k}\left(
		u;\lambda_{j,k},\lambda_{k,j}\right)  -\mathbf{p}_{i}\right)  \times
		\mathbf{f}_{i}\right\Vert }, & \widetilde{\mathbf{b}}_{\left(  j,k\right)
		,i}\left(  u\right)  & = & \dfrac{\left(  \mathbf{p}_{i}-\mathbf{c}%
		_{j,k}\left(  u;\lambda_{j,k},\lambda_{k,j}\right)  \right)  \times
		\widetilde{\mathbf{f}}_{j,k}\left(  u\right)  }{\left\Vert \left(
		\mathbf{p}_{i}-\mathbf{c}_{j,k}\left(  u;\lambda_{j,k},\lambda_{k,j}\right)
		\right)  \times\widetilde{\mathbf{f}}_{j,k}\left(  u\right)  \right\Vert }%
	\end{array}
	\]
	and%
	\[%
	\begin{array}
	[c]{ccccccc}%
	\mathbf{t}_{i,\left(  j,k\right)  }\left(  u\right)  & = & \mathbf{f}%
	_{i}\times\mathbf{b}_{i,\left(  j,k\right)  }\left(  u\right)  , &  &
	\widetilde{\mathbf{t}}_{\left(  j,k\right)  ,i}\left(  u\right)  & = &
	\widetilde{\mathbf{f}}_{j,k}\left(  u\right)  \times\widetilde{\mathbf{b}%
	}_{\left(  j,k\right)  ,i}\left(  u\right)  ,
	\end{array}
	\]
	respectively;
	
	\item by means of the determinant%
	\[
	\widetilde{\Delta}_{i,\left(  j,k\right)  }^{\rho}\left(  u\right)
	=\phi_{1,1}^{\rho}\cdot\phi_{2,2}^{\rho}-\left\langle \mathbf{t}_{i,\left(
		j,k\right)  }\left(  u\right)  ,\widetilde{\mathbf{t}}_{\left(  j,k\right)
		,i}\left(  u\right)  \right\rangle ^{2}\cdot\left(  \phi_{1,2}^{\rho}\right)
	^{2}>0%
	\]
	determine the optimal scaling factors%
	\begin{equation}%
	\label{eq:point-wise_solution}
	\begin{array}
	[c]{ccl}%
	\left[
	\begin{array}
	[c]{c}%
	\widetilde{\lambda}_{i,\left(  j,k\right)  }\left(  u\right) \\
	\widetilde{\lambda}_{\left(  j,k\right)  ,i}\left(  u\right)
	\end{array}
	\right]  & = & -\dfrac{1}{\widetilde{\Delta}_{i,\left(  j,k\right)  }^{\rho
		}\left(  u\right)  }\\[1.5em]
	&&\cdot\left[
	\begin{array}
	[c]{cc}%
	\phi_{2,2}^{\rho} & -\left\langle \mathbf{t}_{i,\left(  j,k\right)  }\left(
	u\right)  ,\widetilde{\mathbf{t}}_{\left(  j,k\right)  ,i}\left(  u\right)
	\right\rangle \cdot\phi_{1,2}^{\rho}\\
	-\left\langle \mathbf{t}_{i,\left(  j,k\right)  }\left(  u\right)
	,\widetilde{\mathbf{t}}_{\left(  j,k\right)  ,i}\left(  u\right)
	\right\rangle \cdot\phi_{1,2}^{\rho} & \phi_{1,1}^{\rho}%
	\end{array}
	\right]\\[1.25em]
	&  & \cdot\left[
	\begin{array}
	[c]{c}%
	\left\langle \mathbf{p}_{i}\cdot\left(  \phi_{0,1}^{\rho}+\phi_{1,1}^{\rho
	}\right)  +\mathbf{c}_{j,k}\left(  u;\lambda_{j,k},\lambda_{k,j}\right)
	\cdot\left(  \phi_{1,2}^{\rho}+\phi_{1,3}^{\rho}\right)  ,\mathbf{t}%
	_{i,\left(  j,k\right)  }\left(  u\right)  \right\rangle \\
	\left\langle \mathbf{p}_{i}\cdot\left(  \phi_{0,2}^{\rho}+\phi_{1,2}^{\rho
	}\right)  +\mathbf{c}_{j,k}\left(  u;\lambda_{j,k},\lambda_{k,j}\right)
	\cdot\left(  \phi_{2,2}^{\rho}+\phi_{2,3}^{\rho}\right)  ,\widetilde
	{\mathbf{t}}_{\left(  j,k\right)  ,i}\left(  u\right)  \right\rangle
	\end{array}
	\right]
	\end{array}
	\end{equation}
	of the unit tangent vectors $\mathbf{t}_{i,\left(  j,k\right)  }\left(
	u\right)  $ and $\widetilde{\mathbf{t}}_{\left(  j,k\right)  ,i}\left(
	u\right)  $ that are associated with the vertex $\mathbf{p}_{i}$ and the curve
	point $\mathbf{c}_{j,k}\left(  u;\lambda_{j,k},\lambda_{k,j}\right)  $, respectively;
	
	\item for an arbitrarily selected barycentric coordinate $\left(  b_{0}%
	,b_{1},b_{2}\right)  \in\Omega_{1-0_+}$ define the parameter value%
	\[
	u\left(  b_{0},b_{1},b_{2}\right)  =\frac{\beta b_{2}}{1-b_{0}}%
	\]
	and evaluate the patch point%
	\begin{align}
	&  ~\widetilde{\mathbf{s}}_{i,\left(  j,k\right)  }\left(  b_{0},b_{1}%
	,b_{2}\right) \label{eq:interpolant_i_jk}\\
	=  &  ~\mathbf{p}_{i}\cdot B_{3,0}\left(  \beta\left(  1-b_{0}\right)  \right)
	\nonumber\\
	&  ~+\left(  \mathbf{p}_{i}+\widetilde{\lambda}_{i,\left(  j,k\right)
	}\left(  u\left(  b_{0},b_{1},b_{2}\right)  \right)  \cdot \mathbf{t}_{i,\left(
	j,k\right)  }\left(  u\left(  b_{0},b_{1},b_{2}\right)  \right)  \right)
\cdot B_{3,1}\left(  \beta\left(  1-b_{0}\right)  \right) \nonumber\\
&  ~+\left(  \mathbf{c}_{j,k}\left(  u\left(  b_{0},b_{1},b_{2}\right)
;\lambda_{i,j},\lambda_{j,i}\right)  +\widetilde{\lambda}_{\left(  j,k\right)
	,i}\left(  u\left(  b_{0},b_{1},b_{2}\right)  \right) \cdot \widetilde{\mathbf{t}%
}_{\left(  j,k\right)  ,i}\left(  u\left(  b_{0},b_{1},b_{2}\right)  \right)
\right)  \cdot B_{3,2}\left(  \beta\left(  1-b_{0}\right)  \right) \nonumber\\
&  ~+\mathbf{c}_{j,k}\left(  u\left(  b_{0},b_{1},b_{2}\right)  ;\lambda
_{i,j},\lambda_{j,i}\right)  \cdot B_{3,3}\left(  \beta\left(  1-b_{0}\right)
\right)  .\nonumber
\end{align}

\end{itemize}

The construction presented above can also be performed in case of boundary and
opposite vertex pairs $(  \mathbf{c}_{k,i}(  v;\lambda_{k,i},\allowbreak{}\lambda_{i,k})  ,\mathbf{p}_{j})  $ and $\left(  \mathbf{c}%
_{i,j}\left(  w;\lambda_{i,j},\lambda_{i,j}\right)  ,\mathbf{p}_{k}\right)  $,
obtaining the triangular surface patches%
\begin{align}
&  ~\widetilde{\mathbf{s}}_{j,\left(  k,i\right)  }\left(  b_{0},b_{1}%
,b_{2}\right) \label{eq:interpolant_j_ki}\\
=  &  ~\mathbf{p}_{j}\cdot B_{3,0}\left(  \beta\left(  1-b_{1}\right)  \right)
\nonumber\\
&  ~+\left(  \mathbf{p}_{j}+\widetilde{\lambda}_{j,\left(  k,i\right)
}\left(  v\left(  b_{0},b_{1},b_{2}\right)  \right)  \cdot \mathbf{t}_{j,\left(
k,i\right)  }\left(  v\left(  b_{0},b_{1},b_{2}\right)  \right)  \right)
\cdot B_{3,1}\left(  \beta\left(  1-b_{1}\right)  \right) \nonumber\\
&  ~+\left(  \mathbf{c}_{k,i}\left(  v\left(  b_{0},b_{1},b_{2}\right)
;\lambda_{k,i},\lambda_{i,k}\right)  +\widetilde{\lambda}_{\left(  k,i\right)
	,j}\left(  v\left(  b_{0},b_{1},b_{2}\right)  \right)  \cdot \widetilde{\mathbf{t}%
}_{\left(  k,i\right)  ,j}\left(  v\left(  b_{0},b_{1},b_{2}\right)  \right)
\right)  \cdot B_{3,2}\left(  \beta\left(  1-b_{1}\right)  \right) \nonumber\\
&  ~+\mathbf{c}_{k,i}\left(  v\left(  b_{0},b_{1},b_{2}\right)  ;\lambda
_{k,i},\lambda_{i,k}\right)  \cdot B_{3,3}\left(  \beta\left(  1-b_{1}\right)
\right)  \nonumber
\end{align}
and%
\begin{align}
&  ~\widetilde{\mathbf{s}}_{k,\left(  i,j\right)  }\left(  b_{0},b_{1}%
,b_{2}\right) \label{eq:interpolant_k_ij}\\
=  &  ~\mathbf{p}_{k}\cdot B_{3,0}\left(  \beta\left(  1-b_{2}\right)  \right)
\nonumber\\
&  ~+\left(  \mathbf{p}_{k}+\lambda_{k,\left(  i,j\right)  }\left(  w\left(
b_{0},b_{1},b_{2}\right)  \right)  \cdot \mathbf{t}_{k,\left(  i,j\right)  }\left(
w\left(  b_{0},b_{1},b_{2}\right)  \right)  \right)  \cdot B_{3,1}\left(
\beta\left(  1-b_{2}\right)  \right) \nonumber\\
&  ~+\left(  \mathbf{c}_{i,j}\left(  w\left(  b_{0},b_{1},b_{2}\right)
;\lambda_{i,j},\lambda_{j,i}\right)  +\widetilde{\lambda}_{\left(  i,j\right)
	,k}\left(  w\left(  b_{0},b_{1},b_{2}\right)  \right) \cdot \widetilde{\mathbf{t}%
}_{\left(  i,j\right)  ,k}\left(  w\left(  b_{0},b_{1},b_{2}\right)  \right)
\right)  \cdot B_{3,2}\left(  \beta\left(  1-b_{2}\right)  \right) \nonumber\\
&  ~+\mathbf{c}_{i,j}\left(  w\left(  b_{0},b_{1},b_{2}\right)  ;\lambda
_{i,j},\lambda_{j,i}\right)  \cdot B_{3,3}\left(  \beta\left(  1-b_{2}\right)
\right)  ,\nonumber
\end{align}
respectively, where%
\[
v\left(  b_{0},b_{1},b_{2}\right)  =\frac{\beta b_{0}}{1-b_{1}}
\]
and
\[
w\left(  b_{0},b_{1},b_{2}\right)  =\frac{\beta b_{1}}{1-b_{2}}%
\]
for all barycentric coordinates $\left(b_0,b_1,b_2\right)\in\Omega_{1-0_+}$.

Finally, the three side-vertex interpolants
(\ref{eq:interpolant_i_jk}), (\ref{eq:interpolant_j_ki}) and
(\ref{eq:interpolant_k_ij}) can convexly be blended to form a final quasi-optimal triangular
patch that both interpolates the boundary curves and at the same time produces
normal vector fields along these boundaries that are compatible with (i.e.,
point-wise parallel to) the initial corresponding  vector
fields of averaged unit normals used for the description of these interpolants. A commonly used convex
combination \citep{Nielson1979, Nielson1987} of the constructed interpolants is the surface%
\begin{equation}
\label{eq:blended_interpolant}
\begin{array}{rcl}
\widetilde{\mathbf{s}}_{i,j,k}\left(  b_{0},b_{1},b_{2}\right)  &=  &
\omega_{i,\left(  j,k\right)  }\left(  b_{0},b_{1},b_{2}\right)
\cdot\widetilde{\mathbf{s}}_{i,\left(  j,k\right)  }\left(  b_{0},b_{1}%
,b_{2}\right)  +
\omega_{j,\left(  k,i\right)  }\left(  b_{0},b_{1}%
,b_{2}\right)  \cdot\widetilde{\mathbf{s}}_{j,\left(  k,i\right)  }\left(
b_{0},b_{1},b_{2}\right) \\
& &+\omega_{k,\left(  i,j\right)  }\left(  b_{0},b_{1},b_{2}\right)
\cdot\widetilde{\mathbf{s}}_{k,\left(  i,j\right)  }\left(  b_{0},b_{1}%
,b_{2}\right)  ,~\left(  b_{0},b_{1},b_{2}\right)  \in\Omega_{1-0_+},
\end{array}
\end{equation}
where the weight functions $\omega_{i,\left(  j,k\right)  }$, $\omega
_{j,\left(  k,i\right)  }$ and $\omega_{k,\left(  i,j\right)  }$ can be
defined e.g. either as the rational functions
\begin{align*}
\omega_{i,\left(  j,k\right)  }\left(  b_{0},b_{1},b_{2}\right)   &  =
\frac{b_{1}^{2}b_{2}^{2}}{b_{0}^{2}b_{1}^{2}+b_{0}^{2}b_{2}^{2}+b_{1}^{2}b_{2}^{2}},\\
\omega_{j,\left(  k,i\right)  }\left(  b_{0},b_{1},b_{2}\right)   &  =
\frac{b_{0}^{2}b_{2}^{2}}{b_{0}^{2}b_{1}^{2}+b_{0}^{2}b_{2}^{2}+b_{1}^{2}b_{2}^{2}},\\
\omega_{k,\left(  i,j\right)  }\left(  b_{0},b_{1},b_{2}\right)   &  =
\frac{b_{0}^{2}b_{1}^{2}}{b_{0}^{2}b_{1}^{2}+b_{0}^{2}b_{2}^{2}+b_{1}^{2}b_{2}^{2}}
\end{align*}
or as their lower degree counterparts%
\begin{align*}
\omega_{i,\left(  j,k\right)  }\left(  b_{0},b_{1},b_{2}\right)   &  =
\frac{b_{1}b_{2}}{b_{0}b_{1}+b_{0}b_{2}+b_{1}b_{2}},\\
\omega_{j,\left(  k,i\right)  }\left(  b_{0},b_{1},b_{2}\right)   &  =
\frac{b_{0}b_{2}}{b_{0}b_{1}+b_{0}b_{2}+b_{1}b_{2}},\\
\omega_{k,\left(  i,j\right)  }\left(  b_{0},b_{1},b_{2}\right)   &  =
\frac{b_{0}b_{1}}{b_{0}b_{1}+b_{0}b_{2}+b_{1}b_{2}}.
\end{align*}
%
%\clearpage{}
Based on the steps of the proposed method, cases (\textit{b}) and (\textit{c}) of Fig.\ \ref{fig:elephants} show the continuous reflection lines of several quasi-optimal Nielson-type transfinite triangular interpolants that form local interpolating visually smooth triangular spline surfaces that are rough approximations of an elephant. 

\begin{figure}[H]
	\centering
	\includegraphics{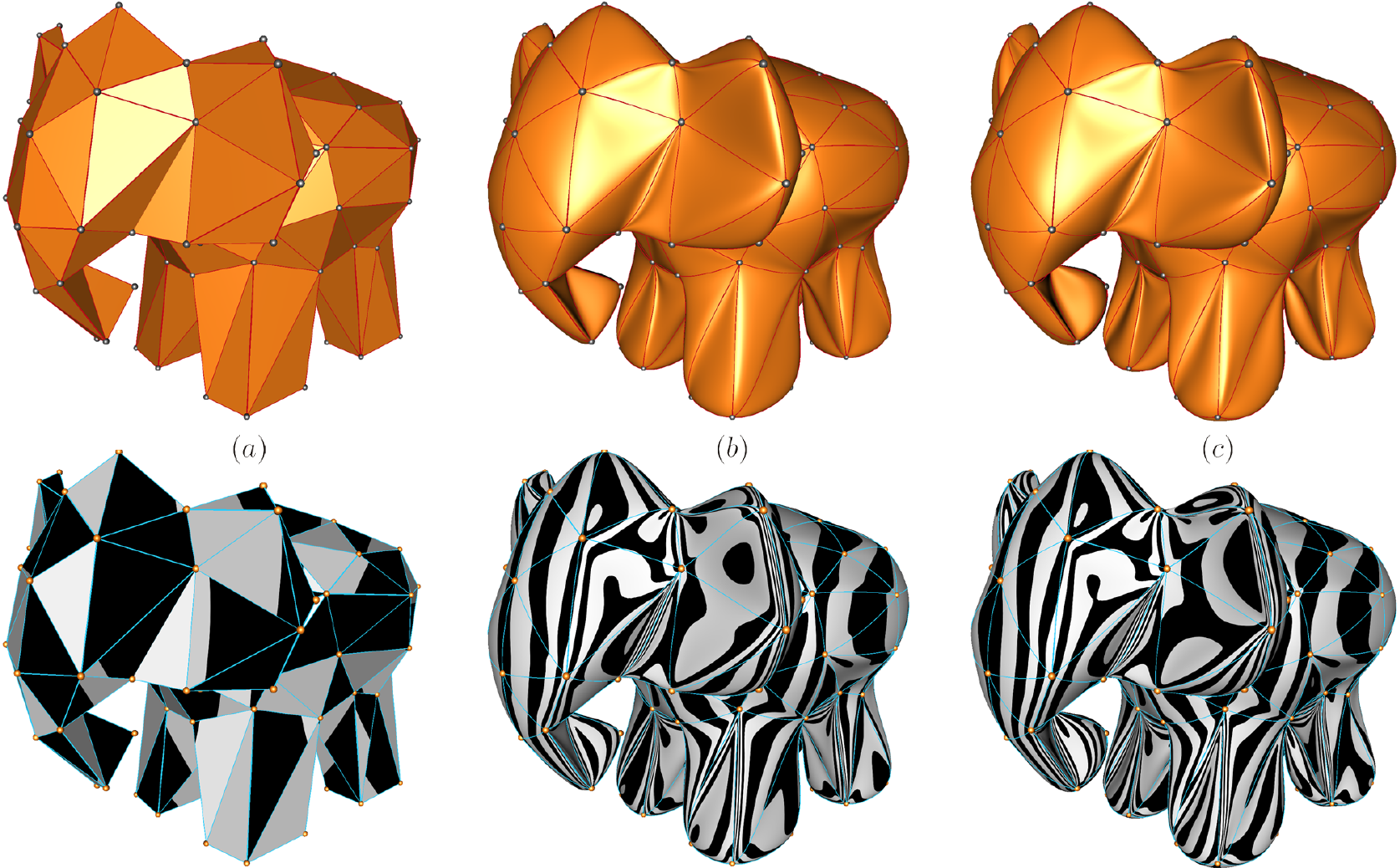}
	\caption{(\textit{a}) An input triangular mesh that resembles an elephant. Outputs (\textit{b}) and (\textit{c}) have continuous reflection lines and were obtained by using the cubic univariate and constrained trivariate Bernstein polynomials with parameter settings $\rho = \gamma = \theta_1 = \varepsilon_1 = 1$ and $\rho = \gamma = 2$, $\theta_1 = \theta_2 = \varepsilon_1 = \varepsilon_2 = 1$, respectively. ($n_v = 102$, $n_f=200$)}
	\label{fig:elephants}
\end{figure}

By applying non-polynomial univariate and constrained trivariate basis functions, one also gains possible shape (or tension) parameters that also affect the output of the proposed method as it is illustrated in cases (\textit{b}), (\textit{c}) and (\textit{d}) of Fig.\ \ref{fig:stellated_dodecahedron}.

\begin{figure}[H]
	\centering
	\includegraphics{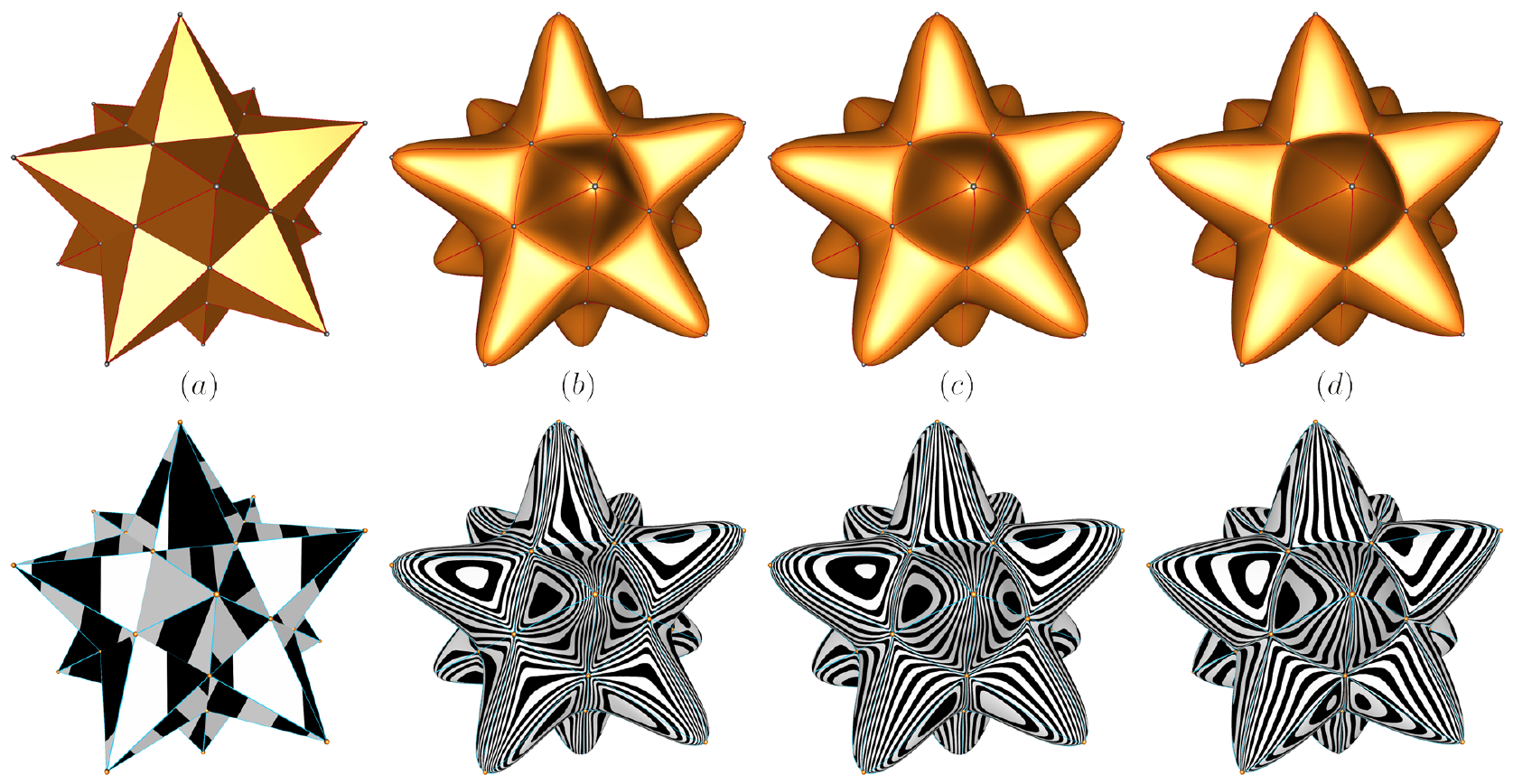}
	\caption{(\textit{a}) A stellated dodecahedron used as the input of the proposed method. Outputs (\textit{b}), (\textit{c}) and (\textit{d}) were obtained by applying the second order trigonometric bases detailed in Examples \ref{exmp:Sanchez} and \ref{exmp:trivariate_trigonometric} with shape (or tension) parameters $\beta = \frac{\pi}{2}$, $\beta = \frac{3\pi}{4}$ and $\beta = 3$, respectively. In case of all outputs we have minimized only the squared variation of first order (partial) derivatives, i.e., $\rho = \gamma = \theta_1 = \varepsilon_1 = 1$. ($n_v = 32$, $n_f=60$)}
	\label{fig:stellated_dodecahedron}
\end{figure}

Starting from the same input triangular mesh, but using different types of univariate and corresponding constrained trivariate basis functions, cases (\textit{b}), (\textit{c}) and (\textit{d}) of Fig.\ \ref{fig:bunnies} show the continuous reflection lines of slightly different $G^1$ continuous spline surfaces formed by quasi-optimal Nielson-type triangular transfinite interpolants.

\begin{figure}[H]
	\centering
	\includegraphics{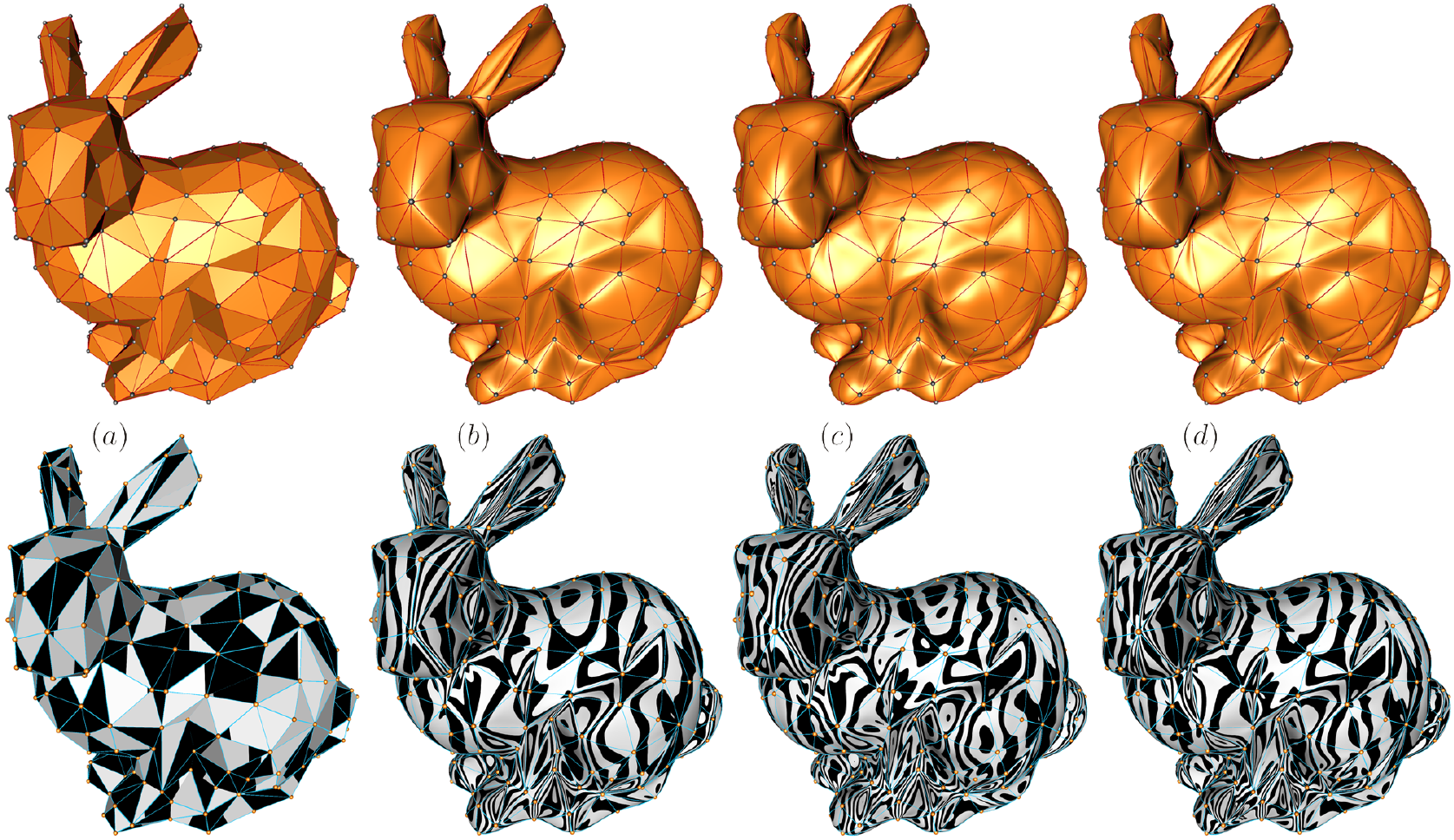}
	\caption{(\textit{a}) An input triangular mesh obtained from the Stanford bunny by means of quadric edge collapse decimation. In case of outputs (\textit{b}), (\textit{c}) and (\textit{d}) we have used the quartic polynomial ($\beta = 1$), the second order trigonometric ($\beta = \frac{3\pi}{4}$) and the first order algebraic-trigonometric ($\beta = \frac{3\pi}{4}$) univariate and constrained trivariate bases detailed in Examples \ref{exmp:Bernstein}/\ref{exmp:trivariate_Bernstein}, \ref{exmp:Sanchez}/\ref{exmp:trivariate_trigonometric} and \ref{exmp:algebraic_trigonometric}/\ref{exmp:trivariate_algebraic_trigonometric}, respectively. In case of all outputs we also used the common parameter settings $\rho = \gamma = 2$, $\theta_1 = \theta_2 = \varepsilon_1 = \varepsilon_2 = 1$. ($n_v = 276$, $n_f=536$)}
	\label{fig:bunnies}
\end{figure}

\vspace{-0.5cm}
Using univariate and constrained trivariate second order trigonometric basis functions, Fig.\ \ref{fig:horse_and_camel} presents further examples.

\begin{figure}[H]
	\centering
	\includegraphics{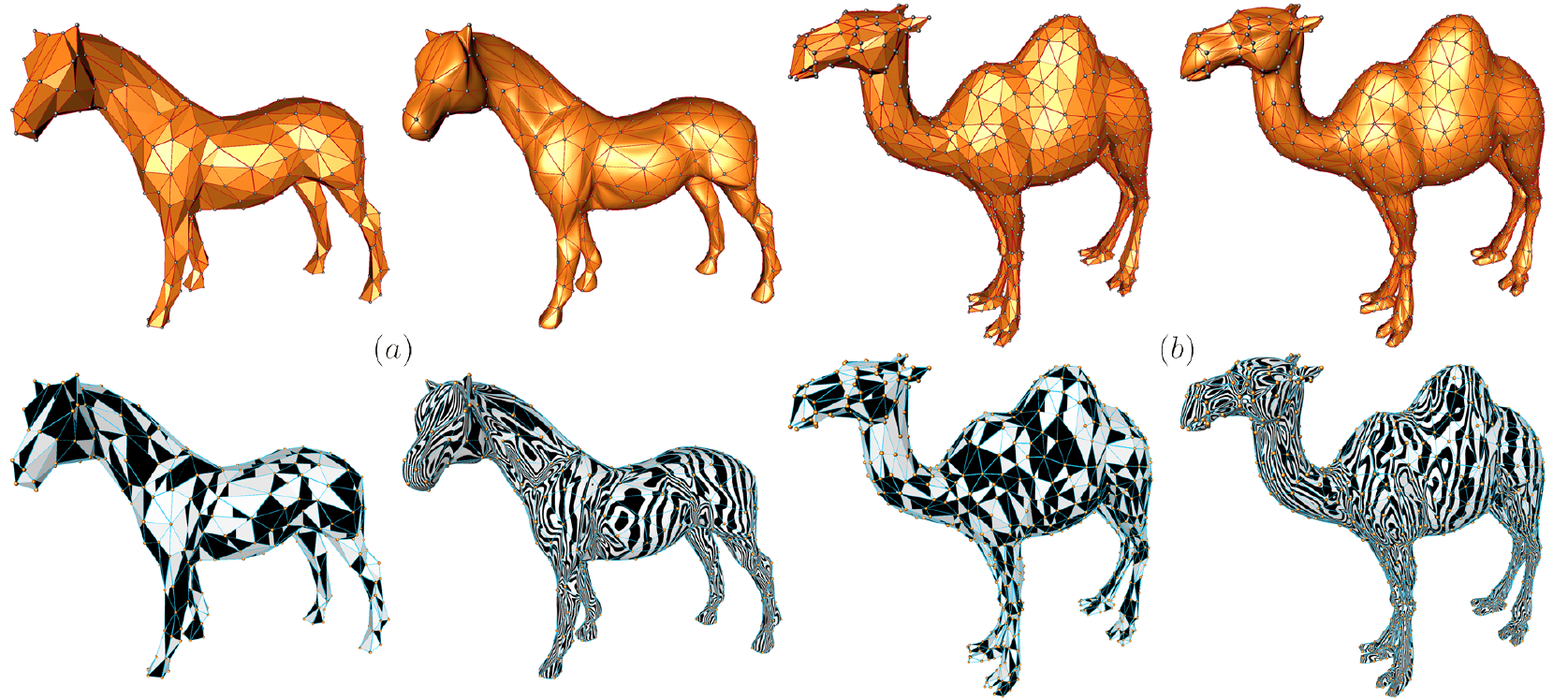}
	\caption{In both cases (\textit{a}) and (\textit{b}) we have used the same univariate and constrained trivariate second order trigonometric basis functions described in Examples \ref{exmp:Sanchez} and \ref{exmp:trivariate_trigonometric}, respectively. The remaining parameter settings were also the same, i.e., $\beta=\frac{\pi}{2}$, $\rho = \gamma = 2$, $\theta_1 = \theta_2 = \varepsilon_1 = \varepsilon_2 = 1$. In case (\textit{a}) $n_v = 341$ and $n_f = 678$, while in case  (\textit{b}) $n_v = 959$ and $n_f = 1914$.}
	\label{fig:horse_and_camel}
\end{figure}

\section{Final remarks}\label{sec:final_remarks}

Quadratic energy functionals like the first and higher order strain or thin-plate spline energies (\ref{eq:strain_energy}) and (\ref{eq:thin_plate_spline}), respectively, are well-known concepts in geometric modeling. In order to ensure more optimal shape preserving properties, we have expressed the strain energy in a slightly more general context by using the unique non-negative normalized B-bases of arbitrary (not necessarily polynomial) reflection invariant EC spaces that also comprise the constants,  while in case of the thin-plate spline energy we have assumed the existence of constrained trivariate non-negative normalized basis functions that are compatible with (i.e., degenerate to) those univariate basis functions along the edges of their triangular definition domain that were constructed by means of non-negative normalized B-basis functions. Concerning the minimization of these energies, we have shown in general that their unique optimum points always exist and we have also provided explicit closed formulas like (\ref{eq:solution}) or (\ref{eq:solution_surface}) for their evaluation. As possible practical applications, we have used special univariate and constrained trivariate polynomial, trigonometric, hyperbolic and algebraic-trigonometric normalized non-negative basis functions. Following these examples, one can easily define new variants of the proposed algorithm. 

The single input of the proposed $G^1$ continuous local interpolating quasi-optimal triangular spline surface modeling tool is a triangular mesh stored in a half-edge data structure. In order to construct a network of piecewise optimal arcs that locally interpolate the vertices of the given mesh, strain energy functionals were locally optimized along the input edges. The obtained network of curves is formed by curvilinear triangles that were fitted by local interpolating piecewise optimal triangular surface patches that were obtained by locally minimizing thin-plate spline energies. In general, these triangular surface patches are only $C^0$ continuous along the common boundary curves, but they can be used to define  continuous quasi-optimal vector fields of averaged unit normals along the shared boundaries. Thus, we have arrived to the last step of our algorithm, namely we generalize the concept of the visually smooth transfinite triangular interpolant of \citep{Nielson1987} by imposing further optimality constraints concerning the isoparametric lines of those groups of three side-vertex triangular interpolants that have to be convexly blended in order to generate the final visually smooth local interpolating  quasi-optimal triangular spline surface.  

After processing the connectivity information stored in the input triangular mesh and creating lookup tables with the values of univariate/double integrals listed in Appendices \ref{app:univariate_integrals}/\ref{app:double_integrals}, all intermediate steps of the proposed triangular spline surface generation technique are highly parallelizable and can efficiently be evaluated by means of explicit formulas (\ref{eq:solution}), (\ref{eq:solution_surface}), (\ref{eq:vector_field_of_averaged_normals})/(\ref{eq:inherited_vector_fieldnormals}), (\ref{eq:point-wise_solution}), (\ref{eq:interpolant_i_jk})--(\ref{eq:interpolant_k_ij}) and (\ref{eq:blended_interpolant}). 

\vspace{-0.35cm}
\begin{small}
\setlength{\bibsep}{0pt plus 0.3ex}

\end{small}
\clearpage{}
\appendix
%\normalsize

\renewcommand*{\thesection}{\Alph{section}}

\section{Integrals of univariate expressions}\label{app:univariate_integrals}

Potentially omitted values in this section either are equal to zero or do not affect the output of the proposed method.

\subsection{Integrals of univariate polynomial expressions}

The values listed in Sections \ref{app:cubic_Bernstein_r_1}--\ref{app:cubic_Bernstein_r_2} belong to the univariate cubic Bernstein polynomials detailed in Example \ref{exmp:Bernstein}.

\subsubsection{Variations of products of first order derivatives}\label{app:cubic_Bernstein_r_1}

\begin{footnotesize}
	\[
	\varphi_{0,1}^1 = \varphi_{2,3}^1 = -\frac{9}{10},~~ %\\
	\varphi_{0,2}^1 = \varphi_{1, 3}^1 = -\frac{3}{5},~~ %\\
	\varphi_{1,1}^1 = \varphi_{2,2}^1 = \frac{6}{5},~~ %\\
	\varphi_{1,2}^1 = \frac{3}{10}.
	\]
\end{footnotesize}

\subsubsection{Variations of products of second order derivatives}\label{app:cubic_Bernstein_r_2}

\begin{footnotesize}
	\[
	\varphi_{0,1}^2 = \varphi_{1,2}^2 = \varphi_{2,3}^2 =  -18,~~
	\varphi_{0,2}^2 = \varphi_{1, 3}^2 = 0,~~
	\varphi_{1,1}^2 = \varphi_{2,2}^2 = 36.
	\]
\end{footnotesize}

\noindent{}The values listed in Sections \ref{app:quartic_Bernstein_r_1}--\ref{app:quartic_Bernstein_r_2} correspond to the univariate quartic polynomial basis detailed in Example \ref{exmp:Bernstein}.

\subsubsection{Variations of products of first order derivatives}\label{app:quartic_Bernstein_r_1}

\begin{footnotesize}
	\[
	\varphi_{0,1}^1 = \varphi_{2,3}^1 = -\frac{52}{35},~~ %\\
	\varphi_{0,2}^1 = \varphi_{1, 3}^1 = -\frac{24}{35},~~ %\\
	\varphi_{1,1}^1 = \varphi_{2,2}^1 = \frac{66}{35},~~ %\\
	\varphi_{1,2}^1 = \frac{2}{7}.
	\]
\end{footnotesize}

\subsubsection{Variations of products of second order derivatives}\label{app:quartic_Bernstein_r_2}

\begin{footnotesize}
	\[
	\varphi_{0,1}^2 = \varphi_{2,3}^2 = -\frac{204}{5},~~
	\varphi_{0,2}^2 = \varphi_{1, 3}^2 = \frac{36}{5},~~
	\varphi_{1,1}^2 = \varphi_{2,2}^2 = \frac{324}{5},~~
	\varphi_{1,2}^2 = -\frac{156}{5}.
	\]
\end{footnotesize}

\subsection{Integrals of univariate trigonometric expressions}

Let $\beta\in\left(  0,\pi\right)  $ an 
arbitrarily fixed parameter and consider the univariate trigonometric basis functions introduced in Example \ref{exmp:Sanchez}. The corresponding variations of products of first and second order derivatives are listed in Sections \ref{app:trigonometric_r_1} and \ref{app:trigonometric_r_2}, respectively.%

\subsubsection{Variations of products of first order derivatives}\label{app:trigonometric_r_1}

\begin{footnotesize}
	\begin{align*}
	\varphi_{0,1}^{1}  &  =\varphi_{2,3}^{1}={\frac{\left(  24+4\cos\left(
			\beta\right)  -18\cos^{2}\left(  \beta\right)  +5\cos^{3}\left(  \beta\right)
			\right)  \sin\left(  \beta\right)  -3\beta\left(  6+2\cos\left(  \beta\right)
			-3\cos^{2}\left(  \beta\right)  \right)  }{96\sin^{8}\left(  \frac{\beta}%
			{2}\right)  },}\\
	\varphi_{0,2}^{1}  &  =\varphi_{1,3}^{1}={\frac{\left(  -12+24\cos\left(
			\beta\right)  +2\cos^{2}\left(  \beta\right)  +\cos^{3}\left(  \beta\right)
			\right)  \sin\left(  \beta\right)  +3\beta\left(  2-2\cos\left(  \beta\right)
			-5\cos^{2}\left(  \beta\right)  \right)  }{96\sin^{8}\left(  \frac{\beta}%
			{2}\right)  },}\\
	\varphi_{1,1}^{1}  &  =\varphi_{2,2}^{1}={\frac{\left(  -32-12\cos\left(
			\beta\right)  +34\cos^{2}\left(  \beta\right)  -5\cos^{3}\left(  \beta\right)
			\right)  \sin\left(  \beta\right)  +3\beta\left(  8+4\cos\left(  \beta\right)
			-5\cos^{2}\left(  \beta\right)  -2\cos^{3}\left(  \beta\right)  \right)
		}{6\left(  1-4\,\cos\left(  \beta\right)  +6\,\cos^{2}\left(  \beta\right)
		-4\,\cos^{3}\left(  \beta\right)  +\cos^{4}\left(  \beta\right)  \right)  }%
	,}\\
\varphi_{1,2}^{1}  &  ={\frac{\left(  20-16\cos\left(  \beta\right)
		-18\cos^{2}\left(  \beta\right)  -\cos^{3}\left(  \beta\right)  \right)
		\sin\left(  \beta\right)  -3\beta\left(  4-7\cos^{2}\left(  \beta\right)
		-2\cos^{3}\left(  \beta\right)  \right)  }{6\left(  1-4\,\cos\left(
		\beta\right)  +6\,\cos^{2}\left(  \beta\right)  -4\,\cos^{3}\left(
		\beta\right)  +\cos^{4}\left(  \beta\right)  \right)  }}.
\end{align*}
\end{footnotesize}

\subsubsection{Variations of products of second order derivatives}\label{app:trigonometric_r_2}

\begin{footnotesize}
	\begin{align*}
	\varphi_{0,1}^{2}&=\varphi_{2,3}^{2}={\frac{\left(  12+5\cos\left(
			\beta\right) +9\cos^{2}\left(  \beta\right) -14\cos^{3}\left(  \beta\right)  
			\right)  \sin\left(  \beta\right)  -3\beta\left(  6+\cos\left(  \beta\right)
			-3\cos^{2}\left(  \beta\right)  \right)  }{3\left(  1-4\,\cos\left(
			\beta\right)  +6\cos^{2}\left(  \beta\right)  -4\cos^{3}\left(  \beta\right)
			+\cos^{4}\left(  \beta\right)  \right)  }},
	\\
	\varphi_{0,2}^{2}&=\varphi_{1,3}^{2}={\frac{\left(  -18+21\cos\left(
			\beta\right)  +7\cos^{2}\left(  \beta\right)  +2\cos^{3}\left(  \beta\right)
			\right)  \sin\left(  \beta\right)  +3\beta\left(  4-\cos\left(  \beta\right)
			-7\cos^{2}\left(  \beta\right)  \right)  }{3\left(  1-4\,\cos\left(
			\beta\right)  +6\cos^{2}\left(  \beta\right)  -4\cos^{3}\left(  \beta\right)
			+\cos^{4}\left(  \beta\right)  \right)  }},
	\\
	\varphi_{1,1}^{2}&=\varphi_{2,2}^{2}={\frac{\left(  -16-24\cos\left(
			\beta\right)  +11\cos^{2}\left(  \beta\right)  +17\cos^{3}\left(
			\beta\right)  \right)  \sin\left(  \beta\right)  +3\beta\left(  10-7\cos
			^{2}\left(  \beta\right)  +2\cos\left(  \beta\right)  -\cos^{3}\left(
			\beta\right)  \right)  }{3\left(  1-4\,\cos\left(  \beta\right)  +6\cos
			^{2}\left(  \beta\right)  -4\cos^{3}\left(  \beta\right)  +\cos^{4}\left(
			\beta\right)  \right)  }},
	\\
	\varphi_{1,2}^{2}&={\frac{\left(  22-2\cos\left(  \beta\right)  -27\cos
			^{2}\left(  \beta\right)  -5\cos^{3}\left(  \beta\right)  \right)  \sin\left(
			\beta\right)  -3\beta\left(  8-11\cos^{2}\left(  \beta\right)  -\cos
			^{3}\left(  \beta\right)  \right)  }{3\left(  1-4\,\cos\left(  \beta\right)
			+6\cos^{2}\left(  \beta\right)  -4\cos^{3}\left(  \beta\right)  +\cos
			^{4}\left(  \beta\right)  \right)  }}.
	\end{align*}
\end{footnotesize}

\subsection{Integrals of univariate hyperbolic expressions}

Let $\beta>0$ be an arbitrarily fixed tension parameter. Using the hyperbolic basis functions presented in Example \ref{exmp:hyperbolic}, the variations of products of first and second order derivatives are listed in Sections \ref{app:hyperbolic_r_1} and \ref{app:hyperbolic_r_2}, respectively.

\subsubsection{Variations of products of first order derivatives}\label{app:hyperbolic_r_1}

\begin{footnotesize}
	\begin{align*}
	\varphi_{0,1}^{1}   =\varphi_{2,3}^{1}=&~\frac{48\beta\left(  6+2\cosh\left(
		\beta\right)  -3\cosh^{2}\left(  \beta\right)  \right)  -312\sinh\left(
		\beta\right)  -52\sinh\left(  2\beta\right)  +72\sinh\left(  3\beta\right)
		-10\sinh\left(  4\beta\right)  }{1536\sinh^{8}\left(  \frac{\beta}{2}\right)
	},\\
	\varphi_{0,2}^{1}   =\varphi_{1,3}^{1}=&~\frac{-48\beta\left(  2-2\cosh\left(
		\beta\right)  -5\cosh^{2}\left(  \beta\right)  \right)  +184\sinh\left(
		\beta\right)  -196\sinh\left(  2\beta\right)  -8\sinh\left(  3\beta\right)
		-2\sinh\left(  4\beta\right)  }{1536\sinh^{8}\left(  \frac{\beta}{2}\right)
	},\\
	\varphi_{1,1}^{1}   =\varphi_{2,2}^{1}=&~\frac{-48\beta\left(  8+4\cosh\left(
		\beta\right)  -5\cosh^{2}\left(  \beta\right)  -2\cosh^{3}\left(
		\beta\right)  \right)  +376\sinh\left(  \beta\right)  +116\sinh\left(
		2\beta\right)  -136\sinh\left(  3\beta\right)  +10\sinh\left(  4\beta\right)
	}{96\left(  1-4\cosh\left(  \beta\right)  +6\cosh^{2}\left(  \beta\right)
	-4\cosh^{3}\left(  \beta\right)  +\cosh^{4}\left(  \beta\right)  \right)  },\\
\varphi_{1,2}^{1}   =&~{\frac{48\beta\left(  4-7\cosh^{2}\left(  \beta\right)
		-2\cosh^{3}\left(  \beta\right)  \right)  -248\mathrm{\sinh}\left(
		\beta\right)  +132\sinh\left(  2\beta\right)  +72\sinh\left(  3\beta\right)
		+2\sinh\left(  4\beta\right)  }{96\left(  1-4\cosh\left(  \beta\right)
		+6\cosh^{2}\left(  \beta\right)  -4\cosh^{3}\left(  \beta\right)  +\cosh
		^{4}\left(  \beta\right)  \right)  }}.
\end{align*}
\end{footnotesize}

\subsubsection{Variations of products of second order derivatives}\label{app:hyperbolic_r_2}

\begin{footnotesize}
\begin{align*}
\varphi_{0,1}^{2}   =\varphi_{2,3}^{2}=&~{\frac{-24\beta\left(  6+\cosh\left(
		\beta\right)  -3\cosh^{2}\left(  \beta\right)  \right)  +114\sinh\left(
		\beta\right)  -8\sinh\left(  2\beta\right)  +18\sinh\left(  3\beta\right)
		-14\sinh\left(  4\beta\right)  }{24\left(  1-4\cosh\left(  \beta\right)
		+6\cosh^{2}\left(  \beta\right)  -4\cosh^{3}\left(  \beta\right)  +\cosh
		^{4}\left(  \beta\right)  \right)  },}\\
\varphi_{0,2}^{2}   =\varphi_{1,3}^{2}=&~\frac{24\beta\left(  4-\cosh\left(
	\beta\right)  -7\cosh^{2}\left(  \beta\right)  \right)  -130\sinh\left(
	\beta\right)  +88\sinh\left(  2\beta\right)  +14\sinh\left(  3\beta\right)
	+2\sinh\left(  4\beta\right)  }{24\left(  1-4\cosh\left(  \beta\right)
	+6\cosh^{2}\left(  \beta\right)  -4\cosh^{3}\left(  \beta\right)  +\cosh
	^{4}\left(  \beta\right)  \right)  },\\
\varphi_{1,1}^{2}   =\varphi_{2,2}^{2}=&~\frac{48\beta\left(  10+2\cosh\left(
	\beta\right)  -7\cosh^{2}\left(  \beta\right)  -\cosh^{3}\left(  \beta\right)
	\right)  -212\sinh\left(  \beta\right)  -124\sinh\left(  2\beta\right)
	+44\sinh\left(  3\beta\right)  +34\sinh\left(  4\beta\right)  }{48\left(
	1-4\cosh\left(  \beta\right)  +6\cosh^{2}\left(  \beta\right)  -4\cosh
	^{3}\left(  \beta\right)  +\cosh^{4}\left(  \beta\right)  \right)  },\\
\varphi_{1,2}^{2}  &~ =\frac{-48\beta\left(  8-11\cosh^{2}\left(  \beta\right)
	-\cosh^{3}\left(  \beta\right)  \right)  +244\sinh\left(  \beta\right)
	-36\sinh\left(  2\beta\right)  -108\sinh\left(  3\beta\right)  -10\sinh\left(
	4\beta\right)  }{48\left(  1-4\cosh\left(  \beta\right)  +6\cosh^{2}\left(
	\beta\right)  -4\cosh^{3}\left(  \beta\right)  +\cosh^{4}\left(  \beta\right)
	\right)  }.
\end{align*}
\end{footnotesize}

\subsection{Integrals of univariate algebraic-trigonometric expressions}

Let $\beta \in \left(0, 2\pi\right)$ a fixed tension parameter and consider the univariate algebraic-trigonometric basis functions described in Example \ref{exmp:algebraic_trigonometric}. In this case, the variations of products of the first and second order derivatives are listed in Sections \ref{app:algebraic_trigonometric_r_1} and \ref{app:algebraic_trigonometric_r_2}, respectively.

\subsubsection{Variations of products of first order derivatives}\label{app:algebraic_trigonometric_r_1}

\begin{footnotesize}
	\begin{align*}
	\varphi_{0,1}^{1}=\varphi_{2,3}^{1}= &  ~\frac{\left(  -3\beta+6\sin\left(
		\beta\right)  +2\beta\cos\left(  \beta\right)  -3\sin\left(  2\beta\right)
		+\beta\cos\left(  2\beta\right)  \right)  \sin\left(  \beta\right)  }{4\left(
		2\sin\left(  \beta\right)  -\beta-\beta\cos\left(  \beta\right)  \right)
		\left(  \beta-\sin\left(  \beta\right)  \right)  ^{2}},\\
	\varphi_{0,2}^{1}=\varphi_{1,3}^{1}= &  ~\frac{\left(  -\beta-\sin\left(
		\beta\right)  +\beta^{2}\sin\left(  \beta\right)  +\beta\cos\left(
		\beta\right)  +\cos\left(  \beta\right)  \sin\left(  \beta\right)  \right)
		\sin\left(  \beta\right)  }{2\left(  2\sin\left(  \beta\right)  -\beta
		-\beta\cos\left(  \beta\right)  \right)  \left(  \beta-\sin\left(
		\beta\right)  \right)  ^{2}},\\
	\varphi_{1,1}^{1}=\varphi_{2,2}^{1}= &  ~\frac{\left(  2\beta^{3}-4\sin\left(
		\beta\right)  -4\beta^{2}\sin\left(  \beta\right)  +4\beta\cos\left(  \beta\right)
		+2\sin\left(  2\beta\right)  -\beta^{2}\sin\left(  2\beta\right)  -4\beta
		\cos\left(  2\beta\right)  \right)  \sin^{2}\left(  \beta\right)  }{4\left(
		2\sin\left(  \beta\right)  -\beta-\beta\cos\left(  \beta\right)  \right)
		^{2}\left(  \beta-\sin\left(  \beta\right)  \right)  ^{2}},\\
	\varphi_{1,2}^{1}   =&~\frac{\left(  6\beta-2\sin\left(  \beta\right)
		-3\beta^{2}\sin\left(  \beta\right)  -6\beta\cos\left(  \beta\right)
		+\beta^{3}\cos\left(  \beta\right)  +\sin\left(  2\beta\right)  \right)
		\sin^{2}\left(  \beta\right)  }{2\left(  2\sin\left(  \beta\right)
		-\beta-\beta\cos\left(  \beta\right)  \right)  ^{2}\left(  \beta-\sin\left(
		\beta\right)  \right)  ^{2}}\\
	\end{align*}
\end{footnotesize}

\subsubsection{Variations of products of second order derivatives}\label{app:algebraic_trigonometric_r_2}

\begin{footnotesize}
	\begin{align*}
	\varphi_{0,1}^{2}=\varphi_{2,3}^{2}= &  ~\frac{\left(  -\beta-2\sin\left(
		\beta\right)  +2\beta\cos\left(  \beta\right)  +\sin\left(  2\beta\right)
		-\beta\cos\left(  2\beta\right)  \right)  \sin\left(  \beta\right)  }{4\left(
		2\sin\left(  \beta\right)  -\beta-\beta\cos\left(  \beta\right)  \right)
		\left(  \beta-\sin\left(  \beta\right)  \right)  ^{2}},\\
	\varphi_{0,2}^{2}=\varphi_{1,3}^{2}= &  ~\frac{\left(  -\beta-\sin\left(
		\beta\right)  +\beta^{2}\sin\left(  \beta\right)  +\beta\cos\left(
		\beta\right)  +\cos\left(  \beta\right)  \sin\left(  \beta\right)  \right)
		\sin\left(  \beta\right)  }{2\left(  2\sin\left(  \beta\right)  -\beta
		-\beta\cos\left(  \beta\right)  \right)  \left(  \beta-\sin\left(
		\beta\right)  \right)  ^{2}},\\
	\varphi_{1,1}^{2}=\varphi_{2,2}^{2}= &  ~\frac{\left(  2\beta+2\beta^{3}%
		+4\sin\left(  \beta\right)  -4\beta^{2}\sin\left(  \beta\right)  -4\beta
		\cos\left(  \beta\right)  -2\sin\left(  2\beta\right)  +\beta^{2}\sin\left(
		2\beta\right)  +2\beta\cos\left(  2\beta\right)  \right)  \sin^{2}\left(
		\beta\right)  }{4\left(  2\sin\left(  \beta\right)  -\beta-\beta\cos\left(
		\beta\right)  \right)  ^{2}\left(  \beta-\sin\left(  \beta\right)  \right)
		^{2}},\\
	\varphi_{1,2}^{2}= &  ~\frac{\left(  \beta+2\sin\left(  \beta\right)
		-\beta^{2}\sin\left(  \beta\right)  -2\beta\cos\left(  \beta\right)
		+\beta^{3}\cos\left(  \beta\right)  -\sin\left(  2\beta\right)  +\beta
		\cos\left(  2\beta\right)  \right)  \sin^{2}\left(  \beta\right)  }{2\left(
		2\sin\left(  \beta\right)  -\beta-\beta\cos\left(  \beta\right)  \right)
		^{2}\left(  \beta-\sin\left(  \beta\right)  \right)  ^{2}}.
	\end{align*}
\end{footnotesize}

\section{Double integrals of constrained trivariate expressions}\label{app:double_integrals}

Potentially omitted values in this section either are equal to zero or do not affect the output of the proposed method.

\subsection{Double integrals of constrained trivariate  polynomial expressions}\label{app:polynomial}

The values listed in Sections \ref{app:cubic_Bernstein_g_1}--\ref{app:cubic_Bernstein_g_2} correspond to the constrained trivariate cubic Bernstein polynomials detailed in Example \ref{exmp:trivariate_Bernstein}.

\subsubsection{Variations of products of first order partial derivatives}\label{app:cubic_Bernstein_g_1}

\begin{footnotesize}
	\begin{align*}
	\tau_{0,0,3}^{0,1}&=
	\tau_{0,0,3}^{1,0}=
	\tau_{0,1,2}^{1,0}=
	\tau_{0,3,0}^{0,1}=
	\tau_{1,0,2}^{0,1}=
	\tau_{1,2,0}^{0,1}=
	\tau_{2,1,0}^{1,0}=
	\tau_{3,0,0}^{1,0}=-\frac{1}{10},\\
	\tau_{0,1,2}^{0,1}&=\tau_{0,2,1}^{0,1}=\tau_{1,0,2}^{1,0}=\tau_{2,0,1}^{1,0}=\frac{1}{10},\\
	\tau_{1,1,1}^{0,1}&=\tau_{1,1,1}^{1,0}=\frac{1}{5}.
	\end{align*}
\end{footnotesize}

\subsubsection{Variations of products of second order partial derivatives}\label{app:cubic_Bernstein_g_2}

\begin{footnotesize}
	\begin{align*}
     \tau^{0,2}_{0,0,3} &= \tau^{2,0}_{0,0,3} = \tau^{1,1}_{0,1,2} = \tau^{0,2}_{0,3,0} = \tau^{1,1}_{1,0,2} = \tau^{1,1}_{1,2,0} = \tau^{1,1}_{2,1,0} = \tau^{2,0}_{3,0,0} = -3,
     \\
     \tau^{1,1}_{0,0,3} &=  \tau^{2,0}_{0,2,1} = \tau^{1,1}_{0,3,0} = \tau^{2,0}_{0,3,0} = \tau^{2,0}_{1,2,0} = \tau^{0,2}_{2,0,1} = \tau^{0,2}_{2,1,0} = \tau^{0,2}_{3,0,0} = \tau^{1,1}_{3,0,0} = 0,
     \\
     \tau^{0,2}_{0,1,2} &= \tau^{0,2}_{0,2,1} = \tau^{2,0}_{1,0,2} = \tau^{1,1}_{0,2,1} = \tau^{1,1}_{2,0,1} = \tau^{2,0}_{2,0,1} = 3,
     \\
     \tau^{2,0}_{0,1,2} &= \tau^{0,2}_{1,0,2} = \tau^{0,2}_{1,2,0} = \tau^{2,0}_{2,1,0} = -6,
     \\
     \tau^{0,2}_{1,1,1} &= \tau^{2,0}_{1,1,1} =  12,
     \\
     \tau^{1,1}_{1,1,1} &=   6.
	\end{align*}
\end{footnotesize}

\noindent{}The values listed in Sections \ref{app:quartic_Bernstein_g_1}--\ref{app:quartic_Bernstein_g_2} are related to the constrained trivariate quartic polynomial basis functions detailed in Example \ref{exmp:trivariate_Bernstein}.

\subsubsection{Variations of products of first order partial derivatives}\label{app:quartic_Bernstein_g_1}

\begin{footnotesize}
\begin{align*}
\tau_{0,0,3}^{0,1}  &  =\tau_{0,0,3}^{1,0}=\tau_{0,3,0}^{0,1}=\tau
_{3,0,0}^{1,0}=-{\frac{6}{35},}\\
\tau_{0,1,2}^{0,1}  &  =\tau_{0,2,1}^{0,1}=\tau_{1,0,2}^{1,0}=\tau
_{2,0,1}^{1,0}={\frac{6}{35},}\\
\tau_{0,1,2}^{1,0}  &  =\tau_{1,0,2}^{0,1}=\tau_{1,2,0}^{0,1}=\tau
_{2,1,0}^{1,0}=-{\frac{11}{35},}\\
\tau_{0,2,1}^{1,0}  &  =\tau_{1,2,0}^{1,0}=\tau_{2,0,1}^{0,1}=\tau
_{2,1,0}^{0,1}=-{\frac{3}{35},}\\
\tau_{0,3,0}^{1,0}  &  =\tau_{3,0,0}^{0,1}=0,\\
\tau_{1,1,1}^{0,1}  &  =\tau_{1,1,1}^{1,0}=4/5.
\end{align*}
\end{footnotesize}

\subsubsection{Variations of products of second order partial derivatives}\label{app:quartic_Bernstein_g_2}

\begin{footnotesize}
\begin{align*}
\tau_{0,0,3}^{0,2}  & =\tau_{0,0,3}^{2,0}=\tau_{0,3,0}^{0,2}=\tau
_{3,0,0}^{2,0}=-{\frac{24}{5},}\\
\tau_{0,0,3}^{1,1}  & ={\frac{12}{5},}\\
\tau_{0,1,2}^{0,2}  & =\tau_{0,2,1}^{0,2}=\tau_{0,2,1}^{1,1}=\tau
_{1,0,2}^{2,0}=\tau_{2,0,1}^{1,1}=\tau_{2,0,1}^{2,0}={\frac{24}{5},}\\
\tau_{0,1,2}^{1,1}  & =\tau_{0,2,1}^{2,0}=\tau_{1,0,2}^{1,1}=\tau
_{1,2,0}^{2,0}=\tau_{2,0,1}^{0,2}=\tau_{2,1,0}^{0,2}=-{\frac{36}{5},}\\
\tau_{0,1,2}^{2,0}  & =\tau_{1,0,2}^{0,2}=\tau_{1,2,0}^{0,2}=\tau
_{2,1,0}^{2,0}=-{\frac{84}{5},}\\
\tau_{0,3,0}^{1,1}  & =\tau_{0,3,0}^{2,0}=\tau_{3,0,0}^{0,2}=\tau
_{3,0,0}^{1,1}=0,\\
\tau_{1,1,1}^{0,2}  & =2\tau_{1,1,1}^{1,1}=\tau_{1,1,1}^{2,0}=48,\\
\tau_{1,2,0}^{1,1}  & =\tau_{2,1,0}^{1,1}=-{\frac{54}{5}}.
\end{align*}
\end{footnotesize}

\subsection{Double integrals of constrained trivariate trigonometric expressions}\label{app:trigonometric}

Using the design parameter $\beta \in \left(0,\pi\right)$, the values of Sections \ref{app:trigonometric_g_1}--\ref{app:trigonometric_g_2} correspond to the constrained trivariate trigonometric basis functions detailed in Example \ref{exmp:trivariate_trigonometric}.

\subsubsection{Variations of products of first order partial derivatives}\label{app:trigonometric_g_1}

\begin{footnotesize}%
	\begin{align*}
	&  ~\tau_{0,0,3}^{0,1}=\tau_{0,0,3}^{1,0}=\tau_{0,3,0}^{0,1}=\tau
	_{3,0,0}^{1,0}=\\
	= &  ~\frac{1}{1152\sin^{8}\left(  \frac{\beta}{2}\right)  }\left(
	384-81{\beta}^{2}-2330\cos^{2}\left(  \frac{\beta}{2}\right)  +288{\beta}%
	^{2}\cos^{2}\left(  \frac{\beta}{2}\right)  +1770\cos^{4}\left(  \frac{\beta
	}{2}\right)  -180{\beta}^{2}\cos^{4}\left(  \frac{\beta}{2}\right)  \right.
	\\
	&  ~\left.  +312\cos^{6}\left(  \frac{\beta}{2}\right)  -136\cos^{8}\left(
	\frac{\beta}{2}\right)  -3\beta\cos\left(  \frac{\beta}{2}\right)  \sin\left(
	\frac{\beta}{2}\right)  \left(  37-286\cos^{2}\left(  \frac{\beta}{2}\right)
	\right)  \right)  ,\\
	& \\
	&  ~\tau_{0,1,2}^{1,0}=\tau_{1,0,2}^{0,1}=\tau_{1,2,0}^{0,1}=\tau
	_{2,1,0}^{1,0}=\\
	= &  ~{\frac{1}{2304\sin^{8}\left(  \frac{\beta}{2}\right)  }}\left(
	{-27{\beta}^{2}+2\left(  890+18{\beta}^{2}\right)  \cos^{2}\left(  \frac
		{\beta}{2}\right)  -180\left(  1+{\beta}^{2}\right)  \cos^{4}\left(
		\frac{\beta}{2}\right)  +48\left(  3{\beta}^{2}-31\right)  \cos^{6}\left(
		\frac{\beta}{2}\right)  }\right.  \\
	&  ~\left.  {-112\cos^{8}\left(  \frac{\beta}{2}\right)  -12\beta\sin\left(
		\frac{\beta}{2}\right)  \cos\left(  \frac{\beta}{2}\right)  \left(
		37-6\cos^{2}\left(  \frac{\beta}{2}\right)  +130\cos^{4}\left(  \frac{\beta
		}{2}\right)  -20\cos^{6}\left(  \frac{\beta}{2}\right)  \right)  }\right)  \\
	& \\
	&  ~\tau_{0,1,2}^{0,1}=\tau_{0,2,1}^{0,1}=\tau_{1,0,2}^{1,0}=\tau
	_{2,0,1}^{1,0}=\\
	= &  ~{\frac{1}{2304\sin^{8}\left(  \frac{\beta}{2}\right)  }}\left(
	640-45{\beta}^{2}-12\left(  143+3{\beta}^{2}\right)  \cos^{2}\left(
	\frac{\beta}{2}\right)  +36\left(  93-7{\beta}^{2}\right)  \cos^{4}\left(
	\frac{\beta}{2}\right)  -16\left(  169-9{\beta}^{2}\right)  \cos^{6}\left(
	\frac{\beta}{2}\right)  \right.  \\
	&  ~{\left.  +432\cos^{8}\left(  \frac{\beta}{2}\right)  +12\beta\sin\left(
		\frac{\beta}{2}\right)  \cos\left(  \frac{\beta}{2}\right)  \left(
		1+14\cos^{2}\left(  \frac{\beta}{2}\right)  -38\cos^{4}\left(  \frac{\beta}%
		{2}\right)  -4\cos^{6}\left(  \frac{\beta}{2}\right)  \right)  \right)  ,}\\
	& \\
	&  ~\tau_{0,2,1}^{1,0}=\tau_{1,2,0}^{1,0}=\tau_{2,0,1}^{0,1}=\tau
	_{2,1,0}^{0,1}=\\
	= &  ~{\frac{1}{2304\sin^{8}\left(  \frac{\beta}{2}\right)  }}\left(
	-27{\beta}^{2}+4\left(  113-27{\beta}^{2}\right)  \cos^{2}\left(  \frac{\beta
	}{2}\right)  +12\left(  13-9{\beta}^{2}\right)  \cos^{4}\left(  \frac{\beta
}{2}\right)  -1200\cos^{6}\left(  \frac{\beta}{2}\right)  +592\cos^{8}\left(
\frac{\beta}{2}\right)  \right.  \\
&  ~{\left.  +12\beta\sin\left(  \frac{\beta}{2}\right)  \cos\left(
	\frac{\beta}{2}\right)  \left(  9+6\cos^{2}\left(  \frac{\beta}{2}\right)
	-14\cos^{4}\left(  \frac{\beta}{2}\right)  +20\cos^{6}\left(  \frac{\beta}%
	{2}\right)  \right)  \right)  ,}\\
& \\
&  ~\tau_{1,1,1}^{0,1}=\tau_{1,1,1}^{1,0}=\\
= &  ~{\frac{1}{1152\sin^{8}\left(  \frac{\beta}{2}\right)  }\left(
	-1408+261{\beta}^{2}+4\left(  1036-117{\beta}^{2}\right)  \cos^{2}\left(
	\frac{\beta}{2}\right)  -12\left(  572-75{\beta}^{2}\right)  \cos^{4}\left(
	\frac{\beta}{2}\right)  +32\left(  149-9{\beta}^{2}\right)  \cos^{6}\left(
	\frac{\beta}{2}\right)  \right.  }\\
&  ~\left.  -640\cos^{8}\left(  \frac{\beta}{2}\right)  +6\beta\sin\left(
\frac{\beta}{2}\right)  \cos\left(  \frac{\beta}{2}\right)  \left(
91-338\cos^{2}\left(  \frac{\beta}{2}\right)  +364\cos^{4}\left(  \frac{\beta
}{2}\right)  -72\cos^{6}\left(  \frac{\beta}{2}\right)  \right)  \right)  ,\\
& \\
&  ~\tau_{3,0,0}^{0,1}=\tau_{0,3,0}^{1,0}=0.
\end{align*}
\end{footnotesize}

%\vspace{-0.5cm}
\subsubsection{Variations of products of second order partial derivatives}\label{app:trigonometric_g_2}

\begin{footnotesize}%
	\begin{align*}
	& ~\tau_{0,0,3}^{0,2}=\tau_{0,0,3}^{2,0}=\tau_{0,3,0}^{0,2}=\tau_{3,0,0}%
	^{2,0}=\\
	=&~{\frac{1}{576\sin^{8}\left(  \frac{\beta}{2}\right)  }}\left(
	{-336-27{\beta}^{2}-4}\left(  319{-63{\beta}^{2}}\right)  {\cos^{2}\left(
		\frac{\beta}{2}\right)  +12}\left(  181{-21{\beta}^{2}}\right)  {\cos
		^{4}\left(  \frac{\beta}{2}\right)  -864\cos^{6}\left(  \frac{\beta}%
		{2}\right)  }\right.  \\
	& ~\left.  +{304\cos^{8}\left(  \frac{\beta}{2}\right)  +12\beta\sin\left(
		\frac{\beta}{2}\right)  \cos\left(  \frac{\beta}{2}\right)  }\left(
	{1+74\cos^{2}\left(  \frac{\beta}{2}\right)  }\right)  \right)  ,\\
	& \\
	\tau_{0,0,3}^{1,1}   =&~{\frac{1}{576\sin^{8}\left(  \frac{\beta}{2}\right)
		}}\left(  -{144-27{\beta}^{2}-4}\left(  451{-63{\beta}^{2}}\right)  {\cos
		{^{2}}\left(  \frac{\beta}{2}\right)  +84}\left(  23{-3{\beta}^{2}}\right)
	{\cos{^{4}}\left(  \frac{\beta}{2}\right)  +288\cos{^{6}}\left(  \frac{\beta
		}{2}\right)  }\right.  \\
	& ~\left.  -{272\cos{^{8}}\left(  \frac{\beta}{2}\right)  +12\beta\cos\left(
		\frac{\beta}{2}\right)  \sin\left(  \frac{\beta}{2}\right)  }\left(
	{19+62\cos{^{2}}\left(  \frac{\beta}{2}\right)  }\right)  \right)  ,\\
	& \\
	& ~\tau_{0,1,2}^{0,2}=\tau_{0,2,1}^{0,2}=\tau_{1,0,2}^{2,0}=\tau_{2,0,1}%
	^{2,0}=\\
	=&~{\frac{1}{1152\sin^{8}\left(  \frac{\beta}{2}\right)  }}\left(
	{224-9{\beta}^{2}+6}\left(  137{-39{\beta}^{2}}\right)  {\cos{^{2}}\left(
		\frac{\beta}{2}\right)  +6}\left(  43{+6{\beta}^{2}}\right)  {\cos{^{4}%
		}\left(  \frac{\beta}{2}\right)  +8}\left(  80{+9{\beta}^{2}}\right)
	{\cos{^{6}}\left(  \frac{\beta}{2}\right)  }\right.  \\
	& ~\left.  -{1944\cos{^{8}}\left(  \frac{\beta}{2}\right)  +3\beta\sin\left(
		\frac{\beta}{2}\right)  \cos\left(  \frac{\beta}{2}\right)  }\left(
	131{-242\cos{^{2}}\left(  \frac{\beta}{2}\right)  -472\cos{^{4}}\left(
		\frac{\beta}{2}\right)  -80\cos{^{6}}\left(  \frac{\beta}{2}\right)  }\right)
	\right)  ,\\
	& \\
	& ~\tau_{0,1,2}^{1,1}=\tau_{1,0,2}^{1,1}=\\
	=&~{\frac{1}{2304\sin^{8}\left(  \frac{\beta}{2}\right)  }}\left(
	{-224-45{\beta}^{2}+8}\left(  541{-45{\beta}^{2}}\right)  {\cos{^{2}}\left(
		\frac{\beta}{2}\right)  -36}\left(  50{-11{\beta}^{2}}\right)  {\cos{^{4}%
		}\left(  \frac{\beta}{2}\right)  -16}\left(  202{-9{\beta}^{2}}\right)
	{\cos{^{6}}\left(  \frac{\beta}{2}\right)  }\right.  \\
	& ~\left.  {+928\cos{^{8}}\left(  \frac{\beta}{2}\right)  -6\beta\sin\left(
		\frac{\beta}{2}\right)  \cos\left(  \frac{\beta}{2}\right)  }\left(
	193{+14\cos{^{2}}\left(  \frac{\beta}{2}\right)  +236\cos{^{4}}\left(
		\frac{\beta}{2}\right)  +40\cos{^{6}}\left(  \frac{\beta}{2}\right)  }\right)
	\right)  ,\\
	& \\
	& ~\tau_{0,1,2}^{2,0}=\tau_{1,0,2}^{0,2}=\tau_{1,2,0}^{0,2}=\tau_{2,1,0}%
	^{2,0}\\
	=&~{\frac{1}{1152\sin^{8}\left(  \frac{\beta}{2}\right)  }}\left(
	-384-54{\beta}^{2}+2\left(  1357{{-99\beta}^{2}}\right)  {{\cos{^{2}}\left(
			\frac{\beta}{2}\right)  }-6}\left(  115{-48{\beta}^{2}}\right)  {\cos{^{4}%
		}\left(  \frac{\beta}{2}\right)  -24}\left(  88{-3{\beta}^{2}}\right)
	{\cos{^{6}}\left(  \frac{\beta}{2}\right)  }\right.  \\
	& ~\left.  {+472\cos{^{8}}\left(  \frac{\beta}{2}\right)  -3\beta\sin\left(
		\frac{\beta}{2}\right)  \cos\left(  \frac{\beta}{2}\right)  }\left(
	341{-150\cos{^{2}}\left(  \frac{\beta}{2}\right)  +176\cos{^{4}}\left(
		\frac{\beta}{2}\right)  +224\cos{^{6}}\left(  \frac{\beta}{2}\right)
	}\right)  \right)  ,\\
	& \\
	& ~\tau_{0,2,1}^{1,1}=\tau_{2,0,1}^{1,1}=\\
	=&~{\frac{1}{2304\sin^{8}\left(  \frac{\beta}{2}\right)  }}\left(
	{-224-45{\beta}^{2}+}12\left(  131{-15{\beta}^{2}}\right)  {\cos{^{2}}\left(
		\frac{\beta}{2}\right)  -12}\left(  79{+15{\beta}^{2}}\right)  {\cos{^{4}%
		}\left(  \frac{\beta}{2}\right)  +2432\cos{^{6}}\left(  \frac{\beta}%
		{2}\right)  }\right.  \\
	& ~\left.  -{2832\cos{^{8}}\left(  \frac{\beta}{2}\right)  +24\beta\sin\left(
		\frac{\beta}{2}\right)  \cos\left(  \frac{\beta}{2}\right)  }\left(
	49{-53\cos{^{2}}\left(  \frac{\beta}{2}\right)  -43\cos{^{4}}\left(
		\frac{\beta}{2}\right)  -10\cos{^{6}}\left(  \frac{\beta}{2}\right)  }\right)
	\right)  ,\\
	& \\
	& ~\tau_{0,2,1}^{2,0}=\tau_{1,2,0}^{2,0}=\tau_{2,0,1}^{0,2}=\tau_{2,1,0}%
	^{0,2}=\\
	=&~{\frac{1}{576\sin^{8}\left(  \frac{\beta}{2}\right)  }\left(
		-192-27{\beta}^{2}+4\left(  143-27{\beta}^{2}\right)  \cos^{2}\left(
		\frac{\beta}{2}\right)  +12\left(  115-9{\beta}^{2}\right)  \cos^{4}\left(
		\frac{\beta}{2}\right)  -1296\cos^{6}\left(  \frac{\beta}{2}\right)  \right.
	}\\
	& ~{\left.  -464\cos^{8}\left(  \frac{\beta}{2}\right)  +12\beta\sin\left(
		\frac{\beta}{2}\right)  \cos\left(  \frac{\beta}{2}\right)  \left(
		6-38\cos^{4}\left(  \frac{\beta}{2}\right)  -28\cos^{6}\left(  \frac{\beta
		}{2}\right)  \right)  \right)  ,}\\
	& \\
	& ~\tau_{1,1,1}^{0,2}=2\tau_{1,1,1}^{1,1}=\tau_{1,1,1}^{2,0}=\\
	=&~{\frac{1}{576\sin^{8}\left(  \frac{\beta}{2}\right)  }}\left(
	1216+171{\beta}^{2}-16\left(  133-9{\beta}^{2}\right)  \cos^{2}\left(
	\frac{\beta}{2}\right)  -12\left(  556-33{\beta}^{2}\right)  \cos^{4}\left(
	\frac{\beta}{2}\right)  +16\left(  362-9{\beta}^{2}\right)  \cos^{6}\left(
	\frac{\beta}{2}\right)  \right.  \\
	& \left.  +1792\cos^{8}\left(  \frac{\beta}{2}\right)  +6\beta\sin\left(
	\frac{\beta}{2}\right)  \cos\left(  \frac{\beta}{2}\right)  \left(
	77-250\cos^{2}\left(  \frac{\beta}{2}\right)  +476\cos^{4}\left(  \frac{\beta
	}{2}\right)  +264\cos^{6}\left(  \frac{\beta}{2}\right)  \right)  \right)  ,\\
	& \\
	& ~\tau_{1,2,0}^{1,1}=\tau_{2,1,0}^{1,1}=\\
	=&~{\frac{1}{2304\sin^{8}\left(  \frac{\beta}{2}\right)  }}\left(
	{-480-27{\beta}^{2}-}4\left(  41{+27{\beta}^{2}}\right)  {\cos{^{2}}\left(
		\frac{\beta}{2}\right)  +}12\left(  463{-9{\beta}^{2}}\right)  {\cos{^{4}%
		}\left(  \frac{\beta}{2}\right)  -5568\cos{^{6}}\left(  \frac{\beta}%
		{2}\right)  }\right.  \\
	& ~\left.  {+{656\cos{^{8}}\left(  \frac{\beta}{2}\right)  }-24\beta\sin\left(
		\frac{\beta}{2}\right)  \cos\left(  \frac{\beta}{2}\right)  }\left(
	{39-57\cos{^{2}}\left(  \frac{\beta}{2}\right)  +17\cos{^{4}}\left(
		\frac{\beta}{2}\right)  +46\cos{^{6}}\left(  \frac{\beta}{2}\right)
	}\right)  \right)  ,\\
	& \\
	\tau_{0,3,0}^{1,1}   =&~\tau_{0,3,0}^{2,0}=\tau_{3,0,0}^{0,2}=\tau
	_{3,0,0}^{1,1}=0.
	\end{align*}
\end{footnotesize}%

\subsection{Double integrals of constrained trivariate algebraic-trigonometric expressions}
\label{app:algebraic_trigonometric}

Using the constants
\begin{footnotesize}
\begin{align*}
c_{1} &  =\frac{1}{\beta-\sin\left(  \beta\right)  },\\
c_{2} &  =\frac{\sin\left(  \beta\right)  }{\left(  2\sin\left(  \beta\right)
	-\beta-\beta\cos\left(  \beta\right)  \right)  \left(  \beta-\sin\left(
	\beta\right)  \right)  },\\
c_{3} &  =\frac{4\left(  3\beta+4\sin\left(  \beta\right)  -\beta\cos\left(
	\beta\right)  \right)  \cos\left(  \frac{\beta}{2}\right)  }{\left(
	2\sin\left(  \beta\right)  -\beta-\beta\cos\left(  \beta\right)  \right)
	\left(  \beta-\sin\left(  \beta\right)  \right)  },\\
c_{4} &  =\frac{4\sin\left(  \beta\right)  \cos\left(  \frac{\beta}{2}\right)
}{\left(  2\sin\left(  \beta\right)  -\beta-\beta\cos\left(  \beta\right)
\right)  \left(  \beta-\sin\left(  \beta\right)  \right)  },~\beta \in \left(0,2\pi\right),%
\end{align*}
\end{footnotesize}the values listed in Sections \ref{app:algebraic_trigonometric_g_1}--\ref{app:algebraic_trigonometric_g_2} are related to the constrained trivariate algebraic-trigonometric basis functions detailed in Example \ref{exmp:trivariate_algebraic_trigonometric}.

\subsubsection{Variations of products of first order partial derivatives}\label{app:algebraic_trigonometric_g_1}

\begin{footnotesize}
\begin{align*}
&  ~\tau_{0,0,3}^{0,1}=\tau_{0,0,3}^{1,0}=\tau_{0,3,0}^{0,1}=\tau_{3,0,0}^{1,0}=\\
= &  ~\frac{c_{1}}{48}\left(  -3c_{{3}}\left(  \beta\sin\left(  \frac{\beta
}{2}\right)  +2\cos\left(  \frac{\beta}{2}\right)  -{\beta}^{2}\cos\left(
\frac{\beta}{2}\right)  -2\cos^{3}\left(  \frac{\beta}{2}\right)  \right)
\right.  \\
&  ~\left.  +c_{{4}}\left(  6\beta\sin\left(  \frac{\beta}{2}\right)  +{\beta
}^{3}\sin\left(  \frac{\beta}{2}\right)  +12\cos\left(  \frac{\beta}%
{2}\right)  -9{\beta}^{2}\cos\left(  \frac{\beta}{2}\right)  -12\cos
^{3}\left(  \frac{\beta}{2}\right)  +6\beta\sin\left(  \frac{\beta}{2}\right)
\cos^{2}\left(  \frac{\beta}{2}\right)  \right)  \right),
\\
\\
&  ~\tau_{0,1,2}^{0,1}=\tau_{0,2,1}^{0,1}=\tau_{1,0,2}^{1,0}=\tau
_{2,0,1}^{1,0}=\\
= &  ~{\frac{c_{{2}}}{96}}\left(  c_{{3}}\left(  \beta\sin\left(  \frac{\beta}{2}\right)  \left(  45-4{\beta
}^{2}-12\cos^{2}\left(  \frac{\beta}{2}\right)  \right)  +3\cos\left(
\frac{\beta}{2}\right)  \left(  14-9{\beta}^{2}-14\cos^{2}\left(  \frac{\beta
}{2}\right)  \right)  \right)  \right.  \\
&  ~\left.  -c_{{4}}\left(  \beta\sin\left(  \frac{\beta}{2}\right)  \left(  39+18\cos
^{2}\left(  \frac{\beta}{2}\right)  -{\beta}^{2}\right)  +\cos\left(
\frac{\beta}{2}\right)  \left(  78-36{\beta}^{2}-{\beta}^{4}-6\left(  2{\beta
}^{2}+13\right)  \cos^{2}\left(  \frac{\beta}{2}\right)  \right)  \right)  \right),
\\
\\
&  ~\tau_{0,1,2}^{1,0}=\tau_{1,0,2}^{0,1}=\tau_{1,2,0}^{0,1}=\tau
_{2,1,0}^{1,0}=\\
= &  ~\frac{\beta c_{{2}}\sin\left(  \frac{\beta}{2}\right)  }{16}\left(
c_{{3}}\left(  8-{\beta}^{2}-8\cos^{2}\left(  \frac{\beta}{2}\right)
-2\beta\sin\left(  \frac{\beta}{2}\right)  \cos\left(  \frac{\beta}{2}\right)
\right) +\beta c_{{4}}\left(  {\beta}+2{\beta}\cos^{2}\left(  \frac{\beta}%
{2}\right)  -6\sin\left(  \frac{\beta}{2}\right)  \cos\left(
\frac{\beta}{2}\right)  \right)  \right),
\\
\\
&  ~\tau_{0,2,1}^{1,0}=\tau_{1,2,0}^{1,0}=\tau_{2,0,1}^{0,1}=\tau
_{2,1,0}^{0,1}=\\
= &  ~{\frac{c_{{2}}}{96}}\left(  c_{{3}}\left(  \beta\left(  15-2{\beta}%
^{2}-84\cos^{2}\left(  \frac{\beta}{2}\right)  \right)  \sin\left(
\frac{\beta}{2}\right)  \right.   +3\left(  58-7{\beta}^{2}-\left(  58-4{\beta}^{2}\right)  \cos
^{2}\left(  \frac{\beta}{2}\right)  \right)  \cos\left(  \frac{\beta}%
{2}\right)  \right)  \\
&  ~+c_{{4}}\left(  
\beta\left(  15+{\beta}^{2}\right)  
\sin\left(
\frac{\beta}{2}
\right)  -6\beta\left(  25-2{\beta}^{2}\right)  \sin\left(
\frac{\beta}{2}\right)  \cos^{2}\left(  \frac{\beta}{2}\right)      +\left(  30-12{\beta}^{2}-{\beta}^{4}-\left(  30-72{\beta
}^{2}\right)  \cos^{2}\left(  \frac{\beta}{2}\right)  \right)  \cos\left(
\frac{\beta}{2}\right)  \right),  
%\right),
\\
\\
%\tau_{0,3,0}^{0,1}=&~\frac{c_{{1}}}{32}\left(  c_{{3}}\left(  -2\beta
%\sin\left(  \frac{\beta}{2}\right)  -\left(  1-2{\beta}^{2}\right)
%\cos\left(  \frac{\beta}{2}\right)  +\cos\left(  \frac{3\beta}{2}\right)
%\right)  \right.  \\
%&  ~\left.  +c_{{4}}\left(  \beta\left(  {5}\sin\left(  \frac{\beta}%
%{2}\right)  +\frac{2}{3}{\beta}^{2}\sin\left(  \frac{\beta}{2}\right)
%+\sin\left(  \frac{3\beta}{2}\right)  \right)  +2\left(  \left(  1-3{\beta
%}^{2}\right)  \cos\left(  \frac{\beta}{2}\right)  -\cos\left(  \frac{3\beta
%}{2}\right)  \right)  \right)  \right),
%\\
%\\
&~\tau_{0,3,0}^{1,0}=\tau_{3,0,0}^{0,1}=0,
\\
\\
&  ~\tau_{1,1,1}^{0,1}=\tau_{1,1,1}^{1,0}=\\
= &  ~\frac{1}{96}\left(  -3{c_{{3}}^{2}}\left(  4{-\beta}^{2}-4\cos\left(
\beta\right)  -\beta\sin\left(  \beta\right)  \right)  -6c_{{3}}c_{{4}}\left(
2{\beta}^{2}-3\beta\sin\left(  \beta\right)  +{\beta}^{2}\cos\left(
\beta\right)  \right)    +\beta{c_{{4}}^{2}}\left(  {\beta}^{3}-3{\beta}^{2}\sin\left(
\beta\right)  -6{\beta}\cos\left(  \beta\right)  +6\sin\left(  \beta\right)
\right)  \right).
\end{align*}
\end{footnotesize}

\subsubsection{Variations of products of second order partial derivatives}\label{app:algebraic_trigonometric_g_2}

\begin{footnotesize}
\begin{align*}
	&  ~\tau_{0,0,3}^{0,2}=\tau_{0,0,3}^{2,0}=\tau_{0,3,0}^{0,2}=\tau
	_{3,0,0}^{2,0}=\\
	=  &  ~-\frac{c_{{1}}}{48}\left(  3c_{{3}}\left(  3\beta\sin\left(
	\frac{\beta}{2}\right)  -2\cos\left(  \frac{\beta}{2}\right)  -{\beta}^{2}%
	\cos\left(  \frac{\beta}{2}\right)  +2\cos^{3}\left(  \frac{\beta}{2}\right)
	\right)  \right. \\
	&  ~\left.  +c_{{4}}\left(  -\beta\left(  36+{\beta}^{2}\right)  \sin\left(
	\frac{\beta}{2}\right)  +3\left(  8+3{\beta}^{2}-8\cos^{2}\left(  \frac{\beta
	}{2}\right)  +2\beta\sin\left(  \frac{\beta}{2}\right)  \cos\left(
	\frac{\beta}{2}\right)  \right)  \cos\left(  \frac{\beta}{2}\right)  \right)
	\right)  ,
\\
\\
	\tau_{0,0,3}^{1,1}=  &  ~\frac{c_{{1}}}{48}\left(  -3c_{{3}}\left(  \beta
	\sin\left(  \frac{\beta}{2}\right)  +\left(  2-{\beta}^{2}-2\cos^{2}\left(
	\frac{\beta}{2}\right)  \right)  \cos\left(  \frac{\beta}{2}\right)  \right)
	\right. \\
	&  ~\left.  +c_{{4}}\left(  \beta\left(  6+{\beta}^{2}\right)  \sin\left(
	\frac{\beta}{2}\right)  +3\left(  4-3{\beta}^{2}-4\cos^{2}\left(  \frac{\beta
	}{2}\right)  +2\beta\sin\left(  \frac{\beta}{2}\right)  \cos\left(
	\frac{\beta}{2}\right)  \right)  \cos\left(  \frac{\beta}{2}\right)  \right)
	\right)  ,
\\
\\
	&  ~\tau_{0,1,2}^{0,2}=\tau_{0,2,1}^{0,2}=\tau_{1,0,2}^{2,0}=\tau
	_{2,0,1}^{2,0}=\\
	=  &  ~{\frac{c_{{2}}}{96}}\left(  c_{{3}}\left(  \beta\left(  51-4{\beta}%
	^{2}+12\cos^{2}\left(  \frac{\beta}{2}\right)  \right)  \sin\left(
	\frac{\beta}{2}\right)  -7\left(  6+3{\beta}^{2}-6\cos^{2}\left(  \frac{\beta
	}{2}\right)  \right)  \cos\left(  \frac{\beta}{2}\right)  \right)  \right. \\
	&  ~\left.  -c_{{4}}\left(  3\beta\left(  69-5{\beta}^{2}+2\cos^{2}\left(
	\frac{\beta}{2}\right)  \right)  \sin\left(  \frac{\beta}{2}\right)  -\left(
	162+78{\beta}^{2}+{\beta}^{4}-6\left(  27+2{\beta}^{2}\right)  \cos^{2}\left(
	\frac{\beta}{2}\right)  \right)  \cos\left(  \frac{\beta}{2}\right)  \right)
	\right)  ,
\\
\\
	&  ~\tau_{0,1,2}^{1,1}=\tau_{1,0,2}^{1,1}=\\
	=  &  ~\frac{c_{{2}}}{96}\left(  c_{{3}}\left(  \beta\left(  15-4{\beta}%
	^{2}\right)  \sin\left(  \frac{\beta}{2}\right)  -3\left(  6-{\beta}%
	^{2}+2\left(  2\beta\sin\left(  \frac{\beta}{2}\right)  -3\cos\left(
	\frac{\beta}{2}\right)  \right)  \cos\left(  \frac{\beta}{2}\right)  \right)
	\cos\left(  \frac{\beta}{2}\right)  \right)  \right. \\
	&  ~\left.  -c_{{4}}\left(  \beta\left(  63-19{\beta}^{2}\right)  \sin\left(
	\frac{\beta}{2}\right)  -\left(  66-18{\beta}^{2}+{\beta}^{4}+6\left(
	7\beta\sin\left(  \frac{\beta}{2}\right)  -\left(  11-2{\beta}^{2}\right)
	\cos\left(  \frac{\beta}{2}\right)  \right)  \cos\left(  \frac{\beta}%
	{2}\right)  \right)  \cos\left(  \frac{\beta}{2}\right)  \right)  \right)  ,
\\
\\
	&  ~\tau_{0,1,2}^{2,0}=\tau_{1,0,2}^{0,2}=\tau_{1,2,0}^{0,2}=\tau
	_{2,1,0}^{2,0}=\\
	=  &  ~\frac{c_{2}}{16}\left(  c_{{3}}\left(  -{\beta}^{3}\sin\left(
	\frac{\beta}{2}\right)  -\left(  16-4{\beta}^{2}-2\left(  2\beta\sin\left(
	\frac{\beta}{2}\right)  +\left(  8-{\beta}^{2}\right)  \cos\left(  \frac
	{\beta}{2}\right)  \right)  \cos\left(  \frac{\beta}{2}\right)  \right)
	\cos\left(  \frac{\beta}{2}\right)  \right)  \right. \\
	&  ~\left.  +c_{{4}}\left(  5{\beta}^{3}\sin\left(  \frac{\beta}{2}\right)
	+2\beta\left(  -4{\beta}+\left(  \left(  6-{\beta}^{2}\right)  \sin\left(
	\frac{\beta}{2}\right)  +{\beta}\cos\left(  \frac{\beta}{2}\right)  \right)
	\cos\left(  \frac{\beta}{2}\right)  \right)  \cos\left(  \frac{\beta}%
	{2}\right)  \right)  \right)  ,
\\
\\
	&  ~\tau_{0,2,1}^{1,1}=\tau_{2,0,1}^{1,1}=\\
	= &  ~\frac{c_{2}}{48}\left(  3c_{{3}}\left(  \beta\sin\left(  \frac{\beta}%
	{2}\right)  +\left(  2-{\beta}^{2}-2\cos^{2}\left(  \frac{\beta}{2}\right)
	\right)  \cos\left(  \frac{\beta}{2}\right)  \right)  \right.  \\
	&  ~\left.  -c_{{4}}\left(  \beta\left(  6+{\beta}^{2}\right)  \sin\left(
	\frac{\beta}{2}\right)  +3\left(  4-3{\beta}^{2}+2\left(  \beta\sin\left(
	\frac{\beta}{2}\right)  -2\cos\left(  \frac{\beta}{2}\right)  \right)
	\cos\left(  \frac{\beta}{2}\right)  \right)  \cos\left(  \frac{\beta}%
	{2}\right)  \right)  \right)  ,
\\
\\
	&  ~\tau_{0,2,1}^{2,0}=\tau_{1,2,0}^{2,0}=\tau_{2,0,1}^{0,2}=\tau
	_{2,1,0}^{0,2}=\\
	= &  ~{\frac{c_{{2}}}{96}}\left(  -c_{{3}}\left(  \beta\left(  15+2{\beta}%
	^{2}\right)  \sin\left(  \frac{\beta}{2}\right)  -3\left(  38-{\beta}%
	^{2}-2\left(  2\beta\sin\left(  \frac{\beta}{2}\right)  +\left(  19+2{\beta
	}^{2}\right)  \cos\left(  \frac{\beta}{2}\right)  \right)  \cos\left(
	\frac{\beta}{2}\right)  \right)  \cos\left(  \frac{\beta}{2}\right)  \right)
	\right.  \\
	&  ~\left.  +c_{{4}}\left(  \beta\left(  63+11{\beta}^{2}\right)  \sin\left(
	\frac{\beta}{2}\right)  -\left(  66+6{\beta}^{2}+{\beta}^{4}+6\left(
	\beta\left(  19+2{\beta}^{2}\right)  \sin\left(  \frac{\beta}{2}\right)
	-\left(  11+8{\beta}^{2}\right)  \cos\left(  \frac{\beta}{2}\right)  \right)
	\cos\left(  \frac{\beta}{2}\right)  \right)  \cos\left(  \frac{\beta}%
	{2}\right)  \right)  \right)  ,
\\
\\
	&  ~\tau_{1,1,1}^{0,2}=2\tau_{1,1,1}^{1,1}=\tau_{1,1,1}^{2,0}=\\
	= &  ~{\frac{\beta}{96}}\left(  3{c_{{3}}^{2}}\left(  \beta-\sin\left(
	\beta\right)  \right)  -6c_{{3}}c_{{4}}\left(  4\beta-3\sin\left(
	\beta\right)  -\beta\cos\left(  \beta\right)  \right)  +{c_{{4}}^{2}}\left(  48\beta+{\beta}^{3}-30\sin\left(
	\beta\right)  +3{\beta}^{2}\sin\left(  \beta\right)  -18\beta\cos\left(
	\beta\right)  \right)  \right)  ,
\\
\\
	& ~\tau_{1,2,0}^{1,1}=\tau_{2,1,0}^{1,1}=\\
	=  & ~{\frac{c_{{2}}}{96}}\left(  -c_{{3}}\left(  \beta\left(  15+2{\beta}%
	^{2}\right)  \sin\left(  \frac{\beta}{2}\right)  -\left(  18+9{\beta}%
	^{2}+6\left(  2\beta\sin\left(  \frac{\beta}{2}\right)  -\left(  2{\beta}%
	^{2}+3\right)  \cos\left(  \frac{\beta}{2}\right)  \right)  \cos\left(
	\frac{\beta}{2}\right)  \right)  \cos\left(  \frac{\beta}{2}\right)  \right)
	\right.  \\
	& ~\left.  +c_{{4}}\left(  \beta\left(  63+11{\beta}^{2}\right)  \sin\left(
	\frac{\beta}{2}\right)  -\left(  66+18{\beta}^{2}+{\beta}^{4}+\left(
	6\beta\left(  7+2{\beta}^{2}\right)  \sin\left(  \frac{\beta}{2}\right)
	-6\left(  11+4{\beta}^{2}\right)  \cos\left(  \frac{\beta}{2}\right)  \right)
	\cos\left(  \frac{\beta}{2}\right)  \right)  \cos\left(  \frac{\beta}%
	{2}\right)  \right)  \right)  ,
\\
\\
	& ~\tau_{0,3,0}^{1,1}=\tau_{0,3,0}^{2,0}=\tau_{3,0,0}^{0,2}=\tau_{3,0,0}%
	^{1,1}=0.
\end{align*}

\end{footnotesize}

\end{document}